\documentclass[10pt,psamsfonts]{amsart}
\usepackage{amsmath}
\usepackage{amsthm}
\usepackage{amssymb}
\usepackage{amscd}
\usepackage{amsfonts}
\usepackage{amsbsy}
\usepackage{graphicx}
\usepackage[dvips]{psfrag}
\usepackage{array}
\usepackage{color}
\usepackage{epsfig}
\usepackage{url}
\usepackage{overpic}
\usepackage{tikz-cd}
\usepackage{enumitem}
\usepackage{hyperref}
\usepackage{soul}
\usepackage{float}
\usepackage[normalem]{ulem}
\usepackage{nccmath}
\usepackage[foot]{amsaddr}
\usepackage{textcomp}

\newcolumntype{L}{>{\displaystyle}l}
\newcolumntype{C}{>{\displaystyle}c}
\newcolumntype{R}{>{\displaystyle}r}

\newcommand{\R}{\ensuremath{\mathbb{R}}}

\newcommand{\Z}{\ensuremath{\mathbb{Z}}}

\newcommand{\CO}{\ensuremath{\mathcal{O}}}
\newcommand{\ov}{\overline}
\newcommand{\la}{\lambda}

\newcommand{\bx}{{\bf x}}

\newcommand{\bz}{{\bf z}}
\newcommand{\bg}{{\bf g}}
\newcommand{\bu}{{\bf u}}
\newcommand{\f}{{\bf f}}

\newcommand{\s}{\ensuremath{\mathbb{S}}}

\newcommand{\x}{\mathbf{x}}

\def\p{\partial}
\def\e{\varepsilon}

\newtheorem {theorem} {Theorem}
\newtheorem {definition} {Definition}
\newtheorem {proposition}{Proposition}
\newtheorem {corollary}{Corollary}
\newtheorem {lemma}{Lemma}
\newtheorem {example} {Example}
\newtheorem {remark}{Remark}

\newtheorem {mtheorem} {Theorem}

\usepackage[top=3cm,bottom=2cm,left=1.5cm,right=1.5cm]{geometry}

\newcommand\blfootnote[1]{%
	\begingroup
	\renewcommand\thefootnote{}\footnote{#1}%
	\addtocounter{footnote}{-1}%
	\endgroup
}

\usepackage{mathpazo}

\begin{document}
\renewcommand{\arraystretch}{1.5}

\title[Invariant tori via higher order averaging method]
{Invariant tori via higher order averaging method:\\ existence, regularity, convergence, stability, and dynamics}

\author[D.D. Novaes and P.C.C.R. Pereira]
{Douglas D. Novaes$^1$  and Pedro C.C.R. Pereira$^2$}

\address{Departamento de Matem\'{a}tica - Instituto de Matemática, Estatística e Computação Científica (IMECC) - Universidade Estadual de Campinas (UNICAMP), Rua S\'{e}rgio Buarque de Holanda, 651, Cidade
Universit\'{a}ria Zeferino Vaz, 13083--859, Campinas, SP, Brazil}
\email{ddnovaes@unicamp.br$^1$}
\email{pedro.pereira@ime.unicamp.br$^2$}



\begin{abstract}
Important information about the dynamical structure of a differential system can be revealed  by looking into its invariant compact manifolds, such as equilibria, periodic orbits, and invariant tori. This knowledge is significantly increased if asymptotic properties of the trajectories nearby such invariant manifolds can be determined. In this paper, we present a result providing sufficient conditions for the existence of invariant tori in perturbative differential systems. The regularity, convergence, and stability of such tori as well as the dynamics defined on them are also investigated. The conditions are given in terms of their so-called higher order averaged equations. This result is an extension to a wider class of differential systems of theorems due to Krylov, Bogoliubov, Mitropolsky, and Hale.
\end{abstract}
\blfootnote{
	DDN is partially supported by S\~{a}o Paulo Research Foundation (FAPESP) grants 2022/09633-5, 2019/10269-3, and 2018/13481-0, and by Conselho Nacional de Desenvolvimento Cient\'{i}fico e Tecnol\'{o}gico (CNPq) grants 438975/2018-9 and 309110/2021-1. PCCRP is supported by S\~{a}o Paulo Research Foundation (FAPESP) grant 2020/14232-4. \\
\rule{60pt}{0.4pt}}
\maketitle

\section{Introduction and statement of the main result} \label{sec:intro}

The \textit{averaging method} has been employed by Krylov, Bogoliubov, and Mitropolski \cite{BM, BK} to study the existence of invariant tori in the extended phase space of $T$-periodic non-autonomous perturbative differential equations of the kind $\dot\bx=\e F_1(t,\bx).$ Those results were generalized by Hale in \cite{haleinvariant} and \cite{haleoscillations} and assert that the existence of invariant tori is associated to the existence of limit cycles of the so-called averaged equation,
\[
\dot \bx=\dfrac{1}{T}\int_0^T F_1(t,\bx)dt.
\]

In this paper, we are concerned with a wider class of  $T$-periodic non-autonomous perturbative differential equations of the following kind:
\begin{align} \label{eq:e1}
\dot \bx = \sum_{i=1}^N \varepsilon^i F_i(t, \bx) + \varepsilon^{N+1} {\Tilde{F}}(t, \bx,\varepsilon),\quad (t,\bx,\e)\in \R \times D \times [0,\varepsilon_0],
\end{align}
where $D$ is an open bounded subset of $\R^n,$  $\varepsilon_0>0,$ and the functions $F_i:\R \times D \rightarrow \R^n,$ $i\in\{1,\ldots,N\}$, and  $\Tilde{F}:\R \times D \times [0,\varepsilon_0]\to\R^n$ are of class $C^r$, $r\geq 2$, and $T$-periodic in the variable $t$. Our goal is to extend the mentioned results of Krylov, Bogoliubov, Mitropolsky, and Hale concerning the existence of invariant tori to the differential equation \eqref{eq:e1}. More specifically, we aim to provide sufficient conditions for the existence of invariant tori in the extended phase space of \eqref{eq:e1} which, due to the periodicity in the variable $t$, can be seen as a vector field defined on a cylinder:
\begin{equation} \label{eq:systemextended}
			\left\{ \begin{array}{@{}l@{}}
			\tau'=1,\\
				\bx' = \displaystyle \sum_{i=1}^N \varepsilon^i F_i(\tau, \bx) + \varepsilon^{N+1} {\Tilde{F}}(\tau, \bx,\varepsilon),
			\end{array} \right.\quad (\tau,\x)\in\s^1\times D,
		\end{equation}
		where $\s^1=\R/(T\Z)$.
In addition, results concerning the regularity, convergence, and stability of such tori as well as information about the dynamics defined on them will also be presented.

\subsection{Introduction to the averaging theory}\label{sec:avm} Some notions from the averaging theory will appear in the statement of our main result, Theorem \ref{maintheorem}. Thus, in order to state it, we must  provide a brief introduction to the averaging method, with special attention to the concept of higher order averaged functions.

The \textit{averaging method} or \textit{averaging theory} stemmed from the works of Clairaut, Lagrange, and Laplace regarding perturbartions of differential equations (see \cite[Appendix A]{SVM}), even though its formalization was only established much later, by Fatou, Krylov, Bogoliubov, and Mitropolsky (see \cite{BM,Bo,fatou,BK}). It is particularly useful in the study of nonlinear oscillating systems which are affected by small perturbations, by providing asymptotic estimates for solutions of non-autonomous differential equations given in the standard form \eqref{eq:e1}.

The estimates provided by the averaging method depend on the \textit{averaged functions}, $\bg_i:D\rightarrow\R^n$ for $i\in\{1,\ldots,N\},$ which appear as solutions of homological equations when transforming system \eqref{eq:e1} according to the following result.
\begin{theorem}[{\cite[Lemma 2.9.1]{SVM}}]\label{thm:av1}
	There exists a $T$-periodic near-identity transformation of class $C^r$
	\begin{equation}\label{avtrans}
	\bx=U(t,\bz ,\e)=\bz+\sum_{i=1}^N \e^i\, \bu_i(t,\bz ),
	\end{equation}
	satisfying $U(0,\bz,\e)=\bz$, such that the differential equation \eqref{eq:e1} is transformed into
	\begin{equation*}
	\dot \bz=\sum_{i=1}^N\e^i\bg_i(\bz)+\e^{N+1} r_N(t,\bz,\e).
	\end{equation*}
\end{theorem}

The condition $U(0,\bz,\e)=\bz$, called \textit{stroboscopic condition}, ensures that the functions $\bg_i$ are uniquely determined. In that case, $\bg_i$ is named the {\it averaged function of order $i$}. One can easily verify that $\bg_1$ is, indeed, the time-average of $F_1(t,\bx)$, that is,
\begin{equation} \label{eq:formulag1}
    \bg_1(\bz) = \frac{1}{T} \int_0^T F_1(s,\bz) \, ds.
\end{equation}

In general terms, the averaging theory guarantees that, for time $\CO(1/\e)$ and $\e$ small, any solution of (\ref{eq:e1}) remains $\e^N$-close to the solution of the \textit{truncated averaged equation}
\[
\dot \bz = \sum_{i=1}^N \varepsilon^i \bg_i(\bz),
\]
with the same initial conditions.

In addition to the aforesaid quantitative estimates, the averaging theory has found great success when applied to investigate invariant manifolds; for instance, to guarantee the existence of invariant tori, as mentioned in the introduction of this work, \cite{BM,cannov20,haleinvariant,haleoscillations}. It has also been successfully applied to the study of simpler compact invariant manifolds, such as periodic solutions (see, for example, \cite{llibreaveraging,haleordinary,LliNovTei2014,verhulst,Novaes21a,NovaesSilva}). 

Recently, the paper \cite{Novaes21b} provided a general recursive formula for the higher order averaged functions in terms of Melnikov functions.
Accordingly, define the {\it Melnikov function of order $i$}, $\f_i,$ for $i\in\{0,\ldots,N\},$ by
\begin{equation}\label{avfunc}
\f_0(\bz)=0\quad\text{and}\quad \f_i(\bz)=\dfrac{y_i(T,\bz)}{i!},
\end{equation}
where
\begin{equation}\label{yi}
\begin{aligned}
y_1(t,\bz)=& \int_0^tF_1(s,\bz)\,ds\,\, \text{ and }\vspace{0.3cm}\\
y_i(t,\bz)=& \int_0^t\bigg(i!F_i(s,\bz)+\sum_{j=1}^{i-1}\sum_{m=1}^j\dfrac{i!}{j!}\p_{\bx}^m F_{i-j} (s,\bz)B_{j,m}\big(y_1,\ldots,y_{j-m+1}\big)(s,\bz)\bigg)ds,
\end{aligned}
\end{equation}
for $i\in\{2,\ldots,N\}.$ In the formulae above, for  $p$ and $q$ positive integers, $B_{p,q}$ denotes the  {\it partial Bell polynomials} (see, for instance, \cite{comtet}). Roughly speaking, the Melnikov functions determine the $N$-jet in $\e$ of the time-$T$-map of \eqref{eq:e1}, that is, $\varphi(T,\bz,\e)=\bz+\sum_{i=1}^N \e^i \f_i(\bz)+\CO(\e^{N+1})$, where $\varphi(t,\bz,\e)$ corresponds to the solution of \eqref{eq:e1} with initial  condition $\varphi(0,\bz,\e)=\bz$ (see \cite{LliNovTei2014,N17}).

In particular, (\ref{eq:formulag1}) and (\ref{yi}) ensure that $\f_1(\bz)=T\bg_1(\bz).$ The next result states that the same holds for higher order averaged functions provided that some conditions are satisfied.

\begin{proposition}[{\cite[Corollary A]{Novaes21b}}]\label{prop1}
Let $\ell\in\{2,\ldots,N\}$. If either $\f_1=\cdots=\f_{\ell-1}=0$ or $\bg_1=\cdots=\bg_{\ell-1}=0,$ then $\f_i=T\,\bg_i$ for $i\in\{1,\ldots,\ell\}.$
\end{proposition}

The relationship established in Proposition \ref{prop1} allows us to directly calculate the first non-vanishing averaged function, thus motivating the main result of this work, i.e., an extension of the results of Krylov, Bogoliubov, Mitropolsky, and Hale to higher order averaged functions.

\subsection{Statement of the main theorem}
The existence of invariant tori in a differential system, as in the case of existence of equilibria and periodic orbits, reveals important information about the dynamical structure of the differential system. This knowledge is significantly increased if asymptotic properties of the trajectories nearby such invariant tori can be determined. Thus, before introducing our main result, we must set forth the following definition regarding asymptotic stability of invariant manifolds.
\begin{definition}
	Let \begin{align}\label{systemdefinitions}
				\dot \bx = F(\bx)
		\end{align} 
	be an autonomous differential system in $\mathbb{R}^n$ and let $\x(t,\x_0)$ be the solution of (\ref{systemdefinitions}) satisfying $\x(0,\bx_0)=\bx_0$. Let also $M$ be and $m$-dimensional invariant manifold of system (\ref{systemdefinitions})and  $V$ be a neighborhood of the manifold $M$. 
	\begin{enumerate}[label=\alph*)]
		\item The local stable set of $M$ with respect to $V$ is $$\mathcal{S}^V_M:=\{\bx_0 \in V: \bx(t,\bx_0) \in V \; \text{for all} \; t>0 \; \text{and} \; \lim_{t \to \infty} d(\bx(t,\bx_0),M)) = 0 \}.$$
		\item The local unstable set of $M$ with respect to $V$ is $$\mathcal{U}^V_M:=\{\bx_0 \in V: \bx(t,\bx_0) \in V \; \text{for all} \; t<0 \; \text{and} \; \lim_{t \to -\infty} d(\bx(t,\bx_0),M)) = 0 \}.$$
	\end{enumerate}
\end{definition}

Now, we are ready to provide our main result.

\begin{mtheorem} \label{maintheorem}
		Consider the $C^r$, $r\geq2$, differential equation (\ref{eq:e1}) and its extension \eqref{eq:systemextended}.	Suppose that, for some $\ell\in\{1,\ldots,\min(N,r-2)\},$ $\f_0=\ldots=\f_{\ell-1}=0,$  $\f_{\ell}\neq0$.
Assume that the guiding system $\dot \bz = \bg_\ell(\bz)$ has an $\omega$-periodic hyperbolic limit cycle $\varphi(t)$. Then, there exists $\ov\varepsilon>0$ such that, for each $\varepsilon \in (0,\ov\varepsilon]$, the following statements hold:
	\begin{enumerate}[label=\alph*)]
		\item {\bf Existence:} The differential system \eqref{eq:systemextended} has an invariant torus $M_\varepsilon$. In addition, there exists a neighborhood $V\subset D$ of $\Gamma:=\{\varphi(t):\,t\in\R\}$ such that any invariant compact manifold of \eqref{eq:systemextended} contained in $\s^1\times V$ must be contained in $M_{\e}$. In particular, $M_{\e}$ is the unique invariant torus in $\s^1\times V$.
		\item {\bf Regularity:} The invariant torus $M_\varepsilon$ is of class $C^{r-\ell}$. Furthermore, there is a $C^0$-continuous family of $C^{r-\ell}$ functions $\mathcal{F}_{\e}:\R^2\to\R^{n}$, $T-$periodic in the first coordinate and $\omega$-periodic in the second coordinate, such that $M_\varepsilon=\{(\tau,\mathcal{F}_{\e}(\tau,\theta)):\,(\tau,\theta)\in\s^1\times\R\}.$
		\item {\bf Convergence:} There is a continuous function $\delta: [0,\ov \varepsilon] \to \mathbb{R}_+$ satisfying $ \delta(0) = 0$ such that $\|\mathcal{F}_\varepsilon(\tau,\theta) - U(\tau,\varphi(\theta),\e)\|<\delta(\varepsilon)$ for every $(\tau,\theta)\in\R^2$, where $U$ is the transformation given by Theorem \ref{thm:av1}. In particular, $M_{\e}$ converges to $\s^1\times\Gamma$ in the Hausdorff distance as $\e \to 0$.
		\item {\bf Stability:} Let $k \leq n-1$ be the number of characteristic multipliers of $\Gamma$
		with modulus less than $1$. Then, there are neighborhoods $W_s$, $V_s$, $W_u$, and $V_u$ of $M_\varepsilon$ such that
		\begin{enumerate}[label=d.\arabic*)]
			\item $\mathcal{S}^{V_s}_{M_\varepsilon} \cap W_s$ is a $(k+2)$-dimensional manifold embedded in $\mathbb{R}^{n+1}$;
			\item $\mathcal{U}^{V_u}_{M_\varepsilon} \cap W_u$  is a $(n-k+1)$-dimensional manifold embedded in $\mathbb{R}^{n+1}$.
		\end{enumerate}
\item {\bf Dynamics:} The flow of \eqref{eq:systemextended} restricted to $M_{\e}$ defines a first return map $p_\varepsilon:S_{\e}\to S_{\e}$ where, for  $\Sigma=\{(0,\x):\, \x\in D\}$ a transversal section of \eqref{eq:systemextended}, $S_{\e}:=\Sigma\cap M_{\e}$ is $C^{r-\ell}$ diffeomorphic to the circle $\s^1$. Moreover, $p_\varepsilon$ is of class $C^{r-\ell}$; its rotation number $\rho(\e)$  is a continuous function on $\e\in[0,\ov \e]$ satisfying $\rho(\e)=\e^{\ell} T/\omega+\CO(\e^{\ell+1})$; and, finally, if $r-\ell\geq4$, then $\rho$ maps zero Lebesgue measure sets to zero Lebesgue measure sets, and there exists a positive Lebesgue measure set $E\subset [0,\ov \e]$ such that, for every $\e\in E$, $\rho(\e)$ is irrational and $p_{\varepsilon}$ is $C^{r-\ell-3}$ conjugated to an irrational rotation.
\end{enumerate}
\end{mtheorem}

Theorem \ref{maintheorem} is proved in Section \ref{sec:proof} after establishing some preliminary results in Subsection \ref{sec:fl}.

\begin{remark}
Since  $p_\varepsilon$ is at least of class $C^2$ (because $r-\ell\geq 2$), if $\rho(\varepsilon)$ is irrational, then $p_\varepsilon$ is topologically conjugate to an irrational rotation (see, for instance, \cite[Theorems 2.4 and 2.5]{haleordinary}).  In this case, the dynamics of \eqref{eq:systemextended} on the invariant torus $M_{\e}$ corresponds to an irrational flow and, therefore, the torus $M_{\e}$ is a minimal invariant compact manifold, in the sense that there is no other compact invariant manifold of \eqref{eq:systemextended} contained in $M_{\e}$ besides itself. Therefore, since $\rho(\e)=\e^{\ell} T/\omega+\CO(\e^{\ell+1})$ is continuous, we can always find $\e^*\in (0,\ov \e]$ such that $M_{\e^*}$ is minimal.
\end{remark}

\begin{remark}
The conclusion provided by statement e) that ``if $r-\ell\geq4$, then $\rho$ maps zero Lebesgue measure sets to zero Lebesgue measure sets'' is known as Luzin-N-property of the function $\rho$.
\end{remark}

\subsection{Application: invariant tori in $4$D vector fields} \label{subsec:app}
Theorem \ref{maintheorem} provides a means for investigating the existence of invariant tori also in higher dimensional vector fields. For instance, for a positive integer $N\geq 2$, consider the following $4$D  autonomous differential system
\begin{equation} \label{systemexample}
	\begin{aligned}
		&\dot x = -y +\varepsilon^N f_1(x,y,u,v) + \varepsilon^{N+1} g_1(x,y,u,v)+\e^{N+2} h_1(x,y,u,v,\e), \\
		&\dot y = x + \varepsilon^N f_2(x,y,u,v) + \varepsilon^{N+1} g_2(x,y,u,v) +\e^{N+2} h_2(x,y,u,v,\e), \\
		&\dot u = \varepsilon^N f_3(x,y,u,v) + \varepsilon^{N+1} g_3(x,y,u,v) +\e^{N+2} h_3(x,y,u,v,\e), \\
		&\dot v = \varepsilon^N f_4(x,y,u,v) + \varepsilon^{N+1} g_4(x,y,u,v)+\e^{N+2} h_4(x,y,u,v,\e) ,
	\end{aligned}
\end{equation}
where $\varepsilon$ is a small positive parameter; $\mu\in\{-1,1\}$; $f_i$, for $i \in \{1,2,3,4\}$, are functions of class $C^r$, $r\geq4$, satisfying that
\begin{equation}
\begin{aligned}\label{condonfs}
&\theta\mapsto \cos (\theta ) f_1(r \cos (\theta ),r \sin (\theta ),u,v)+\sin (\theta ) f_2(r \cos (\theta ),r \sin (\theta ),u,v),\\
&\theta\mapsto  f_3(r \cos (\theta ),r \sin (\theta ),u,v),  \,\, \text{and}\\
& \theta\mapsto  f_4(r \cos (\theta ),r \sin (\theta ),u,v) \\
\end{aligned}
\end{equation}
have vanishing average over $\theta \in [0,2\pi]$; $g_i$, for $i \in \{1,2,3,4\}$, are given by 
\begin{equation*}
	\begin{aligned}
		&g_1(x,y,u,v) = \mu x(x^2+y^2), \\
		&g_2(x,y,u,v) =  -\mu y (x^2+y^2)^2,\\
		&g_3(x,y,u,v) =  x^2 (u (-u^2-v^2+1)+v),\\
		&g_4(x,y,u,v) = y^2(v (-u^2-v^2+1)-u);
	\end{aligned}
\end{equation*}
and $h_i$, for $i \in \{1,2,3,4\}$, are $C^r$, $r\geq 4$, functions.

\begin{proposition}\label{exprop}
Assume the conditions above for the differential system \eqref{systemexample}. Then,
for any integer $N\geq 2$ and $\e>0$ sufficiently small, the differential system \eqref{systemexample} has an invariant torus $\mathbb{T}_{\e}$ converging, as $\e$ goes to $0$, to $\mathbb{T}=\mathbb{S}^1\times\mathbb{S}^1$. Moreover, the invariant torus is asymptotically stable provided that $\mu=1$ and has an unstable direction provided that $\mu=-1$.
\end{proposition}

Proposition \ref{exprop} is proven in Section \ref{sec:app}.

\begin{example}
Assuming that
\[
\begin{aligned}
f_1(x,y,u,v)=y u,\,\, f_2(x,y,u,v)=-x v,\,\, f_3(x,y,u,v)=x^3,\,\,\text{and}\,\, f_4(x,y,u,v)=y^3,\\
\end{aligned}
\]
one can easily see that the functions given in \eqref{condonfs} have vanishing average. Thus, Proposition \ref{exprop} can be applied to provide the existence of  an invariant torus for $\e>0$ sufficiently small. In Figure \ref{unpertorus}, assuming $N=2, $ $\mu=1$, $h_i=0,$ $i=1,\ldots,4$, and $\e=1/15$, we provide a numeric simulation (performed on Mathematica) of the Poincar\'{e} map defined on the section $\Sigma=\{(x,0,u,v):\,x>0\}$ of the differential system \eqref{systemexample}. The asymptotically stable invariant tori $\mathbb{T}_{\e}$ corresponds to an asymptotically stable invariant closed curve $\gamma_{\e}:=\mathbb{T}_{\e}\cap \Sigma$ for the Poincar\'{e} map.

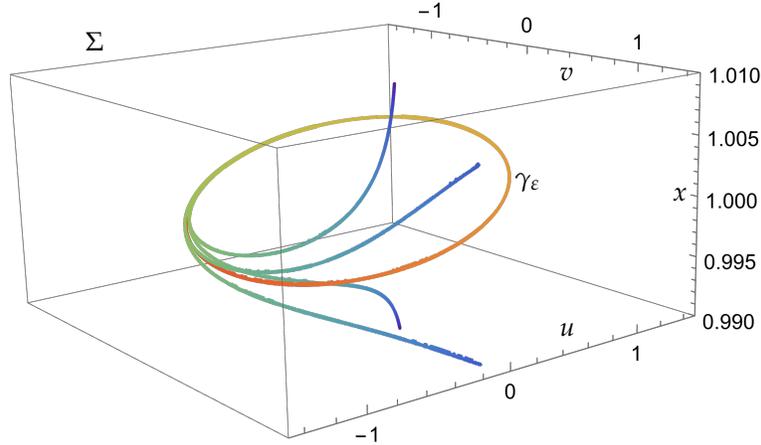
\begin{figure}[H]
	\begin{overpic}[width=10cm]{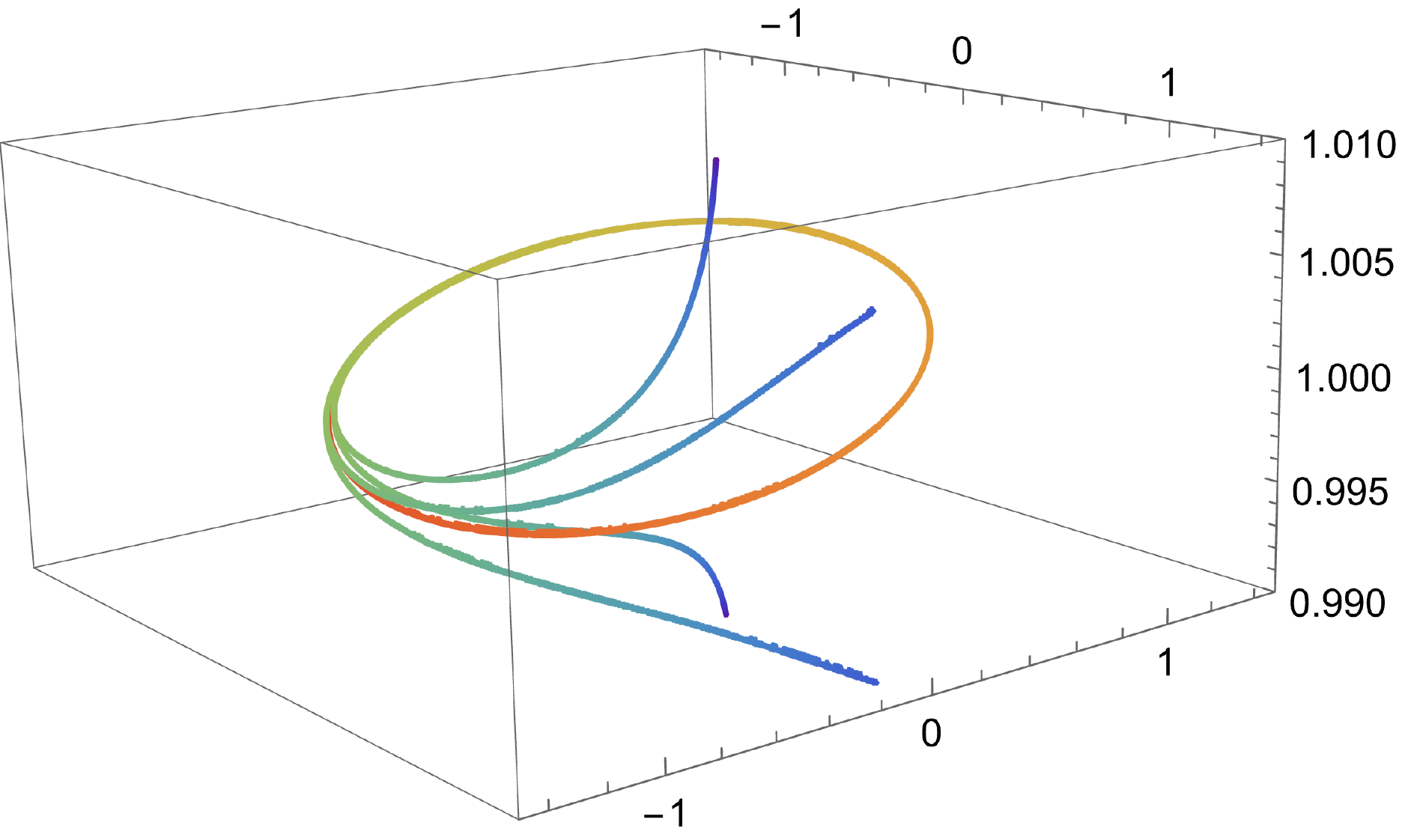}
	
		\put(10,53){$\Sigma$}
		\put(73,15){$u$}
		\put(88,33){$x$}
		\put(73,49){$v$}
		\put(67,35){$\gamma_{\e}$}
	\end{overpic}
	\caption{Assuming $N=2, $ $\mu=1$, $h_i=0,$ $i=1,\ldots,4$, and $\e=1/15$, we show
	$10345$ iterations of the Poincar\'{e} map of \eqref{systemexample}, defined on the section $\Sigma=\{(x,0,u,v):\,x>0\},$ for the initial values: $(1.01,0,2,0)$, $(0.99,0,2,0)$, $(1.01,0,0.5,0),$ and $(0.99,0,0.5,0)$. The orbits are attracted by the closed curve $\gamma_{\e}$, which corresponds to the intersection between the invariant torus $\mathbb{T}_{\e}$ with the section $\Sigma$. For the web version of the paper, purple points indicate a low number of iterations, whilst red points indicate a high number of iterations.}
	\label{unpertorus}	
\end{figure}
\end{example}

\section{Fundamental Lemma} \label{sec:fl}
The proof of Theorem \ref{maintheorem} makes use of some results concerning integral manifolds of a class of perturbed differential systems. Such results, and the methods employed for obtaining them, are similar to those established by Hale in \cite[Lemmas 2.1, 2.2 and 2.3]{haleinvariant} (see also \cite[Section 28, Lemmas 1, 2, and 3]{bogo} and \cite{haleoscillations}). In this section, we state and prove those results in the form of a single Lemma, along with a Proposition addressing the issue of regularity of the integral manifolds obtained.

Throughout the paper, we will adopt the notation $\text{diag}(A_1,\ldots,A_n)$ to represent the direct sum $A_1 \oplus\dots \oplus A_n$ of the square matrices $A_i,$ $i\in\{1,\ldots,n\}$. We will also employ the notation $B_n(p,r)$ for the $n$-dimensional open ball $\{x \in \mathbb{R}^n: \|x-p\|<r \}$.

We consider a one-parameter family of differential systems of the form
\begin{equation} \label{systemlemma}
	\begin{aligned} 
		&\theta' = 1 + \zeta_0(t,\theta,y,z,\varepsilon), \\
		&y' = H_1\cdot y + \zeta_1(t,\theta,y,z,\varepsilon), \\
		&z' = H_2 \cdot z + \zeta_2(t,\theta,y,z,\varepsilon),
	\end{aligned}
\end{equation}
where $\theta \in \mathbb{R}$, $y \in \mathbb{R}^m$, $z \in \R^n$, $\varepsilon$ is a real parameter, $H_1$ is a real $m \times m$ matrix, $H_2$ is a real $n \times n$ matrix, and the continuous functions $\zeta_0:\mathbb{R} \times \mathbb{R} \times B_m(0,\rho_1) \times B_n(0,\rho_2) \times (0,\varepsilon_0] \to \R$, $\zeta_1:\mathbb{R} \times \mathbb{R} \times B_m(0,\rho_1) \times B_n(0,\rho_2) \times (0,\varepsilon_0] \to \R^m$, and $\zeta_2:\mathbb{R} \times \mathbb{R} \times B_m(0,\rho_1) \times B_n(0,\rho_2) \times (0,\varepsilon_0] \to \R^n$ have Lipschitz continuous partial derivatives with respect to $(\theta,y,z)$ up to the $p$-th order, where $p\geq 1$, $\rho_1,\rho_2>0$, and $\varepsilon_0>0$. For conciseness, we define, for each $(\sigma,\mu)\in (0,\rho_1) \times (0,\rho_2)$, the set
\[
\Sigma_{\sigma,\mu}^{\varepsilon_0} := \mathbb{R}\times\mathbb{R} \times \bar{B}_m(0,\sigma) \times \bar{B}_n(0,\mu) \times (0,\varepsilon_0].
\]

We suppose that the following hypotheses are satisfied by \eqref{systemlemma}:
\begin{enumerate}[label=\roman*)]
	\item There is $\omega>0$ such that 
	\[
		\begin{aligned}
			&\zeta_0(t,\theta+\omega,y,z,\e) = \zeta_0(t,\theta,y,-z,\e), \\
			&\zeta_1(t,\theta+\omega,y,z,\e) = \zeta_1(t,\theta,y,-z,\e), \\
			&\zeta_2(t,\theta+\omega,y,z,\e) = -\zeta_2(t,\theta,y,-z,\e).
		\end{aligned}
	\]
	\item There is a continuous function $M:[0,\varepsilon_0]\to \mathbb{R}_+$ such that $M(0)=0$ and the functions $\zeta_i$ satisfy $|\zeta_0(t,\theta,0,0,\varepsilon)|\leq M(\varepsilon)$, $\|\zeta_1(t,\theta,0,0,\varepsilon)\|\leq M(\varepsilon)$, and $\|\zeta_2(t,\theta,0,0,\varepsilon)\|\leq M(\varepsilon)$ for all $(t,\theta,\varepsilon) \in \mathbb{R}\times\mathbb{R} \times (0,\varepsilon_0]$.
	\item There is a continuous function $L:(0,\varepsilon_0] \times [0,\rho_1) \times [0,\rho_2) \to \mathbb{R}_+$ such that 
	$$\lim_{(\e,\sigma,\mu) \to (0,0,0)} L(\e,\sigma,\mu)=0,$$
	and, for $(t,\theta_1,y_1,z_1,\varepsilon),(t,\theta_2,y_2,z_2,\varepsilon) \in \Sigma_{\sigma,\mu}^{\varepsilon_0}$, the following inequalities hold true:
	\[
	\begin{aligned}
	&|\zeta_0(t,\theta_1,y_1,z_1,\varepsilon)- \zeta_0(t,\theta_2,y_2,z_2,\varepsilon)|\leq L(\e,\sigma,\mu)\|(\theta_1,y_1,z_1)-(\theta_2,y_2,z_2)\|, \\
	&\|\zeta_1(t,\theta_1,y_1,z_1,\varepsilon)- \zeta_1(t,\theta_2,y_2,z_2,\varepsilon)\|\leq L(\e,\sigma,\mu)\|(\theta_1,y_1,z_1)-(\theta_2,y_2,z_2)\|, \\  &\|\zeta_2(t,\theta_1,y_1,z_1,\varepsilon)- \zeta_2(t,\theta_2,y_2,z_2,\varepsilon)\|\leq L(\e,\sigma,\mu)\|(\theta_1,y_1,z_1)-(\theta_2,y_2,z_2)\|.
	\end{aligned}
	\]
	
	\item The eigenvalues of $H_1$ and $H_2$ have non-zero real parts.
\end{enumerate}

Let $\big(\theta(t,t_0,\theta_0,y_0,z_0,\e),y(t,t_0,\theta_0,y_0,z_0,\e),z(t,t_0,\theta_0,y_0,z_0,\e)\big)$ denote the solution of (\ref{systemlemma}) with initial conditions $(t_0,\theta_0,y_0,z_0,\e)$. Having set forth the hypotheses above, we are now ready to state the Lemma.
\begin{lemma} \label{lemmahale}
	Consider system \eqref{systemlemma} with the hypotheses presented in this section. There are $\varepsilon_1 \in (0,\varepsilon_0)$ and families of continuous functions $f_\e:\mathbb{R}\times\mathbb{R} \to \mathbb{R}^m$ and $g_\e:\mathbb{R}\times\mathbb{R}\to \mathbb{R}^n$ such that, for each $\varepsilon \in (0,\varepsilon_1]$, the autonomous system
	\begin{equation} \label{systemlemmaautonomized}
		\begin{aligned}
			&t'=1, \\
			&\theta' = 1 + \zeta_0(t,\theta,y,z,\varepsilon), \\
			&y' = H_1\cdot y + \zeta_1(t,\theta,y,z,\varepsilon), \\
			&z' = H_2 \cdot z + \zeta_2(t,\theta,y,z,\varepsilon),
		\end{aligned}
	\end{equation}
	has an invariant manifold given by $y = f_\e(t,\theta)$, $z=g_\e(t,\theta)$. The families $f_\e$ and $g_\e$ also satisfy:
	\begin{enumerate}[label=\alph*)]
		\item \label{thesisD} There is a continuous function $D:[0,\varepsilon_1] \to \mathbb{R}_+$ such that $D(0)=0$
		and \[
		\begin{aligned}
			&\|f_\e(t,\theta)\|\leq D(\varepsilon),
			&\|g_\e(t,\theta)\|\leq D(\varepsilon)
		\end{aligned}
		\] for all $(t,\theta,\varepsilon) \in \mathbb{R}\times\mathbb{R} \times (0,\varepsilon_1]$. Furthermore, there is a constant $C_H>0$, depending only on the matrices $H_1$ and $H_2$, such that $D(\e) = C_H M(\e)$, where $M(\e)$ is the function appearing in hypothesis $ii)$.
		\item \label{thesisdelta} There is a continuous function $\Delta:[0,\varepsilon_1] \to \mathbb{R}_+$ such that $\Delta(0)=0$ and 
		\[
		\begin{aligned}
			&\|f_\e(t,\theta_1) - f_\e(t,\theta_2)\|\leq \Delta(\varepsilon) |\theta_1-\theta_2|,\\
			&\|g_\e(t,\theta_1) - g_\e(t,\theta_2)\|\leq \Delta(\varepsilon) |\theta_1-\theta_2|
		\end{aligned}
		\] for all $t \in \mathbb{R}$, all $\theta_1,\theta_2 \in \mathbb{R}$ and all $\varepsilon \in (0,\varepsilon_1]$.
		\item \label{thesisperiodic} $f_\e$ is $\omega$-periodic in $\theta$ for all $\e \in (0,\e_1]$ and $g_\e$ satisfies $g_\e(t,\theta+\omega) = -g_\e(t,\theta)$ for all $(t,\theta,\e) \in \R \times \R \times (0,\e_1]$;
		\item \label{thesisperiodicT} If, for a given $\e \in (0,\e_1]$, the functions $ \zeta_0(t,\theta,y,z,\e)$, $\zeta_1(t,\theta,y,z,\e)$, and $\zeta_2(t,\theta,y,z,\e)$ are $T_\e$-periodic in the variable $t$, then so are $f_\e$ and $g_\e$;
		\item \label{thesisderivatives} $f_\e$ and $g_\e$ have bounded and uniformly continuous derivatives with respect to $\theta$ up to the $p$-th order for all $\e \in (0,\e_1]$;
		\item \label{thesisstability}Let $\pi_1: \R^m \times \R^n \to \R^m$ and $\pi_2: \R^m \times \R^n \to \R^n$ be the canonical projections. If $m_s\leq m$ and $n_s\leq n$ of the eigenvalues of $H_1$ and $H_2$, respectively have negative real parts, there are positive constants $r$, $\lambda$, $C$, $\sigma_0$, and $\sigma_1$ such that $r\leq\sigma_0$, $D(\e)<\sigma_0<\sigma_1$, and, for each $(t_0,\theta_0,\varepsilon) \in \mathbb{R} \times \mathbb{R} \times (0,\varepsilon_1]$, there is in $\bar{B}_m(0,\sigma_0) \times \bar{B}_n(0,\sigma_0)$ a local $(m_s+n_s)$-dimensional embedded submanifold $S(t_0,\theta_0,\e)$ of $\R^m \times \R^n$, containing the point $(f_\e(t_0,\theta_0),g_\e(t_0,\theta_0))$, and having the following properties:
		\begin{enumerate}[label= f.\arabic*)]
			\item \label{thesisescape} If $(y_0,z_0) \in \bar{B}_m(0,\sigma_0) \times \bar{B}_n(0,\sigma_0) \setminus S(t_0,\theta_0,\varepsilon)$, there is $t_*>t_0$ for which $$(y(t_*,t_0,\theta_0,y_0,z_0,\e),z(t_*,t_0,\theta_0,y_0,z_0,\e)) \notin \bar{B}_m(0,\sigma_1) \times \bar{B}_n(0,\sigma_1).$$
			\item \label{thesisdecay} Reciprocally, if $(y_0,z_0) \in S(t_0,\theta_0,\varepsilon)$, then, for all $t\geq t_0$, $$(y(t,t_0,\theta_0,y_0,z_0,\e),z(t,t_0,\theta_0,y_0,z_0,\e)) \in \bar{B}_m(0,\sigma_1) \times \bar{B}_n(0,\sigma_1)$$ and the following inequalities hold:
			$$\begin{aligned}&\|y(t,t_0,\theta_0,y_0,z_0,\e) - f_\e(t,\theta(t,t_0,\theta_0,f_\e(t_0,\theta_0),g_\e(t_0,\theta_0),\e))\|\leq C e^{-\lambda (t-t_0)} \|y_0-f_\e(t_0,\theta_0)\|, \\
			&\|z(t,t_0,\theta_0,y_0,z_0,\e) - g_\e(t,\theta(t,t_0,\theta_0,f_\e(t_0,\theta_0),g_\e(t_0,\theta_0),\e))\|\leq C e^{-\lambda (t-t_0)} \|z_0-g_\e(t_0,\theta_0)\|.\end{aligned}$$
			\item\label{thesisphi}There is a continuous function $\phi^\e_1: \mathbb{R} \times \mathbb{R} \times \bar{B}_{m_s}(0,r) \times \bar{B}_{n_s}(0,r) \to \bar{B}_{m-m_s}(0,\sigma_0)$ such that \[\pi_1(S(t_0,\theta_0,\varepsilon)) = \{(\phi^\e_1(t_0,\theta_0, \xi_1,\xi_2),\xi_1) :(\xi_1,\xi_2) \in \bar{B}_{m_s}(0,r) \times \bar{B}_{n_s}(0,r)\}.\] Similarly, there is a continuous function $\phi^\e_2: \mathbb{R} \times \mathbb{R} \times \bar{B}_{m_s}(0,r) \times \bar{B}_{n_s}(0,r)  \to \bar{B}_{n-n_s}(0,\sigma_0)$ such that $$\pi_2(S(t_0,\theta_0,\varepsilon)) = \{(\phi^\e_2(t_0,\theta_0,\xi_1, \xi_2),\xi_2) :(\xi_1,\xi_2) \in \bar{B}_{m_s}(0,r) \times \bar{B}_{n_s}(0,r)\}.$$ 
			\item \label{thesisphiperiodic} The functions $\phi_1^\e$ and $\phi_2^\e$ satisfy
			\[
				\phi_1^\e(t_0,\theta_0+\omega,\xi_1,\xi_2) = \phi_1^\e(t_0,\theta_0,\xi_1,-\xi_2)
			\]
			and 
			\[
				\phi_2^\e(t_0,\theta_0+\omega,\xi_1,\xi_2) = 	-\phi_2^\e(t_0,\theta_0,\xi_1,-\xi_2)
			\]
			for all $(t_0,\theta_0,\xi_1,\xi_2) \in \mathbb{R} \times \mathbb{R} \times \bar{B}_{m_s}(0,r) \times \bar{B}_{n_s}(0,r)$.
		\end{enumerate}
	\end{enumerate}
\end{lemma}
	\begin{proof}
		The argument is very similar to the one found in \cite[Lemmas 2.1, 2.2, and 2.3]{haleinvariant} (see also \cite[Section 28, Lemmas 1,2, and 3]{bogo}). We will omit computations when analogous ones can be found in those references, simply referring the reader to them. 
		
		Without loss of generality, suppose that $H_1=\text{diag}(H_1^+,H_1^-)$ and $H_2=\text{diag}(H_2^+,H_2^-)$, with the eigenvalues of $H_i^+$ and $H_i^-$, $i \in \{1,2\}$, having respectively positive and negative real parts. For each $i \in \{1,2\}$, define
		\[
		J_i(t) = \left\{ 
		\begin{array}{l}
			-\left[\begin{array}{cc}
				e^{-t H_i^+} &0 \\
				0 & 0
			\end{array}\right],\quad t>0; \\[2\normalbaselineskip]
			\quad\left[\begin{array}{cc}
				0 &0 \\
				0 & e^{-t H_i^-}
			\end{array}\right], \quad t<0.
		\end{array} \right.
		\]
		
	Consider the complete metric space $\mathcal{P}_\omega(D,\Delta)$ of continuous functions $F \in C(\R^2;\R^m)$ satisfying:
		\begin{itemize}
			\item $F(t,\theta+\omega) = F(t,\theta)$ for all $(t,\theta) \in \R \times \R$;
			\item $\|F(t,\theta)\|\leq D$ for all $(t,\theta) \in \R \times \R$;
			\item $\|F(t,\theta_1) - F(t,\theta_2)\|\leq \Delta |\theta_2-\theta_1|$ for all $(t,\theta_1),(t,\theta_2) \in \R \times \R$,
		\end{itemize}
	where the metric is given by the uniform norm.
	Similarly, consider the complete metric space $\mathcal{A}_\omega(D,\Delta)$ of continuous functions $G \in C(\R^2;\R^n)$ satisfying:
	\begin{itemize}
		\item $G(t,\theta+\omega) = -G(t,\theta)$ for all $(t,\theta) \in \R \times \R$;
		\item $\|G(t,\theta)\|\leq D$ for all $(t,\theta) \in \R \times \R$;
		\item $\|G(t,\theta_1) - G(t,\theta_2)\|\leq \Delta |\theta_2-\theta_1|$ for all $(t,\theta_1),(t,\theta_2) \in \R \times \R$.
	\end{itemize}

	For each $(F,G) \in \mathcal{P}_\omega(D,\Delta) \times \mathcal{A}_\omega(D,\Delta)$, let $T_{F,G}(t,t_0,\theta_0,\e)$ denote the solution of 
	\[\theta'= 1+ \zeta_0(t,\theta,F(t,\theta),G(t,\theta),\e)\]
	satisfying $T_{F,G}(t_0,t_0,\theta_0,\e) = \theta_0$. Since $T_{F,G}(t,t_0,\theta_0+\omega,\e)$ and $T_{F,G}(t,t_0,\theta_0,\e)+\omega$ are both solutions of the same initial problem, it follows that 
	\begin{equation}\label{eq:Ttheta+omega}
		T_{F,G}(t,t_0,\theta_0+\omega,\e)=T_{F,G}(t,t_0,\theta_0,\e)+\omega.
	\end{equation}
	
	For each $\e \in (0,\e_0]$, define the function $S^\e(F,G) = (S_1^\e(F,G),S_2^\e(F,G)) \in C(\R^2;\R^m) \times C(\R^2;\R^n)$, acting on the metric space $\mathcal{P}_\omega(D,\Delta) \times \mathcal{A}_\omega(D,\Delta)$ and given by
	\begin{flalign*}
		S_1^\e(F,G)(t,\theta) = \int_{-\infty}^\infty J_1(x) \zeta_1 (t+x,T_{F,G}(t+x,t,\theta,\e),F&(t+x,T_{F,G}(t+x,t,\theta,\e)),G(t+x,T_{F,G}(t+x,t,\theta,\e)),\e) dx, \\
		S_2^\e(F,G)(t,\theta) = \int_{-\infty}^\infty J_2(x) \zeta_2 (t+x,T_{F,G}(t+x,t,\theta,\e),F&(t+x,T_{F,G}(t+x,t,\theta,\e)),G(t+x,T_{F,G}(t+x,t,\theta,\e)),\e) dx.
	\end{flalign*}
	By performing the change of variable of integration $\tau=x+t$ and differentiating the compositions $F(t,T_{F,G}(t,t_0,\theta_0,\e))$ and $G(t,T_{F,G}(t,t_0,\theta_0,\e))$ with respect to $t$, it is easy to see that, if $(f_\e,g_\e)$ is a fixed-point of $S^\e$, then the manifold given by $y=f_\e(t,\theta)$, $z=g_\e(t,\theta)$ is invariant under \eqref{systemlemmaautonomized}. Thus, the problem is reduced to proving that $S^\e$ admits a fixed-point.
	
	Following the arguments found in \cite[Lemma 2.1]{haleinvariant}, we conclude that it is possible to find $\e_1 \in (0,\e_0]$ and define $D(\e)$ and $\Delta(\e)$ such that, for $\e \in (0,\e_1]$, the function $S^\e$ is a contraction of $\mathcal{P}_\omega(D(\e),\Delta(\e)) \times \mathcal{A}_\omega(D(\e),\Delta(\e))$ into itself. The only change compared to the arguments found in the reference is that, in order to guarantee that $S_2^\e(F,G) \in \mathcal{A}(D(\e),\Delta(\e))$, we need to show that $S_2^\e(F,G)$ satisfies $S_2^\e(F,G)(t,\theta+\omega) = -S_2^\e(F,G)(t,\theta)$. However, this is easily seen by considering \eqref{eq:Ttheta+omega} and the properties of functions $\zeta_1$ and $\zeta_2$. Furthermore, since in \cite{haleinvariant} it is proved that $\Delta(\e) \to 0$ as $\e \to 0$ and that there is $C_H>0$ such that $D(\e) = C_H M(\e)$, properties $\ref{thesisD}$ and $\ref{thesisdelta}$ are ensured to hold.
	
	It remains to show that properties $\ref{thesisperiodic}$ to $\ref{thesisstability}$ are valid. Observe that property $\ref{thesisperiodic}$ follows directly from the fact that $(f_\e,g_\e) \in \mathcal{P}_\omega(D(\e),\Delta(\e)) \times \mathcal{A}_\omega(D(\e),\Delta(\e))$. Property $\ref{thesisperiodicT}$ follows from the same argument found in \cite[Lemma 2.2]{haleinvariant}. Property $\ref{thesisderivatives}$ is not directly discussed in \cite{haleinvariant}, but it is stated and proved in \cite[Section 28, Lemma 1]{bogo}. The proof in our case is essentially the same.
	
	Finally, we proceed to discussing property $\ref{thesisstability}$. For $t_0,\sigma_0,\nu \in \R$ and $k \in \mathbb{N}$, let $\mathcal{C}^{t_0}_{k}(\sigma_0,\nu)$ br the complete metric space of continuous functions $W: [t_0,\infty) \times \R \to \R^k$ satisfying:
	\begin{itemize}
		\item $\|W(t,\theta)\| \leq \sigma_0$ for all $(t,\theta) \in [t_0,\infty) \times \R$;
		\item $\|W(t,\theta_1) - W(t,\theta_2)\| \leq \nu |\theta_1 -\theta_2|$ for all $(t,\theta_1),(t,\theta_2) \in[t_0,\infty) \times \R$.
	\end{itemize}
	The metric of $\mathcal{C}^{t_0}_{k}(\sigma_0,\nu)$ is given by the uniform norm.
	 
	Let $t_0,\sigma_0,\nu \in \R$ be given. For each $b_1 \in \R^m$ and each $b_2 \in \R^n$, define the following functions acting on $\mathcal{C}^{t_0}_m(\sigma_0,\nu)$ and $\mathcal{C}^{t_0}_n( \sigma_0,\nu)$, respectively:
	\begin{flalign*}
		S_1^{\e,b_1}(W_1,W_2)(t,\theta) = J_1&(t_0-t)b_1\\ + &\int_{t_0}^\infty J_1(\tau-t) \zeta_1 (\tau,T_{W_1,W_2}(\tau,t,\theta,\e),W_1(\tau,T_{W_1,W_2}(\tau,t,\theta,\e)), W_2(\tau,T_{W_1,W_2}(\tau,t,\theta,\e)),\e) d\tau, \\
		S_2^{\e,b_2}(W_1,W_2)(t,\theta) = J_2&(t_0-t)b_2\\ + &\int_{t_0}^\infty J_2(\tau-t) \zeta_2 (\tau,T_{W_1,W_2}(\tau,t,\theta,\e),W_1(\tau,T_{W_1,W_2}(\tau,t,\theta,\e)), W_2(\tau,T_{W_1,W_2}(\tau,t,\theta,\e)),\e) d\tau.
	\end{flalign*}
	
	Let $S^{\e,b_1,b_2}$ act on $\mathcal{C}^{t_0}_m(\sigma_0,\nu) \times \mathcal{C}^{t_0}_n( \sigma_0,\nu)$ by $S^{\e,b_1,b_2}(W_1,W_2) = (S^{\e,b_1}_1(W_1,W_2),S^{\e,b_2}_2(W_1,W_2))$. Then, following the same procedure as before, we can ensure that, by taking $\e_1$, $\sigma_0$, $\nu$, and $r<\sigma_0$ sufficiently small, $S^{\e,b_1,b_2}$ becomes a contraction of $\mathcal{C}^{t_0}_m(\sigma_0,\nu) \times \mathcal{C}^{t_0}_n( \sigma_0,\nu)$ into itself if $\|b_1\|,\|b_2\|\leq r$.
	
	Define $\Psi_1^{\e,t_0}$ and $\Psi_2^{\e,t_0}$ to be such that $(t,\theta) \mapsto (\Psi_1^{\e,t_0}(t,\theta,b_1,b_2),\Psi_2^{\e,t_0}(t,\theta,b_1,b_2))$ is the fixed point of the operator $S^{\e,b_1,b_2}$. Then, it is easy to see that there is $C_0>0$ such that
	\begin{equation} \label{eq:lemma1eq1}
	\|\Psi_i^{\e,t_0} (t,\theta,b_1,b_2) - \Psi_i^{\e,t_0} (t,\tilde{\theta},\tilde{b}_1,\tilde{b}_2)\| \leq C_0 e^{\frac{-\alpha}{2}(t-t_0)} \left(\|b_1 - \tilde{b}_1\| + \|b_2 - \tilde{b}_2\| \right) + \nu |\theta - \tilde{\theta}|
	\end{equation}
	for $i \in \{1,2\}$, $t, \theta,\tilde{\theta} \in \R$, $b_1,\tilde{b}_1 \in \bar{B}_m(0,r)$, and $b_2,\tilde{b}_2 \in \bar{B}_n(0,r)$. This ensures, in particular, that, for $i \in\{1,2\}$, $\Psi_i^{\e,t_0}$ is continuous if seen as a function on $\R \times \R \times \bar{B}_m(0,r) \times \bar{B}_n(0,r)$.
	
	Following the argument in \cite[Section 28, Lemma 3]{bogo}, we can prove that, if $\sigma_1\geq\sigma_0$, every solution of \eqref{systemlemmaautonomized} satisfying
	\begin{itemize}
		\item $y_0 \in \bar{B}_m(0,\sigma_0)$ and $y(t,t_0,\theta_0,y_0,z_0,\e) \in \bar{B}_m(0,\sigma_1)$; 
		\item $z_0 \in \bar{B}_n(0,\sigma_0)$ and $z(t,t_0,\theta_0,y_0,z_0,\e) \in \bar{B}_n(0,\sigma_1)$
	\end{itemize} 
	must be of the form $(t,\theta(t),\Psi_1^\e(t,\theta(t),b_1,b_2),\Psi_2^\e(t,\theta(t),b_1,b_2))$ for some $(b_1,b_2) \in B_m(0,r) \times B_n(0,r)$, where $\theta(t)$ denotes $T_{\Psi_1^\e,\Psi_2^\e}(t,t_0,\theta_0,\e)$. Conversely, every solution of the form given above clearly satisfies the two conditions set forth. Therefore, define
	\[
		S(t_0,\theta_0,\e): = \left\{\lim_{t \to t_0^+} \left(\Psi_1^{\e,t_0}(t,\theta_0,b_1,b_2),\Psi_2^{\e,t_0}(t,\theta_0,b_1,b_2)\right): (b_1,b_2) \in \bar{B}_m(0,r) \times \bar{B}_n(0,r)\right\}.
	\]
	Then, considering also \eqref{eq:lemma1eq1}, properties $\ref{thesisescape}$ and $\ref{thesisdecay}$ follow immediately.
	
	Observe that, from the definition of the operators $S_1^{\e,b_1}$ and $S_2^{\e,b_2}$, it follows that the last $m_s$ and $n_s$ elements of the vectors $\Psi_1^{\e,t_0}(t_0^+,\theta,b_1,b_2)$ and $\Psi_2^{\e,t_0}(t_0^+,\theta,b_1,b_2)$ coincide with, respectively, the last $m_s$ and $n_s$ elements of the vectors $b_1$ and $b_2$. Thus, for each $i \in \{1,2\}$, define $\phi_i^\e$ by
	\[
	(\phi_i^\e(t_0,\theta_0,\xi_1,\xi_2),\xi_i) = \lim_{t \to t_0^+}\Psi_i^{\e,t_0}(t,\theta_0,(0,\xi_1),(0,\xi_2)).
	\]
	Then, it is clear that $\phi_i^\e$ is continuous. It is also clear from this definition that property $\ref{thesisphi}$ holds. 
	
	Finally, since the change of variables $(t,\theta,y,z) \to (\tilde{t},\tilde{\theta},\tilde{y},\tilde{z}) = (t,\theta-\omega,y,-z)$ carries system \eqref{systemlemmaautonomized} into an identical system, it follows from the already proved properties $\ref{thesisescape}$ and $\ref{thesisdecay}$ that: a point $(y_*,z_*) \in \R^m \times \R^n$ is in $S(t_0,\theta_0+\omega,\e)$ if, and only if, $(y_*,-z_*) \in S(t_0,\theta_0,\e)$.  Hence, it follows that, for each $(\xi_1,\xi_2) \in \bar{B}_{m_s}(0,r) \times \bar{B}_{n_s}(0,r)$, there is $(\tilde{\xi}_1,\tilde{\xi}_2) \in \bar{B}_{m_s}(0,r) \times \bar{B}_{n_s}(0,r)$ such that 
	\[
		(\phi_1^\e(t_0,\theta_0+\omega,\xi_1,\xi_2),\xi_1) =(\phi_1^\e(t_0,\theta_0,\tilde{\xi}_1,\tilde{\xi}_2),\tilde{\xi}_1),
	\]
	and 
	\[
	(\phi_2^\e(t_0,\theta_0+\omega,\xi_1,\xi_2),\xi_2) = - (\phi_2^\e(t_0,\theta_0,\tilde{\xi}_1,\tilde{\xi}_2),\tilde{\xi}_2).
	\]
	Therefore, $\xi_1=\tilde{\xi}_1$ and $\xi_2=-\tilde{\xi}_2$, and property $\ref{thesisphiperiodic}$ follows. This concludes the proof of the Lemma.
	\end{proof}
	
	The following Corollary addresses the issue of uniqueness of the invariant manifold found in the previous Lemma. Its proof will not be presented here, but it follows essentially from the stability property provided by statement $\ref{thesisstability}$ (see, for instance, \cite[Remark 2.2]{haleinvariant} and, for more details,  \cite[Remark of page 494]{bogo}).
	\begin{corollary}\label{corollaryunique}
		 For each $\e \in(0,\e_1]$, the invariant manifold given by $y=f_\e(t,\theta)$, $z=g_\e(t,\theta)$ is unique in $\R \times \R \times B_m(0,\sigma_0) \times B_n(0,\sigma_0)$, that is, every invariant manifold contained in $\R \times \R \times B_m(0,\sigma_0) \times B_n(0,\sigma_0)$ must be contained in the set given by $y=f_\e(t,\theta)$, $z=g_\e(t,\theta)$.
	\end{corollary}

The following Proposition is presented to address the issue of regularity of the invariant manifold whose existence was established in the previous Lemma. 
\begin{proposition} \label{corollaryregularity1}
	Consider system \eqref{systemlemma} with the hypotheses presented in this section. Suppose that, for each $\varepsilon \in (0,\varepsilon_1]$ and each $i \in \{0,1,2\}$ fixed, the functions $(t,\theta,y,z) \mapsto \zeta_i(t,\theta,y,z,\varepsilon)$ are of class $C^p$. Then, the invariant manifold found in Lemma \ref{lemmahale} above, that is, the manifold $M_\varepsilon=\{(t,\theta,f_\e(t,\theta),g_\e(t,\theta)) \in \R \times \R \times B_m(0,\sigma_0) \times B_n(0,\sigma_0): (t,\theta) \in \mathbb{R}\times\mathbb{R}\}$, is of class $C^p$. Moreover, for each $\varepsilon \in (0,\varepsilon_1]$, the functions $(t,\theta) \mapsto f_\e(t,\theta)$ and $(t,\theta) \mapsto g_\e(t,\theta)$ are also of class $C^p$.
\end{proposition}
The next three Propositions consider the issue of regularity of the family $(f_\e,g_\e)$ with respect to the parameter $\e$. They will be crucial when we discuss the statement concerning dynamics of Theorem \ref{maintheorem}.
\begin{proposition} \label{corollaryregularity2}
	Consider system \eqref{systemlemma} with the hypotheses presented in this section. Suppose that $\zeta_0$, $\zeta_1$, and $\zeta_2$ are of class $C^2$. Let $c: (0,\e_1) \to C(\R^2;\R^m) \times C(\R^2; \R^n) $ be defined by $c(\e) = (f_\e,g_\e)$, where $C(\R^2;\R^m)$ and $C(\R^2; \R^n)$ are equipped with the uniform norm. If $\e_1>0$ is sufficiently small, then $c$ is of class $C^1$.
\end{proposition}

\begin{proposition} \label{propositiondeltatheta}
	Consider system \eqref{systemlemma} with the hypotheses presented in this section. Suppose that $\zeta_0$, $\zeta_1$, and $\zeta_2$ are of class $C^{p+1}$. Let $T_{F,G}$ be defined as in the proof of Lemma \ref{lemmahale}. Then, if $\e_1>0$ is sufficiently small, then the following holds: there is $N_T \in \mathbb{N}$ and, for each compact interval $[a,b] \subset (0,\e_1]$, there are $C_{[a,b]}>0$ and $M_{[a,b]}>0$ such that
	\[
	\left\|\frac{\partial^q f_\e}{\partial \theta^q} \right\| \leq C_{[a,b]} , \qquad \left\|\frac{\partial^q g_\e}{\partial \theta^q} \right\| \leq C_{[a,b]}, 
	\]
	and 
	\[
	\left|\frac{\partial^q T_{f_\e,g_\e}}{\partial \theta_0^q}(t+x,t,\theta,\e) \right| \leq M_{[a,b]}\, e^{N_{T} L(\e)(1+2\Delta(\e)) |x|}
	\]
	for all $q \in \{1,\ldots,p+1\}$ and all $\e \in [a,b]$. 
\end{proposition}

Proposition \ref{propositiondeltatheta} admits the following Corollary, which is a straightforward application of the mean value inequality.

\begin{corollary} \label{corollarydeltatheta}
	Consider system \eqref{systemlemma} with the hypotheses presented in this section. Suppose that $\zeta_0$, $\zeta_1$, and $\zeta_2$ are of class $C^{p+1}$. Then, if $\e_1>0$ is sufficiently small, then the following holds: for each compact interval $[a,b] \subset (0,\e_1]$, there is $C_{[a,b]}>0$ such that
	\[
	\left\|\frac{\partial^q f_\e}{\partial \theta^q} (t,\theta_2) - \frac{\partial^q f_\e}{\partial \theta^q} (t,\theta_1)\right\| \leq C_{[a,b]} |\theta_2-\theta_1|, \qquad \left\|\frac{\partial^q g_\e}{\partial \theta^q} (t,\theta_2) - \frac{\partial^q g_\e}{\partial \theta^q} (t,\theta_1)\right\| \leq C_{[a,b]} |\theta_2-\theta_1|
	\]
	for all $q \in \{1,\ldots,p\}$, all $t,\theta_1,\theta_2 \in \R$, and all $\e \in [a,b]$. 
\end{corollary}

\begin{proposition} \label{corollaryregularity3}
	Consider system \eqref{systemlemma} with the hypotheses presented in this section. Suppose that $\zeta_0$, $\zeta_1$, and $\zeta_2$ are of class $C^{p+1}$. Let $q\leq p$ be a non-negative integer. Then, if $\e_1>0$ is sufficiently small, then the functions
	\[
	\e \mapsto \frac{\partial^{q}f_\e}{\partial \theta^{q}} \quad \text{and} \quad\e \mapsto \frac{\partial^{q}g_\e}{\partial \theta^{q}}
	\]
	are locally Lipschitz continuous in the uniform norm for $\e \in (0,\e_1]$.
\end{proposition}

\section{Proof of Theorem \ref{maintheorem}}\label{sec:proof}
This Section is devoted to the proof of Theorem \ref{maintheorem}. In Subsection \ref{subsec:cv}, we perform the change of variables that transforms system \eqref{eq:e1} into a system to which Lemma \ref{lemmahale} can be applied. In Subsection \ref{subsec:exregcon}, we apply this Lemma in order to prove the statements regarding existence, regularity, and convergence of Theorem \ref{maintheorem}. In Subsection \ref{subsec:stab}, we prove the statement regarding stability, and finally, in Subsection \ref{subsec:dyn}, the statement regarding the dynamics on the object $M_\varepsilon$ of the same Theorem.

\subsection{Change of variables} \label{subsec:cv}
Consider the differential equation \eqref{eq:e1}. We shall first find a change of coordinates transforming this system into one to which we can apply Lemma \ref{lemmahale}. Thus, let $\ell\in\{1,\ldots,\min(N,r-2)\}$ be such that $\f_1=\cdots\f_{\ell-1}=0$ and $\f_{\ell}\neq0$. By Theorem \ref{thm:av1}, there exists a $T$-periodic near-identity transformation \eqref{avtrans} that transforms the differential equation \eqref{eq:e1} into
\begin{equation}\label{eq:trans1}
	\dot \bz=\e^{\ell}\bg_{\ell}(\bz)+\e^{\ell+1} r_{\ell}(t,\bz,\e).
\end{equation}
Observe that the formulas given in \eqref{yi} ensure that $\bg_\ell$ is of class $C^{r-\ell+1}$. Moreover, $r_\ell$ is of class $C^{r-\ell}$.

Consider the $\omega$-periodic hyperbolic limit cycle $\varphi(s)$ of the guiding system $\dot \bz=\bg_\ell(\bz)$. Also, consider the linear variational equation
\begin{align}\label{eq:firstvariationalequation}
	\frac{dy}{dt} = D \bg_\ell (\varphi(t)) \cdot y.
\end{align}
Observe that $\varphi'(t)$ is a solution to the linear periodic system (\ref{eq:firstvariationalequation}). Let $\Phi(t)$ denote a fundamental matrix solution of this system. We will use Floquet theory to obtain a useful change of variables in a neighborhood of the limit cycle $\Gamma$. 

We remind the reader that the characteristic multipliers of (\ref{eq:firstvariationalequation}) are, for any choice of $\Phi$, the eigenvalues of the monodromy matrix $\Phi^{-1}(0) \Phi(\omega)$ (see, for instance, \cite{chicone2006ordinary}). Since $\Gamma$ is hyperbolic, we know that $1$ is an eigenvalue of multiplicity exactly $1$ of this matrix, all its other eigenvalues being outside the unit circle. By taking into account the real Jordan canonical form of the monodromy matrix, we see that $\Phi(t)$ can be chosen satisfying
\begin{equation}
		\Phi^{-1}(0) \Phi(\omega) = \text{diag}(1,\mathcal{J}_1,\mathcal{J}_2),
\end{equation}
where $\mathcal{J}_1 \in \R^{(n-d-1) \times (n-d-1)}$ and $\mathcal{J}_2 \in \R^{d \times d}$ are matrices in the real Jordan canonical form satisfying the following condition: each Jordan block of $\mathcal{J}_1$ associated to a real negative eigenvalue appears an even number of times, and every Jordan block of $\mathcal{J}_2$ is associated to a real negative eigenvalue and appears only once in this matrix. We remark that, with this choice, the first of column of $\Phi$ must be the only $\omega$-periodic solution of $\eqref{eq:firstvariationalequation}$, which is given by $\varphi'(t)$.
 
Let $I_d$ denote the $d \times d$ identity matrix. Under the above-mentioned conditions, by considering the logarithm of the matrices $R_1$ and $R_2$ (see, \cite{culver} and \cite[page 100]{gantmacher}), we know that there exist real matrices $R_1$ and $R_2$ such that $e^{\omega R_1} = \mathcal{J}_1$ and $e^{\omega R_2 + i\pi I_d}= \mathcal{J}_2$. In particular, we remark that the eigenvalues of $R_1$ and $R_2$ all have non-zero real parts.  For the same reason, the number of eigenvalues (counting multiplicity) of the Poincaré map defined in a transversal section of $\Gamma$ with modulus less than $1$ is equal to the number of eigenvalues (counting multiplicity) of $R := \text{diag}(R_1,R_2)$ with strictly negative real part. 

Define the matrices
\[\tilde{B}:= \text{diag}\left(0,R_1,R_2+i\frac{\pi}{\omega}I_d \right)\]
and
\[B:= \text{diag}\left(0,R_1,R_2\right) = \text{diag}(0,R).\]
It is easy to see that $e^{\omega \tilde{B}} = \Phi^{-1}(0) \Phi(\omega)$ and $e^{2\omega B} = e^{2\omega\tilde{B}} =  (\Phi^{-1}(0) \Phi(\omega))^2$.

Since $ D \bg_\ell (\varphi(t))$ is $\omega$-periodic, Floquet's theorem ensures that there are a $\omega$-periodic matrix function $t \mapsto \tilde{P}(t) \in \mathbb{C}^{n \times n}$ and a $2\omega$-periodic matrix function $t\mapsto P(t) \in \mathbb{R}^{n \times n}$, both of class $C^{r-\ell+1}$, such that 
\[\Phi(t)=\tilde{P}(t) e^{t\tilde{B}} = P(t) e^{tB}.\]
In particular, since the first column of $\Phi(t)$ is $\varphi'(t)$, it follows that $P(t)$ is of the form 
\[ P(t) = \left[\varphi'(t) \,\vline \; Q(t)\right], \]
where $t \mapsto Q(t) \in \R^{n \times (n-1)}$ is $2\omega$-periodic. Also, considering that $B$ and $\tilde{B}$ clearly commute, it follows that \[P(t+\omega) = P(t) e^{\omega(\tilde{B}-B)} = \left[\varphi'(t) \; \vline \; Q(t) A\right],\] 
where 
\begin{equation}\label{eq:definitionA}
	A := \text{diag}(I_{n-d-1},-I_d).
\end{equation} 
Thus, it is clear that $Q$ satisfies
\begin{equation}\label{eq:qs+omega}
	Q(t+\omega) = Q(t) A
\end{equation}
for all $t \in \R. $

Since $\Phi(t)$ solves (\ref{eq:firstvariationalequation}), it follows that
\[
	P'(t) +  P(t) \cdot B = D \bg_\ell(\varphi( t)) \cdot P( t).
\]
Thus, by restricting the equality above to the last $n-1$ columns, we obtain
\begin{equation}\label{eq:identityQomegatau}
	Q'( t) + Q( t) \cdot R = D \bg_\ell(\varphi( t)) \cdot Q( t),
\end{equation}
for all $t \in \R.$

We apply the transformation $\bz \mapsto (s,{\bf h} ) \in \R \times \R^{n-1}$ given by
\begin{equation} \label{eq:transformationzsh}
	\bz = \varphi(s)  +Q(s)\cdot  {\bf h}.
\end{equation} 
Observe that, by taking ${\bf h}$ to be sufficiently small and $s \in [0,\omega)$, we can ensure that the transformation $(s,{\bf h}) \mapsto \bz$ is injective. Accordingly, we will assume henceforth that $\|{\bf h}\| \leq 4\rho$, ensuring that our transformation is bijective. Let us find the differential equation in $(s,{\bf h})$ that is equivalent to (\ref{eq:trans1}). In order to do so, we differentiate (\ref{eq:transformationzsh}) with respect to $t$ and obtain
\[
	\dot \bz = (\varphi'(s) + Q'(s) \cdot {\bf h})\, \dot s + Q(s) {\bf \dot h}.
\]
Thus, by (\ref{eq:trans1}), it follows that
\begin{equation} \label{eq:systemtransformed}
	\begin{aligned}
	(\varphi'(s) + Q'(s) \cdot {\bf h})\, \dot s + Q(s) \cdot {\bf \dot h} = \e^\ell \bg_\ell (\varphi(s) + Q(s) \cdot {\bf h}) + \e^{\ell+1} r(t,\varphi(s) + Q(s) \cdot {\bf h},\e).
	\end{aligned}
\end{equation}

Observe that (\ref{eq:identityQomegatau}) ensures that 
\begin{equation} \label{eq:identitya}
	\begin{aligned}
		&\varepsilon^\ell  \varphi'(s) + \e^\ell Q'(s) \cdot {\bf h} + \e^\ell Q(s)  R \cdot {\bf h} = \e^\ell \bg_\ell(\varphi(s)) + \e^\ell D \bg_\ell (\varphi(s)) \cdot Q(s) \cdot {\bf h}.
	\end{aligned}
\end{equation} 
Let us define the functions
\[
	\begin{aligned}
		&Y(s,{\bf h}) := \bg_{\ell}\left(\varphi(s)  +Q(s) \cdot {\bf h}\right) - \bg_{\ell}(\varphi(s)) -  D \bg_\ell(\varphi(s)) \cdot Q(s) \cdot {\bf h}, \\
		&Z(t,s,{\bf h}, \varepsilon) := r_\ell\left(t,\varphi(s)  +Q(s) \cdot {\bf h},\varepsilon\right).
	\end{aligned}
\]
By subtracting (\ref{eq:identitya}) from (\ref{eq:systemtransformed}), we obtain
\begin{equation} \label{eq:systemaftertransform}
	\begin{aligned} 
	&\left(\varphi'(s) + Q'(s) {\bf h}\right) (\dot s - \varepsilon^\ell) +  Q(s) ({\bf \dot h} - \varepsilon^\ell R \cdot {\bf h}) = \varepsilon^\ell Y(s,{\bf h}) + \varepsilon^{\ell+1}Z(t,s,{\bf h}, \varepsilon).	\end{aligned}
\end{equation}

Observe that \eqref{eq:systemaftertransform} can be rewritten as:
\begin{equation} \label{eq:identitybmatrix}
	\left[\begin{array}{c|ccc}
	\varphi'(s) +Q'(s) \cdot {\bf h} & & &\\
		\rule[-1ex]{0.5pt}{6ex}& \multicolumn{3}{c}{\smash{\raisebox{1\normalbaselineskip}{$\text{\;\Large$ Q(s) $}$}}} 
	\end{array}\right] \cdot 
	\left[\begin{array}{c}
		\dot s - \varepsilon^\ell \\[11pt]
		{\bf \dot h} - \varepsilon^\ell R \cdot {\bf h}
	\end{array}\right] = \varepsilon^\ell Y(s,{\bf h}) + \varepsilon^{\ell+1}Z(t,s,{\bf h}, \varepsilon).
\end{equation}

Observe that the $C^{r-\ell}$ matrix function
\[
	C(s,{\bf h})=
	\left[\begin{array}{c|ccc}
		\varphi'(s) +Q'(s) \cdot {\bf h} & & &\\
		\rule[-1ex]{0.5pt}{6ex}& \multicolumn{3}{c}{\smash{\raisebox{1\normalbaselineskip}{$\text{\;\Large$ Q(s) $}$}}} 
	\end{array}\right]
\]
satisfies $C(s,0)=P(s)$ for all $s \in \R$. Since $P(s)$ is non-singular, for sufficiently small values of ${\bf h}$, the matrix $C(s,{\bf h})$ can be inverted. Therefore, assuming that $\rho>0$ is sufficiently small, if $\|{\bf h}\| \leq 4\rho$, then (\ref{eq:identitybmatrix}) can be transformed into 
\begin{equation} \label{eq:systemfinalformbeforerescaling}
	\left[\begin{array}{c}
		\dot s \\
		{\bf \dot h}
	\end{array}\right] = 
	\varepsilon^\ell 
	\left[\begin{array}{c}
		1\\
		R {\bf h}
	\end{array}\right] +
	\varepsilon^{\ell} (C(s,{\bf h}))^{-1} \cdot Y(s, {\bf h})+ \varepsilon^{\ell+1}(C(s,{\bf h}))^{-1} \cdot Z(t,s,{\bf h}, \varepsilon).
\end{equation}
Set ${\bf h} = ({\bf v},{\bf w}) \in \R^{n-d-1} \times \R^d$. Define $\Lambda_0(s, {\bf v},{\bf w})$ and $\Tilde{\Lambda}_0(t,s,{\bf v},{\bf w}, \e)$ to be the first line of the products $(C(s,{\bf h}))^{-1} \cdot Y(s, {\bf h})$ and $(C(s,{\bf h}))^{-1} \cdot Z(t,s,{\bf h}, \varepsilon)$, respectively. Similarly, define $\Lambda_1(s, {\bf v},{\bf w})$ and $\Tilde{\Lambda}_1(t,s,{\bf v},{\bf w}, \e)$ to be the next $n-d-1$ lines and $\Lambda_2(s, {\bf v},{\bf w})$ and $\Tilde{\Lambda}_2(t,s,{\bf v},{\bf w}, \e)$ to be the last $d$ lines of those products. Then, \eqref{eq:systemfinalformbeforerescaling} becomes 
\begin{equation} \label{eq:systemfinalformbeforerescalinglambda}
	\begin{aligned}
	&\dot s = \e^\ell +  \e^\ell \Lambda_0(s,{\bf v},{\bf w}) + \e^{\ell+1} \tilde{\Lambda}_0(t,s,{\bf v},{\bf w}, \e), \\
	&{\bf \dot v} = \e^\ell R_1\cdot {\bf v} +  \e^\ell \Lambda_1(s,{\bf v},{\bf w}) + \e^{\ell+1} \tilde{\Lambda}_1(t,s,{\bf v},{\bf w}, \e), \\
	&{\bf \dot w} = \e^\ell R_2\cdot {\bf w} +  \e^\ell \Lambda_2(s,{\bf v},{\bf w}) + \e^{\ell+1} \tilde{\Lambda}_2(t,s,{\bf v},{\bf w}, \e).
	\end{aligned}
\end{equation}

We apply the time rescaling $\varepsilon^\ell t = \tilde{t}$ to (\ref{eq:systemfinalformbeforerescalinglambda}) and finally obtain
\begin{equation} \label{eq:systemfinalformafterrescalinglambda}
	\begin{aligned}
		&s' = 1 +  \Lambda_0(s,{\bf v},{\bf w}) + \e \tilde{\Lambda}_0(\tilde{t}/\e^\ell,s,{\bf v},{\bf w}, \e), \\
		&{\bf v}' =  R_1\cdot {\bf v} +   \Lambda_1(s,{\bf v},{\bf w}) + \e \tilde{\Lambda}_1(\tilde{t}/\e^\ell,s,{\bf v},{\bf w}, \e), \\
		&{\bf w}' =  R_2\cdot {\bf w} +  \Lambda_2(s,{\bf v},{\bf w}) + \e \tilde{\Lambda}_2(\tilde{t}/\e^\ell,s,{\bf v},{\bf w}, \e).
	\end{aligned}
\end{equation}
where $\prime$ denotes a derivative with respect to $\tilde{t}$. Such differential system is well defined on $(\tilde{t},s,{\bf v}, {\bf w},\e) \in \R \times \R \times B_{n-d-1}(0,2\rho) \times B_{d}(0,2\rho) \times (0,\e_0]$.

\subsection{Existence, Regularity, and Convergence} \label{subsec:exregcon}
Henceforth, we consider that \eqref{eq:systemfinalformafterrescalinglambda} is defined over $\R \times \R \times B_{n-d-1}(0,\rho) \times B_{d}(0,\rho) \times (0,\e_0]$. Observe that \eqref{eq:systemfinalformafterrescalinglambda} is of the form considered in Lemma \ref{lemmahale}. We must now show that the hypotheses required for the application of that Lemma hold in our case. Observe that the fact that the parameter $\varepsilon$ appears in the denominator of the first argument of $\tilde{\Lambda}_0$, $\tilde{\Lambda}_1$, and $\tilde{\Lambda}_2$ in \eqref{eq:systemfinalformafterrescalinglambda} will not be an impediment to the application of the Lemma, since it is not required in its hypotheses that the functions appearing in the system be defined at $\varepsilon=0$. In fact, the conditions concerning boundedness, be it of the functions themselves or of their Lipschitz constants, can still be proved by resorting to the periodicity of $Z$.

For convenience, we will henceforth adopt the following notation
\[
\begin{aligned}
	& (C^{-1} \cdot Y) (s,{\bf h}) := (C (s,{\bf h}))^{-1} \cdot Y (s,{\bf h}), \\
	&(C^{-1} \cdot Z) (t,s,{\bf h}, \varepsilon) := (C (s,{\bf h}))^{-1} \cdot Z (t,s,{\bf h}, \varepsilon),\\
	&(C^{-1} \cdot Y + \e C^{-1} \cdot Z) (t,s,{\bf h}, \varepsilon) := (C (s,{\bf h}))^{-1} \cdot Y (s,{\bf h}) + \e (C (s,{\bf h}))^{-1} \cdot Z (t,s,{\bf h}, \varepsilon).
\end{aligned}
\]
With that in mind, we proceed to proving that Lemma \ref{lemmahale} can be applied to \eqref{eq:systemfinalformafterrescalinglambda}.

As remarked before, the eigenvalues of the matrix $R$ have non-zero real parts, so that it is immediate that hypothesis (iv) holds. Regarding hypothesis (i), observe that, by their definitions, we know that $Y(s+\omega,{\bf h}) = Y(s,A\cdot {\bf h})$ and $Z(t,s+\omega,{\bf h}, \varepsilon) = Z(t,s, A \cdot {\bf h}, \varepsilon)$. Furthermore, 
\[
C(s+\omega, {\bf h}) = 
\left[\begin{array}{c|ccc}
	\varphi'(s) +Q'(s) A \cdot {\bf h} & & &\\
	\rule[-1ex]{0.5pt}{6ex}& \multicolumn{3}{c}{\smash{\raisebox{1\normalbaselineskip}{$\text{\;\Large$ Q(s) A $}$}}} 
\end{array}\right]
= C(s, A\cdot {\bf h}) 
\left[\begin{array}{c|c}
	1 & 0 \\ \hline \\[-1.5\normalbaselineskip]
	0 & A
\end{array}\right].
\]
Thus, since $A^2=I_{n-1}$, it follows that 
\[
(C(s+\omega, {\bf h}))^{-1} = 
\left[\begin{array}{c|c}
	1 & 0 \\ \hline \\[-1.5\normalbaselineskip]
	0 & A
\end{array}\right] 
(C(s, A\cdot {\bf h}))^{-1}.
\]
Hence, it is easily verified that the following conditions hold:
\begin{itemize}
	\item $\Lambda_0(s+\omega ,{\bf v},{\bf w}) = \Lambda_0(s,{\bf v},-{\bf w})$; \\
	\item $\tilde{\Lambda}_0 (t,s+\omega, {\bf v},{\bf w}, \e) =  \tilde{\Lambda}_0 (t,s+\omega, {\bf v},-{\bf w}, \e)$; \\
	\item $\Lambda_1(s+\omega ,{\bf v},{\bf w}) = \Lambda_1(s,{\bf v},-{\bf w})$; \\
	\item $\tilde{\Lambda}_1 (t,s+\omega, {\bf v},{\bf w}, \e) =  \tilde{\Lambda}_1 (t,s+\omega, {\bf v},-{\bf w}, \e)$; \\
	\item $\Lambda_2(s+\omega ,{\bf v},{\bf w}) = -\Lambda_2(s,{\bf v},-{\bf w})$; \\
	\item $\tilde{\Lambda}_2 (t,s+\omega, {\bf v},{\bf w}, \e) = - \tilde{\Lambda}_2 (t,s+\omega, {\bf v},-{\bf w}, \e)$.
\end{itemize}
This ensures that (i) is valid. It remains to show that hypotheses (ii) and (iii) hold.

In order to do so, let $B_n(p,r)$ denote the open ball $\{x \in \mathbb{R}^n: \|x-p\|<r \}$. Also, let us define the following functions:
\[
	\begin{aligned}
		&\alpha_Y(s,{\bf h}) := \left\|\frac{\partial (C^{-1}\cdot Y)}{\partial (s,{\bf h})}(s,{\bf h})\right\|, \\
		&\alpha_Z{(s,{\bf h},t,\varepsilon)} : =\left\|\frac{\partial (C^{-1}\cdot \tilde{Z})}{\partial (s,{\bf h})}(t,s,{\bf h},\varepsilon)\right\|,
	\end{aligned}
\]
where $\|\cdot\|$ denotes the operator norm.

Let $\varepsilon_0>0$ be fixed. For $\sigma \in (0,\rho)$, define
\[
	L(\sigma):= \sup \left\{\alpha_Y{(s,{\bf h})}: (s,{\bf h}) \in \mathbb{R} \times \bar{B}_{n-1}(0,\sigma) \right\}.
\]
Observe that $\displaystyle \lim_{\sigma \to 0} L(\sigma) = 0$, because 
\[
	\alpha_Y(s,{\bf h}) \leq \left\|\frac{\partial (C^{-1}\cdot Y)}{\partial s}(s,{\bf h})\right\| + \left\|\frac{\partial (C^{-1}\cdot Y)}{\partial {\bf h}}(s,{\bf h})\right\|
\]
and $Y(s,0)=\frac{\partial Y}{\partial s} (s,0) = \frac{\partial Y}{\partial {\bf h}} (s,0) = 0$ for all $s \in \mathbb{R}$. Thus, we can extend $L$ continuously to $[0,\rho)$ by setting $L(0)=0$. Moreover, since $(C^{-1} \cdot Y)$ is also of class $C^{r-\ell}$, with $r-\ell\geq 2$, it follows by the mean value inequality that there is $\mathcal{M}_\rho>0$ such that
\[
	\alpha_Y(s,{\bf h}) \leq \mathcal{M}_\rho \|{\bf h}\| \leq \mathcal{M}_\rho \sigma
\]
for all $(s,{\bf h}) \in \R \times \bar{B}_{n-1}(0,\sigma)$, where $\sigma \in (0,\rho)$. 
Then, again by the mean value inequality, we conclude that
\begin{equation}\label{ineq:lipschitz1}
	\|(C^{-1} \cdot Y)(s_1,{\bf h_1}) - (C^{-1} \cdot Y)(s_2,{\bf h_2}) \| \leq L(\sigma) \|(s_1,{\bf h_1}) - (s_2, {\bf h_2}) \| \leq \mathcal{M}_\rho \sigma \|(s_1,{\bf h_1}) - (s_2, {\bf h_2}) \| ,
\end{equation}
for all $(s_1,{\bf h_1}), (s_2, {\bf h_2}) \in \mathbb{R} \times \bar{B}_{n-1}(0,\sigma)$. 

Since $r-\ell\geq2$, and since the $C^{r-\ell}$ function $\|C^{-1}\cdot Z\|$ is $T$-periodic in its first entry and $2\omega$-periodic in its second entry, it follows that there is $M>0$ such that
\[
	\sup \left\{\left\|(C^{-1}\cdot Z)(\tilde{t}/\e^\ell,s,{\bf h},\varepsilon)\right\|: (\tilde{t},s,{\bf h},\varepsilon) \in \mathbb{R} \times \R \times \{0\} \times (0,\varepsilon_0] \right\}\leq M,
\]
and
\[
	\sup \left\{\alpha_Z{(\tilde{t}/\e^\ell,s,{\bf h},\varepsilon)}: (\tilde{t},s,{\bf h},\varepsilon) \in \mathbb{R} \times \R \times \bar{B}_{n-1}(0,\rho) \times (0,\varepsilon_0] \right\}\leq M.
\]

Thus, it follows on the one hand that 
\begin{equation}\label{ineq:boundednessZ0}
	\|\varepsilon (C^{-1} \cdot Z)(\tilde{t}/\e^\ell,s,0,\e)\|=\|(C^{-1} \cdot Y)(s,0) + \varepsilon (C^{-1} \cdot Z)(\tilde{t}/\e^\ell,s,0,\e)\|  \leq \varepsilon M ,
\end{equation}
for all $(\tilde{t},s,\varepsilon) \in \mathbb{R} \times \mathbb{R} \times (0,\varepsilon_0]$, proving that (ii) is valid with $M(\varepsilon)=\varepsilon M$. On the other hand, the mean value inequality ensures that
\begin{equation}\label{ineq:lipschitz2}
	\|(C^{-1} \cdot Z)(\tilde{t}/\e^\ell,s_1,{\bf h_1},\varepsilon) - (C^{-1}\cdot Z)(\tilde{t}/\e^\ell,s_2,{\bf h_2},\varepsilon) \| \leq M \|(s_1,{\bf h_1}) - (s_2, {\bf h_2}) \|,
\end{equation}
for all $(s_1,{\bf h_1}), (s_2, {\bf h_2}) \in \mathbb{R} \times \bar{B}_{n-1}(0,\rho)$ and all $(\tilde{t},\varepsilon) \in \mathbb{R} \times (0,\varepsilon_0]$. Hence, combining (\ref{ineq:lipschitz1}) and (\ref{ineq:lipschitz2}), we conclude that 
\begin{equation}\label{ineq:lipschitz3}
	\begin{aligned}
	\|(C^{-1}\cdot Y+ \varepsilon C^{-1} \cdot Z)(\tilde{t}/\e^\ell,s_1,{\bf h_1},\varepsilon) - (C^{-1}\cdot Y+ \varepsilon C^{-1} \cdot Z)(\tilde{t}/\e^\ell,s_2,{\bf h_2},\varepsilon) \| \\ \leq (\mathcal{M}_L \sigma + \varepsilon M) \|(s_1,{\bf h_1}) - (s_2, {\bf h_2}) \|,
	\end{aligned}
\end{equation}
for $(s_1,{\bf h_1}),(s_2,{\bf h_2}) \in \R \times \bar{B}_{n-1}(0,\sigma)$ and $(\tilde{t},\varepsilon) \in \R \times (0,\e_0]$, ensuring that (iii) also holds.

Hence, all the hypotheses required for the application of Lemma \ref{lemmahale} are valid for system \eqref{eq:systemfinalformafterrescalinglambda}. Since $(C^{-1}\cdot Y)$ and $(t,s,{\bf h}) \mapsto (C^{-1}\cdot Z)(t,s,{\bf h},\varepsilon)$ are of class $C^{r-\ell}$, Proposition \ref{corollaryregularity1} may also be applied with $p=r-\ell$. Applying the above-mentioned results, we obtain $\varepsilon_1>0$ and families of functions $\{f_\varepsilon \in C^{r-\ell}(\mathbb{R}^2; \mathbb{R}^{n-d-1}): {\varepsilon \in (0,\varepsilon_1]}\}$ and $\{g_\varepsilon \in C^{r-\ell}(\mathbb{R}^2; \mathbb{R}^{d}): {\varepsilon \in (0,\varepsilon_1]}\}$ such that
\begin{enumerate}[label=\Roman*)]
	\item \label{propertyinvariance} For each $\varepsilon \in (0,\varepsilon_1]$, the set defined by the relation ${\bf h} = (f_\varepsilon(\tilde{t},s),g_\varepsilon(\tilde{t},s))$ is an invariant manifold for system 
	\begin{align}\label{eq:systemfinalformafterrescalingautonomized}
		\left[\begin{array}{c}
			s^\prime \\
			{\bf  h'}
		\end{array}\right] =
		\left[\begin{array}{c}
			1\\
			H {\bf h}
		\end{array}\right] +
		(C^{-1} \cdot Y)(s, {\bf h})+ \varepsilon (C^{-1} \cdot Z)(\tilde{t}/\e^\ell, s,{\bf h}, \varepsilon), \qquad
		\tilde{t}'=1.
	\end{align} 
	\item \label{propertyD} There is $D(\varepsilon)>0$ such that $\|f_\varepsilon\|_{C^0}\leq D(\varepsilon)$, $\|g_\varepsilon\|_{C^0}\leq D(\varepsilon)$ and $\lim_{\varepsilon \to 0} D(\varepsilon) = 0$. Furthermore, there is a constant $C_R>0$, depending only on the matrix $R$, such that $D(\e) = C_R M(\e) = C_R \e M$.
	\item \label{propertydelta} There is $\Delta(\varepsilon)>0$ such that $f_\varepsilon$ and $g_\e$ are Lipschitz continuous in $s$ with Lipschitz constant $\Delta(\varepsilon)$ and $\lim_{\varepsilon \to 0 } \Delta(\varepsilon) = 0$.
	\item\label{propertyperiodic} $f_\e$ is $\omega$-periodic in $s$ and $g_\e$ satisfies $g_\e(\tilde{t},s+\omega) = -g(\tilde{t},s)$.
	\item  \label{propertyperiodicT} $f_\e$ and $g_\e$ are $\e^\ell T$-periodic in $\tilde{t}$.
	\item \label{propertystability}Let $\pi_1: \R^{n-d-1}\times \R^d \to \R^{n-d-1}$ and $\pi_2: \R^{n-d-1} \times \R^d \to \R^d$ be the canonical projections. Also, let $k_1\leq n-d-1$ and $k_2\leq d$ of the eigenvalues of $R_1$ and $R_2$, respectively have negative real parts. There are positive constants $r$, $\lambda$, $C$, $\sigma_0$, and $\sigma_1$ such that $r<\sigma_0$, $D(\e)<\sigma_0<\sigma_1<\rho$, and, for each $(\tilde{t}_0,s_0,\varepsilon) \in \mathbb{R} \times \mathbb{R} \times (0,\varepsilon_1]$, there is in $\bar{B}_{n-d-1}(0,\sigma_0) \times \bar{B}_d(0,\sigma_0)$ a local $(k_1+k_2)$-dimensional embedded submanifold $S(t_0,s_0,\e)$ of $\R^{n-d-1} \times \R^d$, containing the point $(f_\e(\tilde{t}_0,s_0),g_\e(\tilde{t}_0,s_0))$, and having the following properties:
	\begin{enumerate}[label= VI.\arabic*)]
		\item \label{propertyescape} If ${\bf h_0} \in \bar{B}_{n-d-1}(0,\sigma_0) \times \bar{B}_d(0,\sigma_0) \setminus S(\tilde{t}_0,s_0,\varepsilon)$, there is $\tilde{t}_*>\tilde{t}_0$ for which $${\bf h}(\tilde{t}_*,\tilde{t}_0,s_0,{\bf h_0},\e) \notin \bar{B}_{n-d-1}(0,\sigma_1) \times \bar{B}_d(0,\sigma_1).$$
		\item \label{propertydecay} Reciprocally, if ${\bf h_0} \in S(\tilde{t}_0,s_0,\varepsilon)$, then, for all $\tilde{t}\geq \tilde{t}_0$, $${\bf h}(\tilde{t},\tilde{t}_0,s_0,{\bf h_0},\e) \in \bar{B}_{n-d-1}(0,\sigma_1) \times \bar{B}_d(0,\sigma_1)$$ and the following inequality holds:
		$$\begin{aligned}
			\left\|{\bf h}(\tilde{t},\tilde{t}_0,s_0,{\bf h_0},\e) - {\bf h}(\tilde{t},\tilde{t}_0,s_0,(f_\e(\tilde{t}_0,s_0),g_\e(\tilde{t}_0,s_0)),\e)\right\|\\\leq C e^{-\lambda (\tilde{t}-\tilde{t}_0)} \|{\bf h_0}-(f_\e(\tilde{t}_0,s_0),g_\e(\tilde{t}_0,s_0))\|.
		\end{aligned}$$
		\item\label{propertyphi}There is a continuous function $\phi^\e_1: \mathbb{R} \times \mathbb{R} \times \bar{B}_{k_1}(0,r) \times \bar{B}_{k_2}(0,r) \to \bar{B}_{n-d-k_1-1}(0,\sigma_0)$ such that \[\pi_1(S(\tilde{t}_0,s_0,\varepsilon)) = \{(\phi^\e_1(\tilde{t}_0,s_0, \xi_1,\xi_2),\xi_1) :(\xi_1,\xi_2) \in \bar{B}_{k_1}(0,r) \times \bar{B}_{k_2}(0,r)\}.\] Similarly, there is a continuous function $\phi^\e_2: \mathbb{R} \times \mathbb{R} \times \bar{B}_{k_1}(0,r) \times \bar{B}_{k_2}(0,r)  \to \bar{B}_{d-k_2}(0,\sigma_0)$ such that $$\pi_2(S(\tilde{t}_0,s_0,\varepsilon)) = \{(\phi^\e_2(\tilde{t}_0,s_0,\xi_1, \xi_2),\xi_2) :(\xi_1,\xi_2) \in \bar{B}_{k_1}(0,r) \times \bar{B}_{k_2}(0,r)\}.$$ 
		\item \label{propertyphiperiodic} The functions $\phi_1^\e$ and $\phi_2^\e$ satisfy
		\[
		\phi_1^\e(\tilde{t}_0,s_0+\omega,\xi_1,\xi_2) = \phi_1^\e(\tilde{t}_0,s_0,\xi_1,-\xi_2)
		\]
		and 
		\[
		\phi_2^\e(\tilde{t}_0,s_0+\omega,\xi_1,\xi_2) = 	-\phi_2^\e(\tilde{t}_0,s_0,\xi_1,-\xi_2)
		\]
		for all $(\tilde{t}_0,s_0,\xi_1,\xi_2) \in \mathbb{R} \times \mathbb{R} \times \bar{B}_{k_1}(0,r) \times \bar{B}_{k_2}(0,r)$.
	\end{enumerate}
\end{enumerate}

Let $\mathcal{X}$ be the function associated to the change of coordinates we have performed, that is, 
\[
\mathcal{X}(s,{\bf h})  =  \varphi(s)  + Q(s) \cdot {\bf h}.
\]
Define $w_\e:\R \times \R \to \R \times \R^n$ by $$w_\varepsilon(\tau,s):=\Big(\tau,\mathcal{X}\left(s,\big(f_\e(\varepsilon^\ell \tau,s),g_\e(\e^\ell\tau,s)\big)\right)\Big).$$
Since $\mathcal{X}$ is injective for $s \in [0,\omega)$ and $\|{\bf h}\|\leq \rho$, the function $w_\e$ restricted to $\R \times [0,\omega)$ is injective. It is also clear that $\e_1$ can be taken sufficiently small as to ensure that $w_\e$ is an immersion for all $\e \in (0,\e_1]$.

It is easy to see that property \ref{propertyperiodic} above guarantees that $w_\e$ is $\omega$-periodic in $s$. In fact, considering \eqref{eq:qs+omega} and the definition of $A$ given in \eqref{eq:definitionA}, we have that
\[
	w_\e(\tau,s+\omega) = \Big(\tau,\varphi(s) + Q(s) \, A \cdot \big(f_\e(\varepsilon^\ell \tau,s),-g_\e(\e^\ell\tau,s)\big) \Big) = w_\e(\tau,s).
\]
Thus, $W_\e:=\{w_\e(\tau,s) : (\tau,s ) \in \R \times \R\}\subset \R \times \R^n$ is an embedded cylinder of class $C^{r-\ell}$ that is invariant under the flow of 
\begin{align} \label{eq:systemextendedproof}
	\left\{ \begin{array}{@{}l@{}}
		\bz' = \displaystyle \e^{\ell}\bg_{\ell}(\bz)+\e^{\ell+1} r_{\ell}(\bz,\tau,\e), \\
		\tau'=1.
	\end{array} \right.
\end{align}
Property \ref{propertyperiodicT} ensures that $w_\e$ satisfies $w_\e(\tau+T,s) = (T,0) +w_\e(\tau,s)$. We can thus consider $\tau$ an angular variable modulo $T$ in \eqref{eq:systemextendedproof}, and $W_\varepsilon$ becomes an invariant torus in $\mathbb{S}^1 \times \R^n$. Finally, the torus $M_\e$, invariant under \eqref{eq:systemextended}, is obtained from $W_\e$ by reverting the near-identity periodic transformation $\bx = U(\tau,\bz,\e)$ that we employed in the beginning of the proof. This proves the existence of $M_\e$ stated in Theorem \ref{maintheorem}. The fact that there is a neighborhood $V$ of $\Gamma$ such that any compact manifold that is invariant under \eqref{eq:systemextended} and contained in $\s^1 \times V$ must also be contained in $M_\e$ follows from Corollary \ref{corollaryunique}. 

We proceed to proving the statement regarding regularity of $M_\varepsilon$ in Theorem \ref{maintheorem}. Define $\mathcal{F}_\varepsilon$ by $$\mathcal{F}_\varepsilon(\tau,s) = U\Big(\tau,\mathcal{X}\left(s,\big(f_\e(\varepsilon^\ell \tau,s),g_\e(\e^\ell\tau,s)\big)\right),\e\Big).$$ 
Observe that $\{\mathcal{F}_\varepsilon\}_\varepsilon$ is a family of $C^{r-\ell}$ functions that are also $\omega$-periodic in $s$ and $T$-periodic in $\tau$, and that $M_\varepsilon$ is given by the relation $\bx = \mathcal{F}_\varepsilon(\tau,s)$, i.e., 
\[
M_\e = \{(\tau,\mathcal{F}_\varepsilon(\tau,s)) \in \s^1 \times \R^n : (\tau,s) \in \R \times \R \}.
\]
Moreover, by Proposition \ref{corollaryregularity2}, it follows that the family $\{\mathcal{F}_\varepsilon\}_\varepsilon$ is $C^0$-continuous, that is, continuous in the $C^0$-norm, provided that $\e_1$ is chosen sufficiently small. In fact, this Proposition guarantees that this family is $C^1$ in the $C^0$-norm.

Regarding the statement about convergence, observe that it follows from property \ref{propertyD} that there is $D^*(\varepsilon)$ such that $$\left\|\mathcal{X}\left(s,\big(f_\e(\varepsilon^\ell \tau,s),g_\e(\e^\ell\tau,s)\big)\right) - \varphi(s)\right\|<D^*(\varepsilon)$$ and $\lim_{\varepsilon \to 0} D'(\varepsilon) = 0$. Then, considering that $U$ is locally Lipschitz in its second argument and that both functions appearing inside the norm of the inequality above are periodic, it follows that there is $\delta(\varepsilon)\geq0$ such that $\delta(0)=0$ and $\|\mathcal{F}_\varepsilon(\tau,s) - U(\tau,\varphi(s),\varepsilon)\|<\delta(\varepsilon)$. 

\subsection{Stability} \label{subsec:stab}
Let the non-negative integers $k_1\leq n-d-1$ and $k_2\leq d$ denote the number of eigenvalues with negative real parts of the matrices $R_1$ and $R_2$ respectively. Define the function $q_\e:\R \times \R \times B_{k_1}(0,r) \times B_{k_2}(0,r) \to \R \times \R^n$ by $$q_\e(\tau,s,\xi_1,\xi_2) = \Big( \tau, U\Big(\tau,\mathcal{X}\left(s,\big(\phi_1^\e(\e^\ell \tau,s,\xi_1,\xi_2),\xi_1,\phi_2^\e(\e^\ell \tau,s,\xi_1,\xi_2),\xi_2\big)\right), \,\varepsilon\Big)\Big).$$
Let $S_{M_\e}$ be the image of $q_\e$. We will show that $S_{M_\e}$ is an embedded submanifold in $\R \times \R^n$. 

For convenience, we denote by $q_\e |_{I}$ the restriction of $q_\e$ to the set $\R \times I \times B_{k_1}(0,r) \times B_{k_2}(0,r)$, where $I \subset \R$. Observe that the properties of $\phi_1^\e$ and $\phi_2^\e$ given in $\ref{propertyphiperiodic}$, along with \eqref{eq:qs+omega}, ensure that $S_{M_\e}$ is contained in the image of $q_\e |_{[0,\, \omega)}$. Hence, $S_{M_\e}$ is contained in the union of the images of $q_\e |_{(0,\omega)}$ and $q_\e|_{(-\frac{\omega}{2},\frac{\omega}{2})}$ .

Now, since $r\leq\sigma_0\leq\rho$, $\mathcal{X}(s,{\bf h})$ is injective for $(s, {\bf h}) \in [0,\omega) \times \bar{B}_{n-1}(0,\sigma_0)$, ensuring that $q_\e |_{(0,\omega)}$ and $q_\e|_{(-\frac{\omega}{2},\frac{\omega}{2})}$ are injective. It is then easy to see that $q_\e |_{(0,\omega)}$ and $q_\e|_{(-\frac{\omega}{2},\frac{\omega}{2})}$ are homeomorphisms onto their images, proving that $S_{M_\e}$ is a $(k_1+k_2+2)$-dimensional embedded submanifold of $\R \times \R^n$. As remarked before, in Section \ref{subsec:cv}, if $k$ is the number of characteristic multipliers of the limit cycle $\Gamma$ whose absolute values are less than $1$, then $k=k_1+k_2$. Thus, $S_{M_\e}$ is $k$-dimensional.

We will prove that $S_{M_\e}$ is locally the stable set of $M_\e$. Suppose that ${\bf h_0} \in S(\tilde{t}_0,s_0,\e)$. For convenience, let us define $s_*(\tilde{t}):= s(\tilde{t},\tilde{t}_0,s_0,{\bf h_0},\e)$, $s_{M_\e}(\tilde{t}):=s(\tilde{t},\tilde{t}_0,s_0,(f_\e(\tilde{t}_0,s_0),g_\e(\tilde{t}_0,s_0)),\e)$, ${\bf h_*}(\tilde{t}):={\bf h}(\tilde{t},\tilde{t}_0,s_0,{\bf h_0},\e)$, and finally ${\bf h}_{M_\e}(\tilde{t}):={\bf h}(\tilde{t},\tilde{t}_0,s_0,(f_\e(\tilde{t}_0,s_0),g_\e(\tilde{t}_0,s_0)),\e)$. Also, let
\[
	u(\tilde{t}):=\|s_*(\tilde{t}) - s_b(\tilde{t})\| + \|{\bf h_*}(\tilde{t}) - {\bf h}_{M_\e}(\tilde{t})\|.
\]
Observe that, considering \eqref{eq:systemfinalformafterrescalinglambda}, along with the boundedness and Lipschitz continuity properties that $\Lambda_0$ and $\tilde{\Lambda}_0$ are proved to satisfy, we have that
\[
	u(\tilde{t}) \leq \|{\bf h_*}(\tilde{t}) + {\bf h}_{M_\e}(\tilde{t})\| +\int_{\tilde{t}_0}^{\tilde{t}} (\mathcal{M}_L \sigma_1 + \varepsilon M) u(x) dx.
\]
Thus, considering property $\ref{propertydecay}$ and applying Grönwall's inequality, it follows that
\[
u(\tilde{t}) \leq C e^{(-\lambda+\mathcal{M}_L \sigma_1 + \varepsilon M) (\tilde{t}-\tilde{t}_0)} \|{\bf h_0}-(f_\e(\tilde{t}_0,s_0),g_\e(\tilde{t}_0,s_0))\|.
\]
Hence, if $\rho$ and $\e_1$ are chosen sufficiently small, we ensure that $u(\tilde{t}) \to 0$ as $\tilde{t} \to \infty$. Thus, it follows that, if ${\bf h_0} \in S(\tilde{t}_0,s_0,\e)$, then 
\begin{equation}\label{eq:stability1}
	\lim_{\tilde{t} \to \infty}\|\mathcal{X}(s_*(\tilde{t}),{\bf h_*}(\tilde{t})) - \mathcal{X}(s_{M_e}(\tilde{t}),{\bf h}_{M_\e}(\tilde{t})) \| =0.
\end{equation}

Since $\sigma_1<\rho$, we know that $\mathcal{X}(s,{\bf h})$ is injective for $(s, {\bf h}) \in [0,\omega) \times \bar{B}_{n-1}(0,\sigma_1)$. Consider the following neighborhoods of $M_\e$:
\[
	\begin{aligned}
	V_s := \left\{\left(\tau,U\left(\tau,\mathcal{X}(s,{\bf h}),\e\right)\right): (\tau,s,{\bf h}) \in \R \times \R \times B_{n-1}(0,\sigma_1) \right\}, \\
	W_s := \left\{\left(\tau,U\left(\tau,\mathcal{X}(s,{\bf h}),\e\right)\right): (\tau,s,{\bf h}) \in \R \times \R \times B_{n-1}(0,\sigma_0) \right\}.
	\end{aligned}
\]
It is then clear, considering \eqref{eq:stability1} and the fact that $M_\e$ is an invariant manifold, that the local stable set of $M_\e$ with respect to $V_s$ satisfies $\mathcal{S}^{V_s}_{M_\varepsilon} \cap W_s = S_{M_\e}$.

The same argument with time reversed proves the analogous statement for the local unstable set $\mathcal{U}^{V_u}_{M_\varepsilon} \cap W_u$. In this case, the dimension of the manifold obtained is $2+(n-1-k) = n-k+1$, because the number of eigenvalues of $R$ with positive real part is $n-1-k$.

\subsection{Dynamics} \label{subsec:dyn}
Let $S_\varepsilon\subset \R^{n+1}$ be defined as the section $\tau=0$ of the torus $M_\varepsilon$, that is, the image of the real $1$-periodic function $\Pi_\varepsilon: \theta\mapsto (0,\mathcal{F}_\varepsilon(0,\omega \theta))$. It is clear that $S_\varepsilon$ is $C^{r-\ell}$-diffeomorphic to the circle $S^1$. Once more, let $t \mapsto (s(t,t_0,s_0,{\bf h_0},\varepsilon),{\bf h}(t,t_0,s_0,{\bf h_0},\varepsilon))$ be the solution of \eqref{eq:systemfinalformbeforerescaling} satisfying $(s(t_0,t_0,s_0,{\bf h_0},\varepsilon),{\bf h}(t_0,t_0,s_0,{\bf h_0},\varepsilon)) = (s_0,{\bf h_0})$. Define, for $(\nu,\theta) \in \R \times \R$,
\[s_\varepsilon(\nu,\theta):=s(\nu T,0,\theta,(f_\varepsilon(0,\theta),g_\varepsilon(0,\theta)),\varepsilon). \]

Since $\tau'=1$ in $\eqref{eq:systemextended}$, it follows that the first-return map $p_\varepsilon$ defined on $S_\varepsilon$ under the action of this differential system is well defined. Moreover, it is clear that 
\[
	p_\varepsilon\big(\Pi_\varepsilon(\theta) \big) = \Pi_\varepsilon\left(\frac{s_\varepsilon(1,\omega \theta)}{\omega}\right).
\]
Thus, the real function $$\tilde{p}_\varepsilon:\theta \mapsto \frac{s_\varepsilon(1,\omega \theta)}{\omega}$$ is a lift of $p_\varepsilon$ with respect to the covering map $\Pi_\varepsilon:\R \to S_\varepsilon$. Moreover, this ensures that $p_\varepsilon$ is at least of class $C^{r-\ell}$.
 
Observe that $$\tilde{p}_\varepsilon^n(\theta) = \frac{s_\varepsilon(n,\omega\theta)}{\omega}$$ for all $n \in \mathbb{N}$.
Then, it is clear that the rotation number of $p_\varepsilon$ is given by
\[
	\rho(\varepsilon) : = \lim_{n \to \infty} \frac{\tilde{p}_\e^n(\theta)-\theta}{n} = \lim_{n \to \infty} \frac{s_\varepsilon(n,\omega \theta) - \omega\theta}{n\omega}.
\]
We will rewrite this limit so as to be able to calculate it up to $\ell$-th order of $\e$.

Integrating the first equation of \eqref{eq:systemfinalformbeforerescalinglambda} from $t=0$ to $t=nT$, we obtain
\begin{equation}\label{eq:dynamics1}
	\begin{aligned}
		s_\varepsilon(n,\theta) = &\theta + \varepsilon^\ell nT \\+&  \varepsilon^\ell \int_0^{nT} \Lambda_0\big(s(\tau,0,\theta,(f_\varepsilon(0,\theta),g_\varepsilon(0,\theta)),\varepsilon),{\bf h} (\tau,0,\theta,(f_\varepsilon(0,\theta),g_\varepsilon(0,\theta)),\e)\big) d\tau \\
		+&  \e^{\ell+1}  \int_0^{nT} \tilde{\Lambda}_0\big(\tau,s(\tau,0,\theta,(f_\varepsilon(0,\theta),g_\varepsilon(0,\theta)),\varepsilon),{\bf h} (\tau,0,\theta,(f_\varepsilon(0,\theta),g_\varepsilon(0,\theta)),\e),\e)\big) d\tau .
	\end{aligned}
\end{equation}
Thus, we can define the sequence of functions
\[
	\begin{aligned}
	G_n(\theta,\varepsilon) :=&\frac{1}{n} \int_0^{nT} \frac{\Lambda_0}{\e}\big(s(\tau,0,\theta(f_\varepsilon(0,\theta),g_\varepsilon(0,\theta)),\varepsilon),{\bf h} (\tau,0,\theta,(f_\varepsilon(0,\theta),g_\varepsilon(0,\theta)),\e)\big) d\tau\\
	&+ \frac{1}{n}\int_0^{nT}\tilde{\Lambda}_0\big(\tau,s(\tau,0,\theta,(f_\varepsilon(0,\theta),g_\varepsilon(0,\theta)),{\bf h} (\tau,0,\theta,(f_\varepsilon(0,\theta),g_\varepsilon(0,\theta)),\e)\big) d\tau,
	\end{aligned}
\]
so that \eqref{eq:dynamics1} becomes 
\[
	s_\varepsilon(n,\theta) = \theta + \varepsilon^\ell nT + \varepsilon^{\ell+1} n \, G_n(\theta,\e).
\]

Since $M_\varepsilon$ is an invariant manifold, it is clear that
\[
\begin{aligned}
	{\bf h} (\tau,0,\theta,(f_\e(0,\theta),g_\e(0,\theta)),\e) = (f_\varepsilon(\tau,s(\tau,0,\theta,(f_\e(0,\theta),g_\e(0,\theta)),\e)),g_\varepsilon(\tau,s(\tau,0,\theta,(f_\e(0,\theta),g_\e(0,\theta)),\e))).
\end{aligned}
\]
Then, by changing the variables in the integral, it follows that
\[
\begin{aligned}
	G_n(\theta,\varepsilon) :=& \int_0^{T} \frac{\Lambda_0}{\e}\big(s(n\tau,0,\theta,(f_\varepsilon(0,\theta),g_\varepsilon(0,\theta)),\varepsilon),{\bf h} (n\tau,0,\theta,(f_\varepsilon(0,\theta),g_\varepsilon(0,\theta)),\e)\big) d\tau\\
	&+ \int_0^{T}\tilde{\Lambda}_0\big(n\tau,s(n\tau,0,\theta,(f_\varepsilon(0,\theta),g_\varepsilon(0,\theta)),\varepsilon),{\bf h} (n\tau,0,\theta,(f_\varepsilon(0,\theta),g_\varepsilon(0,\theta)),\e),\e)\big) d\tau.
\end{aligned}
\]
Observe that
\[
\|\Lambda_0(s,{\bf h})  + \e \tilde{\Lambda}_0(t,s,{\bf h},\e)\| \leq \|\Lambda_0(s,{\bf h})  + \e \tilde{\Lambda}_0(t,s,{\bf h},\e) - \e \tilde{\Lambda}_0(t,s,0,\e)\| + \|\e \tilde{\Lambda}_0(t,s,0,\e)\|
\]
for all $(t,s,{\bf h},\e) \in \R \times \R \times B_{n-1}(0,\rho) \times (0,\e_0]$. Then, considering that $\Lambda(s,0)=0$, it follows from \eqref{ineq:boundednessZ0} and \eqref{ineq:lipschitz3} that
\[
	\|\Lambda_0(s,(f_\e(n\tau,s),g_\e(n\tau,s))) + \tilde{\Lambda}_0(n\tau, s,(f_\e(n\tau,s),g_\e(n\tau,s)),\e)\| \leq  \mathcal{M}_L (\|f_\e\| + \|g_\e\|)  + \e M
\]
for all $(s,\tau,\varepsilon) \in \R \times \R \times (0,\varepsilon_1]$. Then, from property $\ref{propertyD}$, it follows that
\[
	\|\Lambda_0(s,(f_\e(n\tau,s),g_\e(n\tau,s))) + \tilde{\Lambda}_0(n\tau, s,(f_\e(n\tau,s),g_\e(n\tau,s)),\e)\| \leq C_G \varepsilon,
\]
where $C_G:=2\mathcal{M}_L C_R M +M$. Hence, it is easy to see that
\begin{equation} \label{eq:boundGn}
	|G_n(\theta,\e)| \leq C_G T
\end{equation}
for all $\theta \in \R$ and all $\e \in (0,\e_1]$. 

Considering that
\[
\frac{s_\varepsilon(n,\omega\theta)-\omega \theta}{n \omega} = \varepsilon^\ell  \frac{T}{\omega}  +  \e^{\ell+1} \frac{G_n(\omega\theta,\e)}{\omega},
\]
and since the limit 
\[
\lim_{n \to \infty} \frac{s_\varepsilon(n,\omega\theta)-\omega\theta}{n\omega}
\]corresponding to the rotation number exists and does not depend on $\theta$, it is ensured that
\[
	G(\e): = \lim_{n \to \infty} G_n(\omega \theta,\e) 
\]
is well defined. Moreover, from \eqref{eq:boundGn}, it is clear that $|G(\varepsilon)|\leq C_G T$. Hence, it follows at once that
\[
\rho(\varepsilon) = \lim_{n \to \infty} \varepsilon^\ell \frac{T}{\omega}  +  \e^{\ell+1} \frac{G_n(\omega\theta,\e)}{\omega} = \varepsilon^\ell \frac{T}{\omega}  + \e^{\ell+1} \frac{G(\e)}{\omega} = \e^\ell \frac{T}{\omega}+  \mathcal{O}(\varepsilon^{\ell+1}).
\]
By Proposition \ref{corollaryregularity3} combined with the definition of $\tilde{p}_\e$, it follows that the family $\{\tilde{p}_\e\}_\e$ is continuous in the space of homeomorphisms of $\s^1$ with the $C^0$ topology. Hence, $\rho$ is continuous in $(0,\e_1]$. Since we also know that system \eqref{eq:systemextended} becomes $\tau'=1, \bx'=0$, when $\e=0$, it follows that $\rho(0)=0$, so that $\rho$ is actually continuous in $[0,\e_1]$. In particular, the relation $\omega \rho(\varepsilon) = \e^{\ell} T + \e^{\ell+1} G(\e)$ ensures that $G$ is also continuous in $(0,\e_1]$.

In order to prove the rest of the statement concerning Dynamics of Theorem \ref{maintheorem}, we will make use of the following result, which can be found in \cite[Theorem 6.1]{herman}.

\begin{theorem} \label{theoremHerman}
	Let $\gamma\geq 3$ and $D^\gamma(\s^1)$ be the class of $C^{\gamma}$-diffeomorphisms of the circle $\s^1$ endowed with the norm $C^\gamma$. Let $c:[a,b] \to D^\gamma(\s^1)$ be a continuous path satisfying: $c$ is of class $C^1$ if considered as a function on $D^0(\s^1)$. Let $\rho(\la)$ denote the rotation number of $c(\la)$, $\la\in[a,b]$. If $\rho(a)\neq \rho(b)$, then the Lebesgue measure $\lambda$ of the set
	\[
	\{x \in [a,b]: c(x) \;\text{is}\; C^{\gamma-2}\text{-conjugated to an irrational rotation}\}
	\]
	is strictly positive. Also, $\rho$ maps zero Lebesgue measure sets to zero Lebesgue measure sets.
\end{theorem}

Suppose that $r -\ell\geq 4$. Then, Propositions \ref{corollaryregularity2} and \ref{corollaryregularity3} ensure that $\e \mapsto \tilde{p}_\e$ satisfies the regularity conditions stated in Theorem \ref{theoremHerman} with $\gamma=r-\ell-1$.  Moreover, from the fact that $\omega\rho(\e) = \e^\ell T + \mathcal{O}(\e^{\ell+1})$, it is clear that there is an interval $[a,b] \subset (0,\e_1]$ such that $\rho(a) \neq \rho(b)$. Hence, there is a subset of $I\subset [a,b]$ of positive Lebesgue measure such that $\tilde{p}_\e$ is $C^{r-\ell-3}$-conjugated to an irrational rotation for all $\e \in I$.

\section{Invariant torus in 4D vector fields}\label{sec:app}
Consider the differential system \eqref{systemexample} under the assumptions established in subsection \ref{subsec:app}. By applying the cylindrical change of coordinates $(x,y,u,v)=(r\cos\theta,r\sin \theta,u,v)$, $r>0$, system \eqref{systemexample} becomes
\begin{equation} \label{eq:systemexampleaftercoordinatechange}
	\begin{alignedat}{2}
		\dot r& = & &\varepsilon ^N \left(\cos (\theta ) f_1(r \cos (\theta ),r \sin (\theta ),u,v)+ \sin (\theta ) f_2(r \cos (\theta ),r \sin (\theta ),u,v) \right)
		\\
		& & -&\frac{\varepsilon^{N+1}}{2}\, r^3 \mu  \left(r^2-\left(r^2+1\right) \cos (2 \theta )-1\right) + \mathcal{O}(\varepsilon^{N+2}), \\
		\dot \theta& = &1+  &\varepsilon^N \left(\frac{ \cos (\theta ) f_2(r \cos (\theta ),r \sin (\theta ),u,v)-\sin (\theta ) f_1(r \cos (\theta ),r \sin (\theta ),u,v)}{r}\right) \\
		& & - &\varepsilon^{N+1} \mu \left(r^2\sin (\theta ) \cos (\theta ) + r^4 \sin (\theta ) \cos (\theta ) \right) + \mathcal{O}(\varepsilon^{N+2}), \\
		\dot u& = & &\varepsilon ^N f_3(r \cos (\theta ),r \sin (\theta ),u,v) + \varepsilon^{N+1} r^2 \cos^2 (\theta) (u-u^3+v-u v^2 ) + \mathcal{O}(\varepsilon^{N+2}), \\
		\dot v& = & &\varepsilon ^N f_4(r \cos (\theta ),r \sin (\theta ),u,v) + \varepsilon^{N+1} r^2 \sin^2 (\theta) (v-u-u^2v -v^3 ) + \mathcal{O}(\varepsilon^{N+2}).\\ 
	\end{alignedat}
\end{equation}
Since $\dot \theta = 1 + \mathcal{O}(\varepsilon^2)>0$, it follows that $\dot \theta >0$ for $\varepsilon$ sufficiently small. Thus, we can take $\theta$ to be the independent variable, and system \eqref{eq:systemexampleaftercoordinatechange} becomes
\begin{equation} \label{eq:systemthetaindependent}
	\begin{aligned}
		&r' = \varepsilon^N R_N(\theta, r, u ,v) + \varepsilon^{N+1} R_{N+1} (\theta, r, u ,v) + \mathcal{O}(\varepsilon^{N+2}), \\
		&u' = \varepsilon^N U_N(\theta,r,u,v) + \varepsilon^{N+1} U_{N+1} (\theta,r,u,v) + \mathcal{O}(\varepsilon^{N+2}), \\
		&v' = \varepsilon^N V_N(\theta,r,u,v) + \varepsilon^{N+1} V_{N+1} (\theta,r,u,v) + \mathcal{O}(\varepsilon^{N+2}),
	\end{aligned} 
\end{equation}
where $'$ indicates derivative with respect to the variable $\theta$, and the functions $R_i$, $U_i$, and $V_i$, $i \in \{N,N+1\}$, are given by
\begin{equation}
	\begin{aligned}
		&R_N(\theta,r,u,v) = \cos (\theta ) f_1(r \cos (\theta ),r \sin (\theta ),u,v)+\sin (\theta ) f_2(r \cos (\theta ),r \sin (\theta ),u,v); \\
		&R_{N+1}(\theta,r,u,v) = \frac{1}{2} r^3 \mu \left(\left(r^2+1\right) \cos (2 \theta )-r^2+1\right); \\
		&U_N(\theta,r,u,v) = f_3(r \cos (\theta ),r \sin (\theta ),u,v); \\
		&U_{N+1}(\theta,r,u,v) = r^2 \cos ^2(\theta ) \left(-u^3-u v^2+u+v\right) ; \\
		&V_N(\theta, r ,u ,v) = f_4(r \cos (\theta ),r \sin (\theta ),u,v) ; \\
		&V_{N+1}(\theta, r, u ,v) = -r^2 \sin ^2(\theta ) \left(u^2 v+u+v^3-v\right).
	\end{aligned} 
\end{equation}
We remark that each of the functions defined above is $2\pi$-periodic in $\theta$. By defining $\bx=(r,u,v)$, system \eqref{eq:systemthetaindependent} can be written as 
\begin{equation}
	\bx' = \varepsilon^N F_N(\theta,\bx) + \varepsilon^{N+1} F_{N+1}(\theta,\bx) + \varepsilon^{N+2} \tilde{F} (\theta,\bx,\varepsilon),
\end{equation}
where 
\begin{equation}
	F_i(\theta,\bx) = (R_i(\theta,\bx),U_i(\theta,\bx),V_i(\theta,\bx)).
\end{equation}
Using formulas \eqref{avfunc} and \eqref{yi}, we can calculate the Melnikov function of order $N$ for this system as
\begin{equation}
	\f_N(\bx) = \int_0^{2\pi} F_N(s,\bx) ds.
\end{equation} 

Since, by hypothesis, the average of functions $R_N$, $U_N$, and $V_N$ over $\theta \in [0,2\pi]$ vanish identically, then it follows that $\f_N=0$, so that formulas \eqref{avfunc} and \eqref{yi} provide
\begin{equation}
\begin{aligned}
	\f_{N+1}(\bx) = &\int_0^{2\pi} F_{N+1}(s,\bx) ds \\
	=& \left(\mu \frac{r^3}{2} \left(1-r^2\right), \frac{r^2}{2} \left(-u^3-u v^2+u+v\right),-\frac{r^2}{2}  \left(u^2 v+u+v^3-v\right)\right),
	\end{aligned}
\end{equation} 
because $F_i=0$ for all $i \in \{1,2,\ldots,N-1\}$ and $y_1=0$ in this case. Thus, since it is clear that $\f_i =0$ for all $i \in \{1,2,\ldots,N-1\}$, it follows from Proposition \ref{prop1} that 
\[
\bg_{N+1} (\bx) = \frac{1}{2\pi} \f_{N+1} (\bx).
\]

Let us prove that the guiding system $\bx' = {\textbf g}_{N+1}(\bx)$ has a hyperbolic limit cycle. First, observe that the curve 
\[
\gamma(t)= \left(1,\cos\left(\frac{t}{4\pi}\right),-\sin\left(\frac{t}{4\pi}\right)\right)
\]
satisfies
\[
\gamma' (t) = \left(0, -\frac{1}{4\pi} \sin \left(\frac{t}{4\pi}\right), -\frac{1}{4\pi} \cos \left(\frac{t}{4\pi}\right)\right) = {\textbf g}_{N+1}(\gamma(t)),
\]
and is therefore a $8\pi^2$-periodic orbit of $\bx' = {\textbf g}_{N+1}(\bx)$. Define $\Gamma$ as the image of $\gamma(t)$. Notice that $\Gamma=\{1\}\times\mathbb{S}^1$.

In order to show that $\Gamma$ is indeed a hyperbolic limit cycle, we shall find the eigenvalues of the Poincaré map $P$ associated to it. Observe that 
\[
\text{div} \, {\textbf g}_{N+1} (\bx) = -\mu \frac{5r^4}{4\pi}  +\frac{r^2}{4\pi} \left(2+3\mu -4u^2-4v^2\right).
\]
By \cite[Corollary 12.5]{teschl}, we know that the determinant of the derivative of $P$ at a point $\bx_0$ in the periodic orbit $\Gamma$ is equal to the determinant of the monodromy matrix associated to $\Gamma$. Thus, by Liouville's formula, we have
\begin{equation} \label{eq:determinantdp}
\det(DP(\bx_0)) = \exp \int_0^{8\pi^2} \text{div} \, {\textbf g}_{N+1} (\gamma(s)) \, ds = e^{-4\pi(1+\mu)}.
\end{equation}

Since the surface given by $r=1$ is an invariant manifold for this system, we can also study $\Gamma$ as a periodic orbit of the system $\bx' = {\textbf g}_{N+1}(\bx)$ restricted to such surface, which is the planar system $(u',v') = \bar{{\textbf g}}_{n+1}(u,v)$, given by
\begin{equation}
	\begin{aligned}
		&u' = \frac{1}{4\pi}(-u^3-uv^2+u+v), \\
		&v' = -\frac{1}{4\pi} (u^2 v+u+v^3-v).
	\end{aligned}
\end{equation}
Let $L$ be the intersection of the surface $r=1$ with the transversal section corresponding to the Poincaré map $P$. Then, once again by \cite[Corollary 12.5]{teschl}, the determinant of derivative of the restriction $P|_L$ at $\bx_0$ is given by
\[
\det(D(P|_L)(\bx_0)) = \exp \int_0^{8\pi^2} \text{div} \, \bar{{\textbf g}}_{N+1} \left(\cos\left(\frac{s}{4\pi}\right),-\sin\left(\frac{s}{4\pi}\right)\right) ds = e^{-4\pi}.
\]
Since $D(P|_L)(\bx_0)$ acts on a one-dimensional space, it follows that its eigenvalue is equal to $e^{-4\pi}$. 

We have thus found one of the eigenvalues of $DP(\bx_0)$, to wit, $e^{-4\pi}<1$. In order to find the other, it suffices to notice that the determinant of $DP(\bx_0)$ must be equal to the product of its two eigenvalues. Therefore, it follows from \eqref{eq:determinantdp} that the other eigenvalue is $e^{-4\pi \mu}\neq1$. Hence, it follows that $\Gamma$ is a hyperbolic limit cycle and that the eigenvalues of the derivative of the Poincaré map associated to it are $\lambda_1=e^{-4\pi}$ and $\lambda_2=e^{-4\pi\mu}$.

Thus, Theorem \ref{maintheorem} ensures that there is $\varepsilon_0>0$ such that, for each $\varepsilon \in [0,\varepsilon_0]$, system 
\[
\theta' = 1, \quad \bx' = \varepsilon^N F_N(\theta,\bx) + \varepsilon^{N+1} F_{N+1}(\theta,\bx) + \varepsilon^{N+2} \tilde{F} (\theta,\bx,\varepsilon)
\]
admits an invariant torus $M_\varepsilon$ of class $C^{r-3}$. Moreover, $M_\varepsilon$ converges to $\s^1 \times \Gamma$ as $\varepsilon \to 0$. The stability of $M_\varepsilon$ is controlled by the parameter $\mu$. If $\mu=1$, then $M_\varepsilon$ is asymptotically stable, since $\mathcal{S}^{V_s}_{M_\varepsilon}$ locally becomes a neighborhood of $M_\varepsilon$. If, on the other hand, $\mu=-1$, then $\mathcal{S}^{V_s}_{M_\varepsilon}$ is locally a 3-dimensional manifold embedded in $\R^4$.

Transforming back to the original coordinates, we obtain, for each $\varepsilon \in [0,\varepsilon_0]$, an invariant torus $\mathbb{T}_\varepsilon$ converging as $\varepsilon \to 0$ to the torus $\mathbb{T}=\mathbb{S}^1\times\mathbb{S}^1$ parameterized by $(\theta, t)\in [0,2\pi] \times [0,2\pi] \mapsto \left(\cos \theta, \sin \theta, \cos t, -\sin t\right)$.

\section*{Appendix}
   
\subsection{Proof of Proposition \ref{corollaryregularity1}}

Let $\varepsilon \in (0,\varepsilon_1]$ be fixed throughout all the proof. We shall prove that $M_\varepsilon$ can be parameterized by a $C^p$ function $\alpha_\varepsilon(t,\theta)$. In fact, let $\tau \mapsto \varphi_\varepsilon(\tau,t_0,\theta_0,y_0,z_0)$ be the flow of system \eqref{systemlemmaautonomized} satisfying $\varphi_\varepsilon(0,t_0,\theta_0,y_0,z_0)=(t_0,\theta_0,y_0,z_0)$. Then, results about smooth dependence on initial conditions (see, for instance, \cite[Corollary 4.1 of Chapter V]{Hartman1964OrdinaryDE})  ensure that $\varphi_\varepsilon$ is of class $C^p$. Define $\alpha_\varepsilon:\mathbb{R}\times\mathbb{R}\to \mathbb{R} \times \R \times B_m(0,\sigma_0) \times B_n(0,\sigma_0)$ by
	\[
	\alpha_\varepsilon(t,\theta): = \varphi_\varepsilon(t,0,\theta,f_\e(0,\theta),g_\e(0,\theta)).
	\]
	Observe that statement $(e)$ of Lemma \ref{lemmahale} guarantees that $\alpha_\varepsilon$ is of class $C^p$.
	Let us prove that $\alpha_\varepsilon$ is injective and that its image is $M_\varepsilon$. 
	
	In order to prove that $\alpha_\varepsilon$ is injective, let $(t_1,\theta_1), (t_2,\theta_2) \in \mathbb{R}\times\mathbb{R}$ be such that $\alpha_\varepsilon(t_1,\theta_1) = \alpha_\varepsilon(t_2,\theta_2)$. Define the functions $t_\varepsilon(\tau,t_0,\theta_0,y_0,z_0)$, $\theta_\varepsilon(\tau,t_0,\theta_0,y_0,z_0)$, $y_\varepsilon(\tau,t_0,\theta_0,y_0,z_0)$, and $z_\varepsilon(\tau,t_0,\theta_0,y_0,z_0)$ as being the components of the flow $\varphi_\e(\tau,t_0,\theta_0,y_0,z_0)$.
	Then, it is clear by \eqref{systemlemmaautonomized} that $t_\varepsilon(\tau,t_0,\theta_0,y_0) = t_0 +\tau$. Hence, $\alpha(t_1,\theta_1) = \alpha(t_2,\theta_2)$ implies at once that $t_1+t_0=t_2+t_0$, that is, $t_1=t_2$. Therefore, the uniqueness of the flow $\varphi_\varepsilon$ ensures that the points $(0,\theta_1,f_\e(0,\theta_1),g_\e(0,\theta_1))$ and $(0,\theta_2,f_\e(0,\theta_2),g_\e(0,\theta_2))$ must be the same. Thus, $\theta_1=\theta_2$, and $\alpha_\varepsilon$ is indeed injective. 
	
	To show that the image of $\alpha_\varepsilon$ is $M_\varepsilon$, we first observe that, since $(0,\theta,f_\e(0,\theta),g_\e(0,\theta)) \in M_\varepsilon$ for all $\theta \in \mathbb{R}$ and $M_\varepsilon$ is invariant, it follows that $\alpha_\varepsilon(t,\theta) = \varphi_\varepsilon(t,0,\theta,f_\e(0,\theta),g_\e(0,\theta)) \in M_\varepsilon$ for all $(t,\theta) \in \mathbb{R}\times\mathbb{R}$, i.e., the image of $\alpha_\varepsilon$ is contained in $M_\varepsilon$. On the other hand, every point in $M_\varepsilon$ is, by definition, of the form $(t,\theta,f_\e(t,\theta),g_\e(t,\theta))$ for some $(t,\theta) \in \mathbb{R}\times\mathbb{R}$. By properties of the flow, defining $\tilde{\theta}_\varepsilon=\theta_\varepsilon(-t,t,\theta,f_\e(t,\theta),g_\e(t,\theta))$, we have
	\[
	(t,\theta,f_\e(t,\theta),g_\e(t,\theta)) = \varphi_\varepsilon\big(t,0,\tilde{\theta}_\varepsilon,f_\e(0,\tilde{\theta}_\e),g_\e(0,\tilde{\theta}_\e)\big) = \alpha_\varepsilon(t,\tilde{\theta}_\varepsilon),
	\]
	which implies that $M_\varepsilon$ is contained in the image of $\alpha_\varepsilon$. Thus, we have proved that $\alpha_\e$ is an injective function of class $C^p$ whose image is $M_\varepsilon$ and, therefore, is a $C^p$ parametrization of $M_\e$. This ensures that $M_\e$ is of class $C^p$.
		
	We shall now prove that $f_\e$ and $g_\varepsilon$ are of class $C^p$. In order to do so, we remark that $\big(t,\theta,f_\e(t,\theta),g_\varepsilon(t,\theta)) \in M_\e$ for every $(t,\theta)\in\R\times\R$. Then, for each $(t,\theta)\in\R\times\R$, there is $(\tilde{t},\tilde{\theta}) \in \mathbb{R}\times\mathbb{R}$ such that
	\begin{equation} \label{eq:parametrizationq}
		\begin{aligned}
		\big(t,\theta,f_\e(t,\theta),g_\varepsilon(t,\theta))=  \alpha_\varepsilon(\tilde{t},\tilde{\theta}\big) = \big(\tilde{t},\theta_\varepsilon(\tilde{t},0,\tilde{\theta},f_\e(0,\tilde{\theta}), g_\varepsilon(0,\tilde{\theta})),y_\varepsilon(\tilde{t},0,\tilde{\theta},f_\e(0,\tilde{\theta}),g_\varepsilon(0,\tilde{\theta})),\quad\\ z_\varepsilon(\tilde{t},0,\tilde{\theta},f_\e(0,\tilde{\theta}),g_\varepsilon(0,\tilde{\theta}))\big).
		\end{aligned}
	\end{equation}
	Define the function $h(\tilde{t},\tilde{\theta}) = \big(\tilde{t},\theta_\varepsilon(\tilde{t},0,\tilde{\theta},f_\e(0,\tilde{\theta}), g_\varepsilon(0,\tilde{\theta}))\big)$. It is then clear that the inverse of $h$ exists and is given by
	\[
	h^{-1}(t,\theta) = \big(t,\theta_\varepsilon(-t,t,\theta,f_\e(t,\theta),g_\varepsilon(t,\theta))\big).
	\]
	Now, by taking $(\tilde{t}(t,\theta), \tilde{\theta}(t,\theta)) = h^{-1}(t,\theta)$, we get from \eqref{eq:parametrizationq} that
	\[
	\begin{aligned}
		f_\varepsilon(t,\theta)=y_\varepsilon(\tilde{t}(t,\theta),0,\tilde{\theta},f_\e(0,\tilde{\theta}(t,\theta)),g_\varepsilon(0,\tilde{\theta}(t,\theta)), \\
		g_\varepsilon(t,\theta)=z_\varepsilon(\tilde{t}(t,\theta),0,\tilde{\theta},f_\e(0,\tilde{\theta}(t,\theta)),g_\varepsilon(0,\tilde{\theta}(t,\theta)).
	\end{aligned}
	\]
	Thus, since $y_\varepsilon$, $z_\e$, and $\theta \mapsto g_\varepsilon(0,\theta)$ are of class $C^p$, in order to prove that $f_\e$ and $g_\varepsilon$ are of class $C^p$ it only remains to show that $h^{-1}$ is of class $C^p$. First, observe that $h$ is clearly of class $C^p$, because $\tilde{\theta} \mapsto f_\varepsilon(0,\tilde{\theta})$ and $\tilde{\theta} \mapsto g_\varepsilon(0,\tilde{\theta})$ are of class $C^p$ by statement $(e)$ of Lemma \ref{lemmahale}. From the Inverse Function Theorem, it suffices then to prove that the derivative of $h$ is non-singular at every point $(t,\theta) \in \mathbb{R}\times\mathbb{R}$. Observe that
	\[
	Dh(\tilde{t},\tilde{\theta}) = 
	\left[\begin{array}{cc}
		1 & 0 \\
		1+\zeta_0(\tilde{t},\theta_\varepsilon,f_\e(\tilde{t},\theta_\e),g_\varepsilon(\tilde{t},\theta_\varepsilon),\e) & \frac{\partial \theta_\varepsilon}{\partial \theta_0} + \frac{\partial \theta_\varepsilon}{\partial y_0} \cdot \frac{\partial f_\varepsilon}{\partial \theta}(0,\tilde{\theta}) +  \frac{\partial \theta_\varepsilon}{\partial z_0} \cdot \frac{\partial g_\varepsilon}{\partial \theta}(0,\tilde{\theta})
	\end{array}\right],
	\]
	where the argument of $\theta_\varepsilon$ and its partial derivatives is $(\tilde{t},0,\tilde{\theta},f_\e(0,\tilde{\theta}),g_\varepsilon(0,\tilde{\theta}))$ and has been omitted for conciseness. Thus, $Dh(\tilde{t},\tilde{\theta})$ is non-singular if, and only if, 
	\begin{equation} \label{eq:conditioncorollary}
		\begin{aligned}
		&\quad \;\frac{\partial \theta_\varepsilon}{\partial \theta_0}(\tilde{t},0,\tilde{\theta},f_\e(0,\tilde{\theta}),g_\varepsilon(0,\tilde{\theta})) \\&+ \frac{\partial \theta_\varepsilon}{\partial y_0}(\tilde{t},0,\tilde{\theta},f_\e(0,\tilde{\theta}),g_\varepsilon(0,\tilde{\theta})) \cdot \frac{\partial f_\varepsilon}{\partial \theta}(0,\tilde{\theta}) \\&+ \frac{\partial \theta_\varepsilon}{\partial z_0}(\tilde{t},0,\tilde{\theta},f_\e(0,\tilde{\theta}),g_\varepsilon(0,\tilde{\theta})) \cdot \frac{\partial g_\varepsilon}{\partial \theta}(0,\tilde{\theta}) \neq 0.
		\end{aligned}
	\end{equation}
	
	The matrix 
	\[
	 \mathcal{M}(\tau,t_0,\theta_0,y_0,z_0): = \left[\begin{array}{cccc}
		\frac{\partial t_\varepsilon}{\partial t_0}& \frac{\partial t_\varepsilon}{\partial \theta_0} & \frac{\partial t_\varepsilon}{\partial y_0} & \frac{\partial t_\varepsilon}{\partial z_0}\\
		\frac{\partial \theta_\varepsilon}{\partial t_0}& \frac{\partial \theta_\varepsilon}{\partial \theta_0} & \frac{\partial \theta_\varepsilon}{\partial y_0} & \frac{\partial \theta_\varepsilon}{\partial z_0} \\
		\frac{\partial y_\varepsilon}{\partial t_0}& \frac{\partial y_\varepsilon}{\partial \theta_0} & \frac{\partial y_\varepsilon}{\partial y_0} & \frac{\partial y_\varepsilon}{\partial z_0} \\
		\frac{\partial z_\varepsilon}{\partial t_0}& \frac{\partial z_\varepsilon}{\partial \theta_0} & \frac{\partial z_\varepsilon}{\partial y_0} & \frac{\partial z_\varepsilon}{\partial z_0}
	\end{array}\right],
	\]
	where the argument of each entry is given by $(\tau,t_0,\theta_0,y_0,z_0)$ is a fundamental solution of the first variational equation associated to \eqref{systemlemmaautonomized}. Thus, $\mathcal{M}(\tau,t_0,\theta_0,y_0,z_0)$ is invertible. Moreover, since $t_\varepsilon(\tau,t_0,\theta_0,y_0) = \tau+t_0$, it follows that
	\[
	\mathcal{M}(\tau,t_0,\theta_0,y_0,z_0) = \left[\begin{array}{cccc}
		1 &0 &0 &0\\
		\frac{\partial \theta_\varepsilon}{\partial t_0}& \frac{\partial \theta_\varepsilon}{\partial \theta_0} & \frac{\partial \theta_\varepsilon}{\partial y_0} & \frac{\partial \theta_\varepsilon}{\partial z_0} \\
		\frac{\partial y_\varepsilon}{\partial t_0}& \frac{\partial y_\varepsilon}{\partial \theta_0} & \frac{\partial y_\varepsilon}{\partial y_0} & \frac{\partial y_\varepsilon}{\partial z_0} \\
		\frac{\partial z_\varepsilon}{\partial t_0}& \frac{\partial z_\varepsilon}{\partial \theta_0} & \frac{\partial z_\varepsilon}{\partial y_0} & \frac{\partial z_\varepsilon}{\partial z_0}
	\end{array}\right],
	\]
	Hence, we conclude that 
	\[
	\mathcal{N}(\tau,t_0,\theta_0,y_0,z_0):=\left[\begin{array}{ccc}
		\frac{\partial \theta_\varepsilon}{\partial \theta_0}(\tau,t_0,\theta_0,y_0,z_0) & \frac{\partial \theta_\varepsilon}{\partial y_0}(\tau,t_0,\theta_0,y_0,z_0) & \frac{\partial \theta_\varepsilon}{\partial z_0}(\tau,t_0,\theta_0,y_0,z_0)\\
		\frac{\partial y_\varepsilon}{\partial \theta_0}(\tau,t_0,\theta_0,y_0,z_0) & \frac{\partial y_\varepsilon}{\partial y_0}(\tau,t_0,\theta_0,y_0,z_0) & \frac{\partial y_\varepsilon}{\partial z_0}(\tau,t_0,\theta_0,y_0,z_0)\\
		\frac{\partial z_\varepsilon}{\partial \theta_0}(\tau,t_0,\theta_0,y_0,z_0) & \frac{\partial z_\varepsilon}{\partial y_0}(\tau,t_0,\theta_0,y_0,z_0) & \frac{\partial z_\varepsilon}{\partial z_0}(\tau,t_0,\theta_0,y_0,z_0)
	\end{array}\right]
	\]
	is invertible for all $(\tau,t_0,\theta_0,y_0,z_0) \in [-\Omega,\Omega] \times \mathbb{R}\times\mathbb{R} \times B_m(0,\rho) \times B_m(0,\rho)$, where $[-\Omega,\Omega]$ is the maximal interval where the flow is defined. In particular, if $t_0=0$, $\theta_0=\tilde{\theta}$, $y_0 =f_\e(0,\tilde{\theta})$, and $z_0=g_\varepsilon(0,\tilde{\theta})$, then the flow is defined for all $\tau \in \mathbb{R}$, and it follows that $\mathcal{N}(\tilde{t},0,\tilde{\theta},f_\e(0,\tilde{\theta}),g_\e(0,\tilde{\theta}))$
	is invertible for all $(\tilde{t},\tilde{\theta}) \in \mathbb{R}\times\mathbb{R}$. Thus, the product
	\[	\begin{aligned}
		\mathcal{N}(\tilde{t},0,\tilde{\theta},f_\e(0,\tilde{\theta}),g_\e(0,\tilde{\theta})) \cdot
		\left[\begin{array}{c}
			1 \\
			\frac{\partial f_\varepsilon}{\partial \theta}(0,\tilde{\theta}) \\
			\frac{\partial g_\varepsilon}{\partial \theta}(0,\tilde{\theta})
		\end{array}\right]	
		= 
		\left[\begin{array}{c}
			\frac{\partial \theta_\varepsilon}{\partial \theta_0} + \frac{\partial \theta_\varepsilon}{\partial y_0} \cdot \frac{\partial f_\varepsilon}{\partial \theta} +\frac{\partial \theta_\varepsilon}{\partial z_0}\cdot \frac{\partial g_\varepsilon}{\partial \theta}\\
			\frac{\partial y_\varepsilon}{\partial \theta_0} + \frac{\partial y_\varepsilon}{\partial y_0} \cdot \frac{\partial f_\varepsilon}{\partial \theta} +\frac{\partial y_\varepsilon}{\partial z_0}\cdot \frac{\partial g_\varepsilon}{\partial \theta} \\
			\frac{\partial z_\varepsilon}{\partial \theta_0} + \frac{\partial z_\varepsilon}{\partial y_0} \cdot \frac{\partial f_\varepsilon}{\partial \theta} +\frac{\partial z_\varepsilon}{\partial z_0}\cdot \frac{\partial g_\varepsilon}{\partial \theta}
		\end{array}\right]
	\end{aligned} \]
	cannot vanish, where the arguments of the derivatives of $\theta_\e$, $y_\e$ and $z_\e$, as well as the arguments of the derivatives of $f_\e$ and $g_\e$, have been omitted, but should read respectively as $(\tilde{t},0,\tilde{\theta},f_\e(0,\tilde{\theta}),g_\e(0,\tilde{\theta}))$ and $(0,\tilde{\theta})$. Suppose, by contradiction, that \eqref{eq:conditioncorollary} does not hold at $(\tilde{t}_*,\tilde{\theta}_*) \in \R^2$, so that the first line of product above vanishes. Let us show that this implies that the product vanishes altogether.
	
	In fact, observe that the invariance of $M_\varepsilon$ ensures that, for all $(\tilde{t},\tilde{\theta}) \in \R^2$,
	\[
	\begin{aligned}
	y_\varepsilon(\tilde{t},0,\tilde{\theta},f_\e(0,\tilde{\theta}),g_\varepsilon(0,\tilde{\theta})) = f_\varepsilon(\tilde{t},\theta_\varepsilon(\tilde{t},0,\tilde{\theta},f_\e(0,\tilde{\theta}),g_\varepsilon(0,\tilde{\theta}))), \\
	z_\varepsilon(\tilde{t},0,\tilde{\theta},f_\e(0,\tilde{\theta}),g_\varepsilon(0,\tilde{\theta})) = g_\varepsilon(\tilde{t},\theta_\varepsilon(\tilde{t},0,\tilde{\theta},f_\e(0,\tilde{\theta}),g_\varepsilon(0,\tilde{\theta}))).
	\end{aligned}
	\]
	Thus, differentiating with respect to $\tilde{\theta}$, it follows that
	\begin{equation*}
		\begin{aligned}
			\frac{\partial y_\varepsilon}{\partial \theta_0} + \frac{\partial y_\varepsilon}{\partial y_0} \cdot \frac{\partial f_\varepsilon}{\partial \theta} +\frac{\partial y_\e}{\partial z_0} \cdot \frac{\partial g_\e}{\partial \theta} =  \frac{\partial f_\varepsilon}{\partial \theta}  \left[	\frac{\partial \theta_\varepsilon}{\partial \theta_0} + \frac{\partial \theta_\varepsilon}{\partial y_0} \cdot \frac{\partial f_\varepsilon}{\partial \theta} +\frac{\partial \theta_\e}{\partial z_0} \cdot \frac{\partial g_\e}{\partial \theta}\right]
		\end{aligned}
	\end{equation*}
	and 
	\begin{equation*}
		\begin{aligned}
			\frac{\partial z_\varepsilon}{\partial \theta_0} + \frac{\partial z_\varepsilon}{\partial y_0} \cdot \frac{\partial f_\varepsilon}{\partial \theta} +\frac{\partial z_\e}{\partial z_0} \cdot \frac{\partial g_\e}{\partial \theta} =  \frac{\partial g_\varepsilon}{\partial \theta}  \left[	\frac{\partial \theta_\varepsilon}{\partial \theta_0} + \frac{\partial \theta_\varepsilon}{\partial y_0} \cdot \frac{\partial f_\varepsilon}{\partial \theta} +\frac{\partial \theta_\e}{\partial z_0} \cdot \frac{\partial g_\e}{\partial \theta}\right]
		\end{aligned}
	\end{equation*}
	where, once again, the arguments of the derivatives of $\theta_\e$, $y_\e$ and $z_\e$, as well as the arguments of the derivatives of $f_\e$ and $g_\e$, have been omitted, but should read respectively as $(\tilde{t},0,\tilde{\theta},f_\e(0,\tilde{\theta}),g_\e(0,\tilde{\theta}))$ and $(0,\tilde{\theta})$.
	
	Since we assumed that \eqref{eq:conditioncorollary} does not hold at $(\tilde{t}_*,\tilde{\theta}_*)$, it follows at once that
	\[
	\left[\begin{array}{c}
		\frac{\partial \theta_\varepsilon}{\partial \theta_0} + \frac{\partial \theta_\varepsilon}{\partial y_0} \cdot \frac{\partial f_\varepsilon}{\partial \theta} +\frac{\partial \theta_\varepsilon}{\partial z_0}\cdot \frac{\partial g_\varepsilon}{\partial \theta}\\
		\frac{\partial y_\varepsilon}{\partial \theta_0} + \frac{\partial y_\varepsilon}{\partial y_0} \cdot \frac{\partial f_\varepsilon}{\partial \theta} +\frac{\partial y_\varepsilon}{\partial z_0}\cdot \frac{\partial g_\varepsilon}{\partial \theta} \\
		\frac{\partial z_\varepsilon}{\partial \theta_0} + \frac{\partial z_\varepsilon}{\partial y_0} \cdot \frac{\partial f_\varepsilon}{\partial \theta} +\frac{\partial z_\varepsilon}{\partial z_0}\cdot \frac{\partial g_\varepsilon}{\partial \theta}
	\end{array}\right] = 0
	\]
	when the arguments of the derivatives of $\theta_\e$, $y_\e$ and $z_\e$ are given by $(\tilde{t}_*,0,\tilde{\theta}_*,f_\e(0,\tilde{\theta}_*),g_\e(0,\tilde{\theta}_*))$ and the arguments of the derivatives of $f_\e$ and $g_\e$ are given by $(0,\tilde{\theta}_*)$. Therefore,
 	\[
 	\mathcal{N}(\tilde{t}_*,0,\tilde{\theta}_*,f_\e(0,\tilde{\theta}_*),g_\e(0,\tilde{\theta}_*)) \cdot
 	\left[\begin{array}{c}
 		1 \\
 		\frac{\partial f_\varepsilon}{\partial \theta}(0,\tilde{\theta}_*) \\
 		\frac{\partial g_\varepsilon}{\partial \theta}(0,\tilde{\theta}_*)
 	\end{array}\right] = 0 .
 	\]
 	As remarked above, this would imply that $\mathcal{N}(\tilde{t}_*,0,\tilde{\theta}_*,f_\e(0,\tilde{\theta}_*),g_\e(0,\tilde{\theta}_*))$, which we have proved to be invertible, is not invertible. Since we have reached a contradiction, it is proved that $h^{-1}$ is indeed of class $C^p$, which concludes the proof of the Proposition.

\subsection{Proof of Proposition \ref{corollaryregularity2}}

It is not difficult to see that we can assume without loss of generality that the function $L$ appearing in hypothesis (ii) of the Lemma satisfies: $L(\e,\sigma,\mu)\geq \e$ for all $(\e,\sigma,\mu) \in (0,\e_0] \times [0,\rho_1) \times [0,\rho_2)$. 
	
	Let $D(\e)$, $\Delta(\e)$, $J_1$, $J_2$ be given as in the proof of Lemma \ref{lemmahale}. It is clear that there are $K>0$ and $\alpha>0$ such that $\|J_i(t)\| \leq K e^{-\alpha|t|}$ for all $i \in \{1,2\}.$ By choosing $\e_1$ to be sufficiently small, we can then ensure that the following inequalities hold for all $\e \in (0,\e_1]$:
	\begin{itemize}
		\item $\Delta(\e)<\frac{1}{2}$ and $D(\e) < \rho:=\min(\rho_1,\rho_2)$;
		\item $32L(\e,D(\e),D(\e)) <\alpha;$
		\item $64 K L(\e,D(\e),D(\e))<\alpha$.
	\end{itemize}
	
	Let $\mathcal{P}_\omega(D,\Delta)$, $\mathcal{A}_\omega(D,\Delta)$, and the operator $S^\e$ be given as in the proof of Lemma \ref{lemmahale}. For each $\e \in (0,\e_1]$, define the sequence $(P_k, A_k)_{k \in \mathbb{N}}$, where
	$ P_k: (0,\e_1) \to \mathcal{P}_\omega(D,\Delta)$ and $ A_k: (0,\e_1) \to \mathcal{A}_\omega(D,\Delta)$ are functions of class $C^1$ given by:
	\begin{itemize}
		\item $(P_0(\e),A_0(\e)) = (0,0)$ for all $\e \in (0,\e_1)$;
		\item $(P_{k+1}(\e),A_{k+1}(\e)) = S^\e(P_k(\e),A_k(\e))$ for all $k \in \mathbb{N}$ and all $\e \in (0,\e_1)$.
	\end{itemize}
	From Lemma \ref{lemmahale}, it is clear that this sequence satisfies
	\begin{equation} \label{eq:pointwiselimit}
		\lim_{k \to \infty} (P_k(\e),A_k(\e)) = (f_\e,g_\e).
	\end{equation}
	
		Let $a,b \in (0,\e_1)$ be fixed. Effecting cumbersome calculations, which are very similar to those presented in the proofs of Propositions \ref{propositiondeltatheta} and \ref{corollaryregularity3}, and for this reason are omitted, we can show that the following hold for all $\e \in [a,b]$ and all $k \in \mathbb{N}$:
	\begin{enumerate}
		\item $\|P_{k+2}(\e) - P_{k+1}(\e) \|  + \|A_{k+2}(\e) - A_{k+1}(\e) \| \leq \frac{1}{4} \left[\|P_{k+1}(\e) - P_k(\e) \|  + \|A_{k+1}(\e) - A_k(\e) \|\right].$
		\item There is $C_1>0$ such that 
		\[
		\begin{aligned}
			\left\|\frac{\partial P_{k+2}(\e)}{\partial \theta} -\frac{\partial P_{k+1}(\e)}{\partial \theta}\right\| + \left\|\frac{\partial A_{k+2}(\e)}{\partial \theta} -\frac{\partial A_{k+1}(\e)}{\partial \theta}\right\| \leq   C_1 \left[\|P_{k+1}(\e) - P_k(\e) \|  + \|A_{k+1}(\e) -A_k(\e) \|\right] \\+\frac{1}{4}\left[	\left\|\frac{\partial P_{k+1}(\e)}{\partial \theta} -\frac{\partial P_{k}(\e)}{\partial \theta}\right\| + \left\|\frac{\partial A_{k+1}(\e)}{\partial \theta} -\frac{\partial A_{k}(\e)}{\partial \theta}\right\|\right].
		\end{aligned}
		\]
		\item There is $C_2>0$ such that 
		\[
		\begin{aligned}
			\left\| P'_{k+2}(\e) -P'_{k+1}(\e)\right\| + \left\| A'_{k+2}(\e) -A'_{k+1}(\e)\right\| \leq   C_2 \Bigg[\|P_{k+1}(\e) - P_k(\e) \|  + \|A_{k+1}(\e) - A_k(\e) \| \\+\left. \left\|\frac{\partial P_{k+1}(\e)}{\partial \theta} -\frac{\partial P_{k}(\e)}{\partial \theta}\right\| + \left\|\frac{\partial A_{k+1}(\e)}{\partial \theta} -\frac{\partial A_{k}(\e)}{\partial \theta}\right\|\right] \\
			+ \frac{1}{4}\left[	\left\| P'_{k+1}(\e) -P'_{k}(\e)\right\| + \left\| A'_{k+1}(\e) -A'_{k}(\e)\right\|\right].
		\end{aligned}
		\]
	\end{enumerate}
	
	From those inequalities, it follows easily that there is $C>0$ such that
	
	\[
			\sup_{\e \in [a,b]}\left\| P'_{k+1}(\e) -P'_{k}(\e)\right\| + \left\| A'_{k+1}(\e) -A'_{k}(\e)\right\| \leq \frac{C}{2^k}
	\]
	for all $k \in \mathbb{N}$. Hence, the sequence $(P'_k\,,A'_k)_{k \in \mathbb{N}}$ converges uniformly on $[a,b]$. Since $a$ and $b$ were arbitrary, this implies that $(P'_k\,,A'_k)_{k \in \mathbb{N}}$ converges uniformly on compact subsets of $(0,\e_1)$. Therefore, considering \eqref{eq:pointwiselimit}, it follows that the function $c$ given in the statement of this Lemma is of class $C^1$ (see, for instance, \cite[Theorem 85, Chapter 1]{banachbook}), concluding the proof.
	
\subsection{Proof of Proposition \ref{propositiondeltatheta}}

Let $D(\e)$, $\Delta(\e)$, $J_1$, $J_2$, and $T_{F,G}$ be given as in the proof of Lemma \ref{lemmahale}. It is clear that there are $K>0$ and $\alpha>0$ such that $\|J_i(t)\| \leq K e^{-\alpha|t|}$ for all $i \in \{1,2\}.$ For convenience, we will denote $L(\e,D(\e),D(\e))$ by $L(\e)$ throughout the proof. As in Proposition \ref{corollaryregularity2}, we assume that the function $L$ appearing in hypothesis (ii) of the Lemma satisfies: $L(\e,\sigma,\mu)\geq \e$ for all $(\e,\sigma,\mu) \in (0,\e_0] \times [0,\rho_1) \times [0,\rho_2)$. 
 	
 	Let $(P_k,A_k)_{k \in \mathbb{N}}$ be the sequence defined in Proposition \ref{corollaryregularity2}. For convenience, define
	\begin{itemize}
		\item $T_{k}^{x,t,\e}(\theta): = T_{P_k(\e),A_k(\e)}(t+x,t,\theta,\e)$;
		\item $\Lambda_k^{x,t,\e}(\theta)=(t+x,\theta,P_k(\e)(t+x,\theta),A_k(\e)(t+x,\theta),\e)$;
		\item$\zeta_{i,k}^{x,t,\e}(\theta):=\zeta_i(t+x,\theta,P_k(\e)(t+x,\theta),A_k(\e)(t+x,\theta),\e) = \zeta_i \circ \Lambda_k^{x,t,\e}(\theta)$.
	\end{itemize}
	We will prove by induction on $q$ that, if $\e_1$ is sufficiently small, the following hold for each $q \in \{1,\ldots, p+1\}$:
	\begin{enumerate}[label=P\arabic*.)]
		\item \label{property0 1} There is $N_q \in \mathbb{N}$ and, for each $[a,b] \subset (0,\e_1]$, there is $C_{0,q}>0$ such that
		\[
		\left| \left(T_k^{x,t,\e}\right)^{(q)}(\theta) \right| \leq C_{0,q}\, e^{N_{q} L(\e)(1+2\Delta(\e)) |x|},
		\]
		for all $k \in \mathbb{N}$ and all $(x,t,\theta,\e) \in \R \times \R \times \R \times [a,b]$.
		\item \label{property0 2} For each $[a,b] \subset (0,\e_1]$, there is $C_{1,q}>0$ such that 
		\[
		\left\|\frac{\partial^q P_k(\e)}{\partial \theta^q} \right\| + \left\|\frac{\partial^q A_k(\e)}{\partial \theta^q} \right\| \leq C_{1,q},
		\]
		for all $k \in \mathbb{N}$ and all $\e \in [a,b]$.
	\end{enumerate}
	
	Let us then consider the case $q=1$. Observe that, from the definition of $T_k^{x,t,\e}$, it follows that
	\begin{equation}\label{eq:partialtheta0partialtTk}
		\begin{aligned}
			\frac{\partial }{\partial \theta} \left(\frac{\partial T_{P_k(\e),A_k(\e)}}{\partial x}\right) (t+x,t,\theta,\e) = D\zeta_0 \left(\Lambda_k^{x,t,\e}\left(T_k^{x,t,\e}(\theta)\right)\right) \cdot \left(\Lambda_k^{x,t,\e}\right)' \left(T_k^{x,t,\e}(\theta)\right) \cdot \left(  T_k^{x,t,\e}\right)'(\theta).
		\end{aligned}
	\end{equation}
	Hence, we obtain by changing order of derivatives and integrating
	\[
	\left| \left( T_k^{x,t,\e}\right)'(\theta) - \left(T_k^{t,t,\e}\right)(\theta) \right| \leq \int_{0}^x L(\e) (1+2\Delta(\e)) \left|\left(T_k^{x,t,\e}\right)'(\theta) \right| dx.
	\]
	Since 
	\[
	\left(T_k^{t,t,\e}\right)'(\theta) =1,
	\]
	it follows by an application of Grönwall's inequality that
	\begin{equation}\label{ineq:boundednesspartialTkpartialtheta0}
		\left|\left(T_k^{x,t,\e}\right)'(\theta)  \right|  \leq e^{L(\e)(1+2\Delta(\e))|x|}.
	\end{equation}
	This proves property \ref{property0 1}. Property \ref{property0 2} follows directly, with $C_{1,1} = \sup_{\e \in [a,b]} \Delta(\e)$, from the fact that $(P_k,A_k) \in \mathcal{P}_\omega(D(\e),\Delta(\e)) \times \mathcal{A}_\omega(D(\e),\Delta(\e))$ for all $k \in \mathbb{N}$.

	Let $N \in \{2,\ldots,p+1\}$ be given and suppose that the Lemma is true for every $q \in \mathbb{N}$ such that $1\leq q \leq N-1$. We will show that the Lemma also holds for $q=N$.
	
	Henceforth, we will employ the index $i$ to denote any element of the set $\{0,1,2\}$, since the considerations done below are the same. By Faà di Bruno's formula, since $\zeta_{i,k}^{\e,x,t} =\zeta_i \circ \Lambda_k^{x,t,\e}$, it follows that
	\[
	\begin{aligned}
		\left(\zeta_{i,k}^{x,t,\e}\right)^{(q)} (\theta) = \sum_{j=1}^{q} D^{(j)} \zeta_i(\Lambda_k^{x,t,\e}(\theta)) \cdot B_{q,j} \left(\left(\Lambda_k^{x,t,\e}\right)'(\theta),\ldots, \left(\Lambda_k^{x,t,\e}\right)^{(q-j+1)}(\theta)\right), 
	\end{aligned}
	\]
	for each $q \in \{1,\ldots,p+1\}$, where $B_{q,j}$ is a Bell polynomial. Observe that $D^{(j)} \zeta_i(\Lambda_k^{x,t,\e}(\theta))$ is a symmetric multilinear map that can be thought of as being applied to a ``product" of vectors. Its application to a polynomial is simply a linear combination of different applications to such ``products". In particular, for $q=N$, we can write
	\[
	\begin{aligned}
		\left(\zeta_{i,k}^{x,t,\e}\right)^{(N)} (\theta) &= D^{(N)} \zeta_i(\Lambda_k^{x,t,\e}(\theta)) \cdot \left(\left(\Lambda_k^{x,t,\e}\right)'(\theta)\right)^{N} \\
		&+D\zeta_i(\Lambda_k^{x,t,\e}(\theta)) \cdot\left(\left(\Lambda_k^{x,t,\e}\right)^{(N)}(\theta)\right) \\
		&+ \sum_{j=2}^{N-1} D^{(j)} \zeta_i(\Lambda_k^{x,t,\e}(\theta)) \cdot B_{N,j} \left(\left(\Lambda_k^{x,t,\e}\right)'(\theta),\ldots, \left(\Lambda_k^{x,t,\e}\right)^{(N-j+1)}(\theta)\right).
	\end{aligned}
	\]
	
	By the same formula, we also have:
	\[
	\begin{aligned}
	 \left(\zeta_{i,k}^{x,t,\e} \left(T_k^{x,t,\e} (\theta)\right)\right)^{(N)} &= \left(\zeta_{i,k}^{x,t,\e}\right)^{(N)} \left(T_k^{x,t,\e}(\theta)\right)	\cdot \left(\left(T^{x,t,\e}_k\right)'(\theta)\right)^{N} \\
		&+ \left(\zeta_{i,k}^{x,t,\e}\right)'\left(T_k^{x,t,\e}(\theta)\right) \cdot  \left(\left(T_k^{x,t,\e}\right)^{(N)}(\theta)\right) \\
		&+ \sum_{j=2}^{N-1} \left(\zeta_{i,k}^{x,t,\e}\right)^{(j)} \left(T_k^{x,t,\e}(\theta)\right) \cdot B_{N,j}\left(\left(T_k^{x,t,\e}\right)'(\theta),\ldots,\left(T_k^{x,t,\e}\right)^{(N-j+1)}(\theta)\right).
	\end{aligned}
	\]
	Thus, it follows that 
	\begin{equation}
		\begin{aligned}
			& \left(\zeta_{i,k}^{x,t,\e} \left(T_k^{x,t,\e} (\theta)\right)\right)^{(N)}\\&= D^{(N)} \zeta_i\left(\Lambda_k^{x,t,\e}\left(T_k^{x,t,\e}(\theta)\right)\right) \cdot \left(\left(\Lambda_k^{x,t,\e}\right)'\left(T_k^{x,t,\e}(\theta)\right)\right)^{N} \cdot \left(\left(T^{x,t,\e}_k\right)'(\theta)\right)^{N} \\
			&+D\zeta_i\left(\Lambda_k^{x,t,\e}\left(T_k^{x,t,\e}(\theta)\right)\right) \cdot\left(\left(\Lambda_k^{x,t,\e}\right)^{(N)}\left(T_k^{x,t,\e}(\theta)\right)\right) \cdot \left(\left(T^{x,t,\e}_k\right)'(\theta)\right)^{N} \\
			&+ \sum_{j=2}^{N-1} D^{(j)} \zeta_i\left(\Lambda_k^{x,t,\e}\left(T_k^{x,t,\e}(\theta)\right)\right) \cdot B_{N,j} \left(\left(\Lambda_k^{x,t,\e}\right)'\left(T_k^{x,t,\e}(\theta)\right),\ldots, \left(\Lambda_k^{x,t,\e}\right)^{(N-j+1)}\left(T_k^{x,t,\e}(\theta)\right)\right)  \left(\left(T^{x,t,\e}_k\right)'(\theta)\right)^{N} \\
			&+ D\zeta_i\left(\Lambda_k^{x,t,\e}\left(T_k^{x,t,\e}(\theta)\right)\right) \cdot\left(\left(\Lambda_k^{x,t,\e}\right)'\left(T_k^{x,t,\e}(\theta)\right)\right)  \cdot  \left(\left(T_k^{x,t,\e}\right)^{(N)}(\theta)\right) \\
			&+ \sum_{j=2}^{N-1} \sum_{l=1}^j \left[ D^{(l)}\zeta_i\left(\Lambda_k^{x,t,\e}\left(T_k^{x,t,\e}(\theta)\right)\right) \cdot  B_{j,l} \left(\left(\Lambda_k^{x,t,\e}\right)'\left(T_k^{x,t,\e}(\theta)\right),\ldots, \left(\Lambda_k^{x,t,\e}\right)^{(j-l+1)}\left(T_k^{x,t,\e}(\theta)\right)\right)\right] \\  & \qquad \qquad   B_{N,j}\left(\left(T_k^{x,t,\e}\right)'(\theta),\ldots,\left(T_k^{x,t,\e}\right)^{(N-j+1)}(\theta)\right).
		\end{aligned}
	\end{equation}
	For simplicity, we will denote the summands on the right-hand side of this equation by $I$, $II$, $III$, $IV$, and $V$, respectively.
	
	By definition of $T_k^{x,t,\e} (\theta)$, it follows that
	\begin{equation} \label{eq:derivativeofTk}
	\begin{aligned}
		\frac{\partial^{N}}{\partial \theta^{N}}  \left(\frac{\partial T_{P_k(\e),A_k(\e)}}{\partial x}\right) (t+x,t,\theta,\e) =  \left(\zeta_{0,k}^{x,t,\e}\left(T_k^{x,t,\e}(\theta)\right)\right)^{(N)}.
	\end{aligned}
	\end{equation}
	Observe that, for $q \in \{1,\ldots,p+1\}$,
	\[
	\left(\Lambda_k^{x,t,\e}\right)^{(q)}(\theta) = \left(0,\delta_{1q},\frac{\partial^q P_k(\e)}{\partial \theta^q}(t+x,\theta),\frac{\partial^q A_k(\e)}{\partial \theta^q}(t+x,\theta),0\right),
	\]
	where $\delta_{ij}$ is the Kronecker delta. Thus, since $N\geq2$, it follows that
	\begin{equation*}
		\left\|D\zeta_i\left(\Lambda_k^{x,t,\e}\left(T_k^{x,t,\e}(\theta)\right)\right) \cdot\left(\left(\Lambda_k^{x,t,\e}\right)^{(N)}\left(T_k^{x,t,\e}(\theta)\right)\right)\right\| \leq L(\e) \left[\left\| \frac{\partial^{N} P_k(\e)}{\partial \theta^{N}} \right\| + \left\|\frac{\partial^{N} A_k(\e)}{\partial \theta^{N}} \right\|\right],
	\end{equation*}
	which, combined with \eqref{ineq:boundednesspartialTkpartialtheta0}, ensures that
	\begin{equation} \label{ineq:II0}
		\|II\| \leq L(\e) \left[\left\| \frac{\partial^{N} P_k(\e)}{\partial \theta^{N}} \right\| + \left\|\frac{\partial^{N} A_k(\e)}{\partial \theta^{N}} \right\|\right] e^{NL(\e)(1+2\Delta(\e))|x|}.
	\end{equation}
	Moreover, we also have
	\begin{equation*}
		\left\|D\zeta_i\left(\Lambda_k^{x,t,\e}\left(T_k^{x,t,\e}(\theta)\right)\right) \cdot\left(\left(\Lambda_k^{x,t,\e}\right)'\left(T_k^{x,t,\e}(\theta)\right)\right)\right\| \leq L(\e) \left[1 + \left\| \frac{\partial P_k(\e)}{\partial \theta} \right\| + \left\|\frac{\partial A_k(\e)}{\partial \theta} \right\|\right],
	\end{equation*}
	so that
	\begin{equation} \label{ineq:IV0}
		\|IV\| \leq L(\e) (1+2\Delta(\e)) \left(T_k^{x,t,\e}\right)^{(N)}(\theta).
	\end{equation}
	
	Observe that, since $\zeta_i$ is of class $C^{p+1}$ and periodic in its first two entries, there is $C_\zeta>0$ such that
	\[
		\left\|D^{(q)} \zeta_i\left(\Lambda_k^{x,t,\e}\left(T_k^{x,t,\e}(\theta)\right)\right)\right\| \leq C_\zeta
	\]
	for all $q \in \{1,\ldots,p+1\}$, all $k \in \mathbb{N}$, and all $(x,t,\theta,\e) \in \R \times \R \times \R \times [a,b]$. Thus, considering the hypothesis of induction, it follows that there are $\tilde{C}>0$ and $\tilde{N} \in \mathbb{N}$, where $\tilde{C}$ depends on the choice of the interval $[a,b]$ but $\tilde{N}$ does not, such that
	\begin{equation}\label{ineq:I+III+V0}
	\|I\|+\|III\|+\|V\| \leq \tilde{C} e^{\tilde{N}L(\e)(1+2\Delta(\e))|x|}.
	\end{equation}
	
	Therefore, considering \eqref{ineq:II0}, \eqref{ineq:IV0}, and \eqref{ineq:I+III+V0}, it follows by changing the order of derivatives of \eqref{eq:derivativeofTk} and integrating with respect to $x$ that
	\begin{equation*}
		\begin{aligned}
			\left| \left(T_k^{x,t,\e}\right)^{(N)}(\theta)\right| \leq& \int_0^x L(\e) (1+2\Delta(\e)) \left| \left(T_k^{\tau,t,\e}\right)^{(N)}(\theta)\right| d\tau  + \frac{\tilde{C} \; e^{\tilde{N}L(\e)(1+2\Delta(\e))|x|}}{\tilde{N}L(\e)(1+2\Delta(\e))} 
			\\ &+ \frac{1}{N(1+2\Delta(\e))} \left[\left\| \frac{\partial^{N} P_k(\e)}{\partial \theta^{N}} \right\| + \left\|\frac{\partial^{N} A_k(\e)}{\partial \theta^{N}} \right\|\right] e^{NL(\e)(1+2\Delta(\e))|x|}.
		\end{aligned}
	\end{equation*}
	Thus, by taking $N_q := \max(N+1,\tilde{N}+1)$, an application of Grönwall's inequality ensures that
	\begin{equation} \label{ineq:partialNpartialthetaNTk}
		\begin{aligned}
			\left|\left(T_k^{x,t,\e}\right)^{(N)}(\theta)\right| \leq& \left(\frac{\tilde{C} \; }{\tilde{N}L(\e)} +  \frac{1}{1+2\Delta(\e)} \left[\left\| \frac{\partial^{N} P_k(\e)}{\partial \theta^{N}} \right\| + \left\|\frac{\partial^{N} A_k(\e)}{\partial \theta^{N}} \right\|\right]\right) 
			 e^{N_q L(\e)(1+2\Delta(\e))|x|}.
		\end{aligned}
	\end{equation}

	Having proved \eqref{ineq:partialNpartialthetaNTk}, we proceed to showing that \ref{property0 2} holds for $q=N$. This will be done by induction on $k \in \mathbb{N}$. Define 
	\[
		C_{1,N} : =\frac{48K \tilde{C}}{\alpha}.	
	\] 
	Since $(P_0,A_0)=(0,0)$, property \ref{property0 2} is trivially true with this constant for $q=N$ and $k=0$. Suppose it holds for all non-negative integers up to a given $k \in \mathbb{N}$. Let us show that is must also hold for $k+1$. Observe that
	\begin{equation*}
		\begin{aligned}
			\frac{\partial^N P_{k+1}(\e)}{\partial \theta^N} (t,\theta) = \int_{-\infty}^{\infty} J_1(x) \left(\zeta_{1,k}^{x,t,\e} \left(T_k^{x,t,\e} (\theta)\right)\right)^{(N)} dx.
		\end{aligned}
	\end{equation*}
	Hence, considering \eqref{ineq:II0}, \eqref{ineq:IV0}, \eqref{ineq:I+III+V0}, and \eqref{ineq:partialNpartialthetaNTk}, it follows that
	\begin{equation*}
		\begin{aligned}
				\left \| \frac{\partial^N P_{k+1}(\e)}{\partial \theta^N}\right \| \leq  &\int_{-\infty}^\infty K e^{-\alpha|x|} \left(\frac{\tilde{C}(1+2\Delta(\e))}{\tilde{N}} + \tilde{C}\right) e^{N_q L(\e) (1+2\Delta(\e))|x|} dx \\
				&+ \int_{-\infty}^\infty K e^{-\alpha|x|} 2L(\e)  \left[\left\| \frac{\partial^{N} P_k(\e)}{\partial \theta^{N}} \right\| + \left\|\frac{\partial^{N} A_k(\e)}{\partial \theta^{N}} \right\|\right]\, e^{N_q L(\e) (1+2\Delta(\e))|x|} dx.
		\end{aligned}
	\end{equation*}
	
	If $\e_1$ is sufficiently small as to ensure that $2N_qL(\e)(1+2\Delta(\e)) \leq \alpha$ for all $\e \in (0,\e_1]$, and considering the hypothesis of induction, it follows that 
	\begin{equation}
		\begin{aligned}
			\left \| \frac{\partial^N P_{k+1}(\e)}{\partial \theta^N}\right \| \leq  \frac{4K}{\alpha} \left(\frac{\tilde{C}(1+2\Delta(\e))}{\tilde{N}} + \tilde{C}\right)  + \frac{8KL(\e)}{\alpha} C_{1,N}.
		\end{aligned}
	\end{equation}
	We proceed identically for $A_{k+1}$ and obtain
	\begin{equation}
		\begin{aligned}
			\left \| \frac{\partial^N A_{k+1}(\e)}{\partial \theta^N}\right \| \leq  \frac{4K}{\alpha} \left(\frac{\tilde{C}(1+2\Delta(\e))}{\tilde{N}} + \tilde{C}\right)  + \frac{8KL(\e)}{\alpha}  C_{1,N}.
		\end{aligned}
	\end{equation}
	If $\e_1$ is also chosen sufficiently small as to ensure that $32KL(\e)<\alpha$ and $2\Delta(\e)<1$ for all $\e \in (0,\e_1]$, then
	\[
			\left \| \frac{\partial^N P_{k+1}(\e)}{\partial \theta^N}\right \| + \left \| \frac{\partial^N A_{k+1}(\e)}{\partial \theta^N}\right \| \leq \frac{24K\tilde{C}}{\alpha} + \frac{C_{1,N}}{2} \leq C_{1,N},
	\]
	proving property \ref{property0 2}. 
	
	Observe that the validity of property \ref{property0 1} for $q=N$ follows immediately from \eqref{ineq:partialNpartialthetaNTk} and the fact that \ref{property0 2} holds for $q=N$. Therefore, by induction on $q$, it is proved that both properties hold for all $q \in \{1,\ldots,p+1\}$. The Lemma then follows by defining $$N_T:= \max_{q \in \{1,\ldots,p+1\}}N_q,$$and, for each interval $[a,b] \subset (0,\e_1]$, the positive constants $$C_{[a,b]}: = \max_{q \in \{1,\ldots,p+1\}}C_{1,q}, \qquad M_{[a,b]}:= \max_{q \in \{1,\ldots,p+1\}}C_{0,q}, \qquad $$
	and observing that $(f_\e,g_\e)$ is the limit of the sequence $(P_k(\e),A_k(\e))_{k \in \mathbb{N}}$.
	
\subsection{Proof of Proposition \ref{corollaryregularity3}}
	Let $D(\e)$, $\Delta(\e)$, $J_1$, $J_2$, and $T_{F,G}$ be given as in the proof of Lemma \ref{lemmahale}. It is clear that there are $K>0$ and $\alpha>0$ such that $\|J_i(t)\| \leq K e^{-\alpha|t|}$ for all $i \in \{1,2\}.$ Once again, we will denote $L(\e,D(\e),D(\e))$ by $L(\e)$ throughout the proof. As in Proposition \ref{corollaryregularity2}, we assume that the function $L$ appearing in hypothesis (ii) of the Lemma satisfies: $L(\e,\sigma,\mu)\geq \e$ for all $(\e,\sigma,\mu) \in (0,\e_0] \times [0,\rho_1) \times [0,\rho_2)$. 
	
	 For convenience, define
	\begin{itemize}
		\item $T^{x,t,\e}(\theta): = T_{f_\e,g_\e} (t+x,t,\theta)$;
		\item $\Lambda^{x,t,\e}(\theta)=(t+x,\theta,f_\e(t+x,\theta),g_\e(t+x,\theta),\e)$;
		\item$\zeta_i^{x,t,\e}(\theta):=\zeta_i(t+x,\theta,f_\e(t+x,\theta),g_\e(t+x,\theta),\e) = \zeta_i \circ \Lambda^{x,t,\e}(\theta)$.
	\end{itemize}

	First, let us consider $k=0$. Let us restrict the possible values of the parameter $\e$ to a compact interval $[a,b] \subset (0,\e_1]$, and let $\rho>0$ be such that $D(\e)< \rho$ for all $\e \in (0,\e_1]$. In this case, the functions $\zeta_0$, $\zeta_1$, $\zeta_2$ are Lipschitz continuous with Lipschitz constant $R$ over $\R \times \R \times \bar{B}_m(0,\rho) \times \bar{B}_n(0,\rho) \times [a,b]$. Thus, it is clear that, if $\e,\tilde{\e} \in [a,b] \subset (0,\e_1]$, then 
	\[
	\begin{aligned}
		| T^{x,t,\e}(\theta) - T^{x,t,\tilde{\e}}(\theta)| \leq& \int_{0}^x L(\e)(1+2\Delta(\e))| T^{\tau,t,\e}(\theta) - T^{\tau,t,\tilde{\e}}(\theta)| d\tau \\
		&+ \int_{0}^x L(\e) \left[\|f_\e-f_{\tilde{\e}}\| + \|g_\e-g_{\tilde{\e}}\|  \right] dx
		\\ &+ \int_{0}^x R|\e-\tilde{\e}|  dx.
	\end{aligned}
	\]
	Hence, from Grönwall's inequality, it follows that
	\begin{equation} \label{ineq:deltaTe}
		\begin{aligned}
		| T^{x,t,\e}(\theta) - T^{x,t,\tilde{\e}}(\theta)| \leq \frac{e^{L(\e)(1+2\Delta(\e))|x|} - 1}{1+2\Delta(\e)} \left[\|f_\e-f_{\tilde{\e}}\| + \|g_\e-g_{\tilde{\e}}\| \right] 
		+ \frac{R(e^{L(\e)(1+2\Delta(\e))|x|} - 1)}{L(\e) (1+2\Delta(\e))} |\e - \tilde{\e}| .
		\end{aligned}
	\end{equation}
	
	Now, since $(f_\e,g_\e)$ is a fixed point of the operator $S^\e$ given in the proof of Lemma \ref{lemmahale}, it follows by subtracting $S_1^{\tilde{\e}}(f_{\tilde{\e}},g_{\tilde{\e}})$ from $S_1^\e(f_\e,g_\e)$ that
	\[
	\begin{aligned}
		\|f_\e(t,\theta) - f_{\tilde{\e}}(t,\theta)\| \leq& \int_{-\infty}^{\infty} K e^{-\alpha |x|} L(\e)(1+2\Delta(\e)) | T^{x,t,\e}(\theta) - T^{x,t,\tilde{\e}}(\theta)| dx \\
		&+ \int_{-\infty}^{\infty} K e^{-\alpha|x|}L(\e) \left[\|f_\e - f_{\tilde{\e}}\| + \|g_\e-g_{\tilde{\e}}\| \right]  dx
		\\&+\int_{-\infty}^{\infty} K e^{-\alpha|x|}R|\e-\tilde{\e}|  dx .
	\end{aligned}
	\]
	Thus, considering \eqref{ineq:deltaTe}, if $\e_1$ is chosen sufficiently small so that $2L(\e)(1+2\Delta(\e))<\alpha$ for all $\e \in (0,\e_1]$, it follows that
	\[
	\|f_\e(t,\theta) - f_{\tilde{\e}}(t,\theta)\| \leq \frac{4K L(\e)}{\alpha} \left[\|f_\e - f_{\tilde{\e}}\| + \|g_\e-g_{\tilde{\e}}\| \right] + R |\e-\tilde{\e}|.
	\] 
	A similar argument ensures that
	\[
	\|g_\e(t,\theta) - g_{\tilde{\e}}(t,\theta)\| \leq \frac{4K L(\e)}{\alpha} \left[\|f_\e - f_{\tilde{\e}}\| + \|g_\e-g_{\tilde{\e}}\| \right] + R |\e-\tilde{\e}|.
	\] 
	Therefore, if $\e_1$ is also small enough to ensure that $16KL(\e)<\alpha$ for all $\e \in (0,\e_1]$, it follows that
	\begin{equation} \label{ineq:lipschitz00}
		\|f_\e - f_{\tilde{\e}}\| + \|g_\e-g_{\tilde{\e}}\| \leq 4R |\e -\tilde{\e}|
	\end{equation}
	if $\e,\tilde{\e} \in [a,b]$. The procedure can be repeated for any choice of interval $[a,b]$ with the exact same conditions required for the choice of $\e_1$, yielding generally different constants $R$, but ensuring local Lipschitz continuity nonetheless.
	
	Consider the following properties, where $q \in \{0,\ldots,p\}$:
	\begin{enumerate}[label=Q.\arabic*)]
		\item \label{property 1}There is $N_q \in \mathbb{N}$ and, for each $[a,b] \subset (0,\e_1]$, there is $C_{0,q}>0$ such that 
		\[
		\begin{aligned}
			\left| \left(T^{x,t,\e}\right)^{(q)}(\theta)-  \left(T^{x,t,\tilde{\e}}\right)^{(q)}(\theta) \right| \leq C_{0,q} |\e-\tilde{\e}| e^{N_q L(\e)(1+2\Delta(\e))|x|} 
		\end{aligned}
		\]
		for all $(x,t,\theta) \in \R \times \R \times \R$ and all $\e,\tilde{\e} \in [a,b]$.
		\item \label{property 2} For each $[a,b] \subset (0,\e_1]$, there is $C_{1,q}>0$ such that 
		\[
		\left\|\frac{\partial^q f_\e}{\partial \theta^q}  - \frac{\partial^q f_{\tilde{\e}}}{\partial \theta^q}\right\| + \left\|\frac{\partial^q g_\e}{\partial \theta^q}  - \frac{\partial^q g_{\tilde{\e}}}{\partial \theta^q}\right\| \leq C_{1,q} |\e-\tilde{\e}|
		\]
		for all $\e,\tilde{\e} \in [a,b]$.
	\end{enumerate}
	We will prove by induction that those properties hold for all $q \in \{0,\ldots,p\}$.
	
	Before we proceed to the proof itself, we make some considerations. Once again, the index $i$ will be used to denote any element of the set $\{0,1,2\}$, since the arguments are the same. Let the interval $[a,b] \subset (0,\e_1]$ be fixed. First, since $\zeta_i$ is of class $C^{p+1}$, periodic in its first two entries, and since $f_\e$ and $g_\e$ are bounded for $\e \in [a,b]$, it follows that there are constants $C_\zeta>0$ and $L_\zeta>0$ such that, for all $j \in \{0,1,\ldots,p\}$, the function $D^{(j)} \zeta_i$ satisfies
	\begin{equation} \label{ineq:boundednessDjzeta}
		\|D^{(j)} \zeta_i (\Lambda^{x,t,\e}(\theta))\| \leq C_\zeta
	\end{equation}
	and
	\begin{equation}\label{ineq:LipschitzDjzeta}
		\|D^{(j)} \zeta_i (\Lambda^{x,t,\e}(\theta)) - D^{(j)} \zeta_i (\Lambda^{x,t,\tilde{\e}}(\tilde{\theta}))\| \leq L_\zeta\|\Lambda^{x,t,\e}(\theta) - \Lambda^{x,t,\tilde{\e}}(\tilde{\theta})\|
	\end{equation}
	for all $j \in \{0,\ldots,p\}$, all $(x,t) \in \R \times \R$ and all $(\theta,\e),(\tilde{\theta},\tilde{\e}) \in \R \times [a,b]$. 
	
	Furthermore, considering the definition of $\Lambda^{x,t,\e}$ and \eqref{ineq:lipschitz00}, it follows that, for each $[a,b] \subset (0,\e_1]$, there is $R>0$ such that
	\begin{equation}\label{ineq:deltaLambdathetaepsilon}
		\|\Lambda^{x,t,\e}(\theta) - \Lambda^{x,t,\tilde{\e}}(\tilde{\theta})\| \leq 4R|\e - \tilde{\e}| + (1+2\Delta(\e)) |\theta - \tilde{\theta}|
	\end{equation}
	for all $(x,t) \in \R \times \R$ and all $(\theta,\e),(\tilde{\theta},\tilde{\e}) \in \R \times [a,b]$. Also, observe that for any $j \in \{1,\ldots,p+1\}$,
	\begin{equation} \label{eq:Lambdaj}
	\left(\Lambda^{x,t,\e}\right)^{(j)}(\theta) = \left(0,\delta_{1j},\frac{\partial^j f_\e}{\partial \theta^j}(t+x,\theta),\frac{\partial^j g_\e}{\partial \theta^j}(t+x,\theta),0\right),
	\end{equation}
	where $\delta_{ij}$ is the Kronecker delta. Thus, it follows that
	\begin{equation} \label{boundednesslambda'}
		\left\| \left(\Lambda^{x,t,\e}\right)'(\theta)\right\| \leq 1+ 2\Delta(\e),
	\end{equation}
	for all $(x,t,\theta,\e) \in \R \times \R \times \R \times (0,\e_1]$. Moreover, from Proposition \ref{propositiondeltatheta}, there is, for each $[a,b] \subset (0,e_1]$, a constant $C_\Lambda>0$ such that
	\begin{equation} \label{boundednesslambda(j)}
		\left\| \left(\Lambda^{x,t,\e}\right)^{(j)}(\theta)\right\| \leq C_\Lambda
	\end{equation}
	for all $j \in \{2,\ldots,p+1\}$ and all $(x,t,\theta,\e) \in \R \times \R \times \R \times [a,b]$. Also, from Corollary \ref{corollarydeltatheta}, it follows that
	\begin{equation} \label{lipschitzlambdaj}
		\left\| \left(\Lambda^{x,t,\e}\right)^{(j)}(\theta) - \left(\Lambda^{x,t,\tilde{\e}}\right)^{(j)}(\tilde{\theta})\right\| \leq \left\|\frac{\partial^j f_\e}{\partial \theta^j} - \frac{\partial^j f_{\tilde{\e}}}{\partial \theta^j}\right\| + \left\|\frac{\partial^j g_\e}{\partial \theta^j} - \frac{\partial^j g_{\tilde{\e}}}{\partial \theta^j}\right\| + C_\Lambda |\theta - \tilde{\theta}|
	\end{equation}
	for all $j \in \{1,\ldots,p\}$, all $(x,t) \in \R \times \R$ and all $(\theta,\e),(\tilde{\theta},\tilde{\e}) \in \R \times [a,b]$.
	
	Finally, observe that, from Proposition \ref{propositiondeltatheta}, there are $N_T$ and, for each $[a,b] \subset (0,\e_1]$, a constant $C_T>0$ such that
	\begin{equation}\label{boundednessTj}
		\left| \left(T^{x,t,\e}\right)^{(j)}(\theta)\right| \leq C_T e^{N_T L(\e)(1+2\Delta(\e))|x|}
	\end{equation}
	for all $j \in \{1,\ldots,p+1\}$ and all $(x,t,\theta,\e) \in \R \times \R \times \R \times [a,b]$.

	We start the discussion of the induction argument. Observe that the case $q=0$ follows directly from \eqref{ineq:deltaTe} and \eqref{ineq:lipschitz00}. Let $N \in \{1,\ldots,p\}$ and assume that properties \ref{property 1} and \ref{property 2} are valid for $0\leq q\leq N-1$. We will show that this ensures that such properties also hold for $q=N$.
	
	Proceeding exactly as in the proof of Proposition \ref{propositiondeltatheta}, we obtain the following from Faà di Bruno's formula: 
	\begin{equation}
		\begin{aligned}
			& \left(\zeta_{i}^{x,t,\e} \left(T^{x,t,\e} (\theta)\right)\right)^{(N)} \\&= D^{(N)} \zeta_i\left(\Lambda^{x,t,\e}\left(T^{x,t,\e}(\theta)\right)\right) \cdot \left(\left(\Lambda^{x,t,\e}\right)'\left(T^{x,t,\e}(\theta)\right)\right)^{N} \cdot \left(\left(T^{x,t,\e}\right)'(\theta)\right)^{N} \\
			&+D\zeta_i\left(\Lambda^{x,t,\e}\left(T^{x,t,\e}(\theta)\right)\right) \cdot\left(\left(\Lambda^{x,t,\e}\right)^{(N)}\left(T^{x,t,\e}(\theta)\right)\right) \cdot \left(\left(T^{x,t,\e}\right)'(\theta)\right)^{N} \\
			&+ \sum_{j=2}^{N-1} D^{(j)} \zeta_i\left(\Lambda^{x,t,\e}\left(T^{x,t,\e}(\theta)\right)\right) \cdot B_{N,j} \left(\left(\Lambda^{x,t,\e}\right)'\left(T^{x,t,\e}(\theta)\right),\ldots, \left(\Lambda^{x,t,\e}\right)^{(N-j+1)}\left(T^{x,t,\e}(\theta)\right)\right)  \left(\left(T^{x,t,\e}\right)'(\theta)\right)^{N} \\
			&+ D\zeta_i\left(\Lambda^{x,t,\e}\left(T^{x,t,\e}(\theta)\right)\right) \cdot\left(\left(\Lambda^{x,t,\e}\right)'\left(T^{x,t,\e}(\theta)\right)\right)  \cdot  \left(\left(T^{x,t,\e}\right)^{(N)}(\theta)\right) \\
			&+ \sum_{j=2}^{N-1} \sum_{l=1}^j \left[ D^{(l)}\zeta_i\left(\Lambda^{x,t,\e}\left(T^{x,t,\e}(\theta)\right)\right) \cdot  B_{j,l} \left(\left(\Lambda^{x,t,\e}\right)'\left(T^{x,t,\e}(\theta)\right),\ldots, \left(\Lambda^{x,t,\e}\right)^{(j-l+1)}\left(T^{x,t,\e}(\theta)\right)\right)\right] \\  & \qquad \qquad   B_{N,j}\left(\left(T^{x,t,\e}\right)'(\theta),\ldots,\left(T^{x,t,\e}\right)^{(N-j+1)}(\theta)\right).
		\end{aligned}
	\end{equation}
	For simplicity, we will denote the summands on the right-hand side of this equation by $I$, $II$, $III$, $IV$, and $V$, respectively. If $\e$ is replaced by $\tilde{\e}$, we will denote those terms by $I'$, $II'$, $III'$, $IV'$, and $V'$, respectively

	Considering the hypothesis of induction combined with \eqref{ineq:boundednessDjzeta}, \eqref{ineq:LipschitzDjzeta}, \eqref{lipschitzlambdaj}, \eqref{boundednesslambda'}, \eqref{lipschitzlambdaj} and \eqref{boundednessTj}, it follows that there is $N_I>0$ and, for each $[a,b] \subset (0,\e_1]$, $C_I>0$ such that 
	\begin{equation}\label{ineq:I'}
		\|I-I'\| \leq C_I e^{N_I L(\e)(1+2\Delta(\e))|x|} |\e-\tilde{\e}|.
	\end{equation}
	for all $(x,t,\theta) \in \R \times \R \times \R$ and all $\e, \tilde{\e} \in [a,b]$. Similarly, since $B_{N,j}$ and $B_{j,l}$ are polynomials, there are $N_{III}>0$ and $N_V>0$, and, for each $[a,b] \subset (0,\e_1]$, $C_{III}>0$ and $C_V>0$ such that 
	\begin{equation}\label{ineq:III'}
		\|III-III'\| \leq C_{III} e^{N_{III} L(\e)(1+2\Delta(\e))|x|} |\e-\tilde{\e}|
	\end{equation}
	and
	\begin{equation}\label{ineq:V'}
		\|V-V'\| \leq C_{V} e^{N_V L(\e)(1+2\Delta(\e))|x|} |\e-\tilde{\e}|.
	\end{equation}
	for all $(x,t,\theta) \in \R \times \R \times \R$ and all $\e, \tilde{\e} \in [a,b]$.
	
	Regarding $II$, observe that \eqref{eq:Lambdaj}, \eqref{lipschitzlambdaj}, and the properties of Lipschitz continuity of $\zeta_i$ given in hypothesis (iii) guarantee that
	\begin{equation*}
		\begin{aligned}
		\left\|D\zeta_i\left(\Lambda^{x,t,\e}\left(T^{x,t,\e}(\theta)\right)\right) \cdot\left(\left(\Lambda^{x,t,\e}\right)^{(N)}\left(T^{x,t,\e}(\theta)\right) - \left(\Lambda^{x,t,\tilde{\e}}\right)^{(N)}\left(T^{x,t,\tilde{\e}}(\theta)\right)\right)\right\| \leq& L(\e) \left[\left\| \frac{\partial^{N} f_\e}{\partial \theta^{N}} - \frac{\partial^{N} f_{\tilde{\e}}}{\partial \theta^{N}}\right\| + \left\|\frac{\partial^{N} g_\e}{\partial \theta^{N}}  - \frac{\partial^{N} g_{\tilde{\e}}}{\partial \theta^{N}}\right\|\right] \\
		+& L(\e) C_\Lambda \left| T^{x,t,\e}(\theta) - T^{x,t,\tilde{\e}}(\theta)\right|.
		\end{aligned}
	\end{equation*}
	Thus, the hypothesis of induction, combined with \eqref{ineq:LipschitzDjzeta}, \eqref{lipschitzlambdaj}, \eqref{boundednesslambda'}, \eqref{lipschitzlambdaj} and \eqref{boundednessTj}, ensures that there is $N_{II}>0$ and, for each $[a,b] \subset (0,\e_1]$, $C_{II}>0$ such that 
	\begin{equation}\label{ineq:II'}
		\|II-II'\| \leq C_{II} e^{N_{II} L(\e)(1+2\Delta(\e))|x|} |\e-\tilde{\e}| + L(\e) \left[\left\| \frac{\partial^{N} f_\e}{\partial \theta^{N}} - \frac{\partial^{N} f_{\tilde{\e}}}{\partial \theta^{N}}\right\| + \left\|\frac{\partial^{N} g_\e}{\partial \theta^{N}}  - \frac{\partial^{N} g_{\tilde{\e}}}{\partial \theta^{N}}\right\|\right] e^{N_{II} L(\e)(1+2\Delta(\e))|x|}.
	\end{equation}
	for all $(x,t,\theta) \in \R \times \R \times \R$ and all $\e, \tilde{\e} \in [a,b]$.
	 
	 Finally, a similar argument ensures that there is $N_{IV} \in \mathbb{N}$ and, for each $[a,b] \subset (0,\e_1]$, $C_{IV}>0$ such that 
	 \begin{equation}\label{ineq:IV'}
	 	\|IV - IV'\| \leq C_{IV} e^{N_{IV} L(\e)(1+2\Delta(\e))|x|} |\e-\tilde{\e}| + L(\e)(1+2\Delta(\e)) \left|\left(T^{x,t,\e}\right)^{(N)}(\theta) - \left(T^{x,t,\tilde{\e}}\right)^{(N)}(\theta)\right| 
	 \end{equation}
	for all $(x,t,\theta) \in \R \times \R \times \R$ and all $\e, \tilde{\e} \in [a,b]$.

	By definition of $T^{x,t,\e}$, it follows that
	\begin{equation*} 
		\begin{aligned}
			\frac{\partial^{N}}{\partial \theta^{N}}  \left(\frac{\partial T_{f_\e,g_\e}}{\partial x}\right) (t+x,t,\theta,\e) =  \left(\zeta_{0}^{x,t,\e}\left(T^{x,t,\e}(\theta)\right)\right)^{(N)}.
		\end{aligned}
	\end{equation*}
	Thus, considering inequalities \eqref{ineq:I'} to \eqref{ineq:IV'}, it follows that there is $\tilde{N}_T \in \mathbb{N}$ and, for each $[a,b] \subset (0,\e_1]$, $\tilde{C}_T>0$ such that	
	\begin{equation*}
		\begin{aligned}
			\left|\left(T^{x,t,\e}\right)^{(N)}(\theta) - \left(T^{x,t,\tilde{\e}}\right)^{(N)}(\theta)\right| &\leq \int_0^x L(\e) (1+2\Delta(\e)) \left|\left(T^{\tau,t,\e}\right)^{(N)}(\theta) - \left(T^{\tau,t,\tilde{\e}}\right)^{(N)}(\theta)\right| d\tau \\
			&+\frac{1}{1+2\Delta(\e)} \left[\left\| \frac{\partial^{N} f_\e}{\partial \theta^{N}} - \frac{\partial^{N} f_{\tilde{\e}}}{\partial \theta^{N}}\right\| + \left\|\frac{\partial^{N} g_\e}{\partial \theta^{N}}  - \frac{\partial^{N} g_{\tilde{\e}}}{\partial \theta^{N}}\right\|\right] e^{\tilde{N}_T L(\e)(1+2\Delta(\e))|x|} \\
			&+ \tilde{C}_T |\e - \tilde{\e}| e^{\tilde{N}_T L(\e)(1+2\Delta(\e))|x|}
		\end{aligned}
	\end{equation*}
	for all $(x,t,\theta) \in \R \times \R \times \R$ and all $\e, \tilde{\e} \in [a,b]$. From Grönwall's inequality, it follows that
	\begin{equation} \label{ineq:TN}
		\begin{aligned}
			\left|\left(T^{x,t,\e}\right)^{(N)}(\theta) - \left(T^{x,t,\tilde{\e}}\right)^{(N)}(\theta)\right| \leq
			\left[\tilde{C}_T|\e-\tilde{\e}| + \frac{1}{1+2\Delta(\e)} \left[\left\| \frac{\partial^{N} f_\e}{\partial \theta^{N}} - \frac{\partial^{N} f_{\tilde{\e}}}{\partial \theta^{N}}\right\| + \left\|\frac{\partial^{N} g_\e}{\partial \theta^{N}}  - \frac{\partial^{N} g_{\tilde{\e}}}{\partial \theta^{N}}\right\|\right] \right] e^{(\tilde{N}_T+1) L(\e)(1+2\Delta(\e))|x|}
		\end{aligned}
	\end{equation}
	for all $(x,t,\theta) \in \R \times \R \times \R$ and all $\e, \tilde{\e} \in [a,b]$.
	
	Let us prove that property \ref{property 2} holds for $q=N$. Observe that
	\begin{equation*}
		\begin{aligned}
			\frac{\partial^N f_\e}{\partial \theta^N} (t,\theta) = \int_{-\infty}^{\infty} J_1(x) \left(\zeta_{1}^{x,t,\e} \left(T^{x,t,\e} (\theta)\right)\right)^{(N)} dx.
		\end{aligned}
	\end{equation*}
	Thus, proceeding just as above, we obtain, for each $[a,b] \subset (0,\e_1]$, a constant $\tilde{C}>0$ such that
	\begin{equation*} 
		\begin{aligned}
			\left\|\frac{\partial^N f_\e}{\partial \theta^N} - \frac{\partial^N f_{\tilde{\e}}}{\partial \theta^N}\right\| &\leq \int_{-\infty}^\infty K e^{-\alpha|x|} L(\e) (1+2\Delta(\e)) \left|\left(T^{\tau,t,\e}\right)^{(N)}(\theta) - \left(T^{\tau,t,\tilde{\e}}\right)^{(N)}(\theta)\right| d\tau \\
			&+\int_{-\infty}^\infty K e^{-\alpha |x|} L(\e) \left[\left\| \frac{\partial^{N} f_\e}{\partial \theta^{N}} - \frac{\partial^{N} f_{\tilde{\e}}}{\partial \theta^{N}}\right\| + \left\|\frac{\partial^{N} g_\e}{\partial \theta^{N}}  - \frac{\partial^{N} g_{\tilde{\e}}}{\partial \theta^{N}}\right\|\right] e^{N_{II} L(\e)(1+2\Delta(\e))|x|} \\
			&+ \int_{-\infty}^\infty K e^{-\alpha|x|} \tilde{C} |\e - \tilde{\e}| e^{\tilde{N}_T L(\e)(1+2\Delta(\e))|x|}.
		\end{aligned}
	\end{equation*}
	Define
	\[
	C_{1,N} := \frac{16K \tilde{C}}{\alpha}, \qquad \tilde{N}:= \max \{\tilde{N}_T+1, N_{II}\}.
	\] 
	If $\e_1$ is sufficiently small as to ensure that $2\tilde{N}L(\e)(1+2\Delta(\e)) \leq \alpha$ for all $\e \in (0,\e_1]$, then it follows by integrating and considering \eqref{ineq:TN} that
	\begin{equation*} 
		\begin{aligned}
			\left\|\frac{\partial^N f_\e}{\partial \theta^N} - \frac{\partial^N f_{\tilde{\e}}}{\partial \theta^N}\right\|  \leq \frac{8KL(\e)}{\alpha} \left[\left\| \frac{\partial^{N} f_\e}{\partial \theta^{N}} - \frac{\partial^{N} f_{\tilde{\e}}}{\partial \theta^{N}}\right\| + \left\|\frac{\partial^{N} g_\e}{\partial \theta^{N}}  - \frac{\partial^{N} g_{\tilde{\e}}}{\partial \theta^{N}}\right\|\right] + \frac{4K\tilde{C}}{\alpha} |\e - \tilde{\e}|.
		\end{aligned}
	\end{equation*}
	Proceeding similarly for $g_\e$, we obtain
	\begin{equation*} 
		\begin{aligned}
			\left\|\frac{\partial^N g_\e}{\partial \theta^N} - \frac{\partial^N g_{\tilde{\e}}}{\partial \theta^N}\right\|  \leq \frac{8KL(\e)}{\alpha} \left[\left\| \frac{\partial^{N} f_\e}{\partial \theta^{N}} - \frac{\partial^{N} f_{\tilde{\e}}}{\partial \theta^{N}}\right\| + \left\|\frac{\partial^{N} g_\e}{\partial \theta^{N}}  - \frac{\partial^{N} g_{\tilde{\e}}}{\partial \theta^{N}}\right\|\right] + \frac{4K\tilde{C}}{\alpha} |\e - \tilde{\e}|.
		\end{aligned}
	\end{equation*}
	Hence, if $\e_1$ is also chosen small enough to ensure that $32L(\e)(1+2\Delta(\e))<\alpha$, it follows that
	\[
		\left\| \frac{\partial^{N} f_\e}{\partial \theta^{N}} - \frac{\partial^{N} f_{\tilde{\e}}}{\partial \theta^{N}}\right\| + \left\|\frac{\partial^{N} g_\e}{\partial \theta^{N}}  - \frac{\partial^{N} g_{\tilde{\e}}}{\partial \theta^{N}}\right\| \leq \frac{16K \tilde{C}}{\alpha} |\e - \tilde{\e}| = C_{1,N} |\e - \tilde{\e}|,
	\]
	proving that property \ref{property 2} holds for $q=N$. Thus, the validity of property \ref{property 1} for this value of $q$ follows immediately considering \eqref{ineq:TN}. 
	
	Therefore, we have proved by induction that properties \ref{property 1} and \ref{property 2} are valid for $q \in \{0,\ldots,p\}$. This concludes the proof of the Lemma, because property \ref{property 2} ensures local Lipschitz continuity of the functions considered.

\section*{Data availability}
All data generated or analysed during this study are included in this published article

\section*{Conflict of interest}
On behalf of all authors, the corresponding author states that there is no conflict of interest.
\bibliographystyle{abbrv}
\bibliography{references}

\end{document}